\documentclass[12pt]{amsart}
\title[Waldhausen Additivity: Classical and Quasicategorical]{Waldhausen Additivity: Classical and Quasicategorical}
\author{Thomas M. Fiore}
\author{Malte Pieper}

\address{Thomas M. Fiore \\ Department of Mathematics and Statistics\\
University of Michigan-Dearborn \\ 4901 Evergreen Road \\ Dearborn,
MI 48128 \\ U.S.A. \newline and \newline
NWF I - Mathematik \\
Universit\"at Regensburg \\
Universit\"atsstra{\ss}e 31 \\
93040 Regensburg \\
Germany}
\email{tmfiore@umd.umich.edu}
\urladdr{http://www-personal.umd.umich.edu/~tmfiore/}

\address{Malte Pieper \\ Mathematisches Institut der Universit\"at Bonn\\
                Bonn International Graduate School - BIGS\\
                Endenicher Allee 60\\
                53115 Bonn \\ Germany}
\email{pieper@math.uni-bonn.de}

\usepackage{mathrsfs}

\usepackage[margin=1.2in]{geometry}

\theoremstyle{plain}
\newtheorem{theorem}{Theorem}[section]
\newtheorem{lemma}[theorem]{Lemma}
\newtheorem{proposition}[theorem]{Proposition}
\newtheorem{corollary}[theorem]{Corollary}

\theoremstyle{definition}
\newtheorem{definition}[theorem]{Definition}

\newtheorem{example}[theorem]{Example}

\newtheorem{remark}[theorem]{Remark}
\newtheorem{notation}[theorem]{Notation}

{\catcode`@=11\global\let\c@equation=\c@theorem}

\usepackage{color}

\usepackage{amsmath}
\usepackage{amssymb}
\usepackage{amscd}
\usepackage{amsfonts}  
\usepackage[bookmarksnumbered,colorlinks,bookmarks,citecolor=blue,linkcolor=blue,urlcolor=blue,breaklinks,linktocpage]{hyperref}
\usepackage[all]{xy} \SelectTips{cm}{}
\usepackage{graphicx}    
\usepackage{multicol}    
\usepackage{url}

\DeclareMathAlphabet\EuR{U}{eur}{m}{n}
\SetMathAlphabet\EuR{bold}{U}{eur}{b}{n}

\makeindex             

\newcommand{\comsquare}[8]                   
{\begin{CD}
#1 @>#2>> #3\\
@V{#4}VV @V{#5}VV\\
#6 @>#7>> #8
\end{CD}
}

\newcommand{\xycomsquareminus}[8]                   
{\xymatrix
{#1 \ar[r]^-{#2} \ar[d]_{#4} &
#3 \ar[d]^{#5}  \\
#6\ar[r]^-{#7} &
#8
}
}

\newcommand{\co}{\colon}

\newcommand{\cala}{{\mathcal A}}
\newcommand{\calb}{{\mathcal B}}
\newcommand{\calc}{{\mathcal C}}
\newcommand{\cald}{{\mathcal D}}
\newcommand{\cale}{{\mathcal E}}

\newcommand{\cali}{{\mathcal I}}
\newcommand{\calj}{{\mathcal J}}

\newcommand{\calm}{{\mathcal M}}

\newcommand{\calo}{{\mathcal O}}

\newcommand{\calr}{{\mathcal R}}

\newcommand{\bfA}{{\mathbf A}}

\newcommand{\bfC}{{\mathbf C}}
\newcommand{\bfD}{{\mathbf D}}

\newcommand{\bfK}{{\mathbf K}}

\newcommand{\higherlim}[3]{{\setbox1=\hbox{\rm lim}
        \setbox2=\hbox to \wd1{\leftarrowfill} \ht2=0pt \dp2=-1pt
        \mathop{\vtop{\baselineskip=5pt\box1\box2}}
        _{#1}}^{#2}#3}

\newcommand{\version}[1]                       
{\begin{center} last edited on #1\\
last compiled on \today\\
name of tex-file: \jobname
\end{center}
}

\begin{document}

\typeout{----------------------------  KTheoryForQuasiCats.tex  ----------------------------}

\typeout{-----------------------  Abstract  ------------------------}


\begin{abstract}
We use a simplicial product version of Quillen's Theorem A to prove classical Waldhausen Additivity of $wS_\bullet$, which says that the ``subobject'' and ``quotient'' functors of cofiber sequences induce a weak equivalence $wS_\bullet \cale(\cala,\calc,\calb) \to wS_\bullet\cala \times wS_\bullet \calb$. A consequence is Additivity for the Waldhausen $K$-theory spectrum of the associated split exact sequence, namely a stable equivalence of spectra $\mathbf{K}(\cala) \vee \mathbf{K}(\calb) \to \mathbf{K}(\cale(\cala,\calc,\calb))$. This paper is dedicated to transferring these proofs to the quasicategorical setting and developing Waldhausen quasicategories and their sequences. We also give sufficient conditions for a split exact sequence to be equivalent to a standard one. These conditions are always satisfied by stable quasicategories, so Waldhausen $K$-theory sends any split exact sequence of pointed stable quasicategories to a split cofiber sequence. Presentability is not needed. In an effort to make the article self-contained, we recall all the necessary results from the theory of quasicategories, and prove a few quasicategorical results that are not in the literature.
\end{abstract}

\maketitle


\typeout{--------------------   Section 0: Introduction --------------------------}

{\bf Note: This version of this paper is obsolete. The more current, pre-publication accepted version of the article (``authors' accepted manuscript'') is available on the website \\ \url{http://www-personal.umd.umich.edu/~tmfiore/1/research.html}. \\
The pre-publication accepted version from that website will be posted on the arXiv to replace this obsolete version after the embargo period. \\

The copy-edited version of the ``authors' accepted manuscript'' will be soon published in the {\it Journal of Homotopy and Related Structures}.\\ \url{https://link.springer.com/journal/volumesAndIssues/40062}
}

\tableofcontents

\setcounter{section}{-1}
\section{Introduction and Statement of Results}

The Additivity Theorems of Quillen \cite[\S 3, Theorem~2]{Quillen} and Waldhausen~\cite[Theorem~1.4.2]{WaldhausenAlgKTheoryI} are fundamental theorems in $K$-theory; many other results follow from them, see for instance \cite{GraysonExactSequences} and \cite{Staffeldt}.
In this paper, we follow Rognes' approach to prove classical $wS_\bullet$ Additivity using a simplicial product version of Quillen's Theorem~$A$ proved by Jones--Kim--Mhoon--Santhanam--Walker--Grayson \cite{JonesAdditivity}, called Theorem~$\widehat{A}^*$. Our first innovation is to construct via pushouts a crucial simplicial homotopy (inspired by \cite{McCarthyAdditivity}) in a way that transfers to the quasicategorical setting to prove $(S_\bullet^\infty)_\text{equiv}$ Additivity for Waldhausen quasicategories. Much of this paper is dedicated to developing Waldhausen quasicategories and transferring this proof of Additivity to this setting.  We also show how Additivity on the spectrum level for standard split exact sequences follows from Additivity on the level of $wS_\bullet$ and $(S_\bullet^\infty)_\text{equiv}$. In particular, since every split exact sequence of stable quasicategories is equivalent to a standard one, we conclude that the Waldhausen $K$-theory spectrum sends every split exact sequence of pointed stable quasicategories to a split cofiber sequence, an important property in the universal characterization of $K$-theory in the formulation of \cite{BlumbergGepnerTabuadaI}.

Our notion of Waldhausen quasicategory is equivalent to Barwick's notion of Waldhausen $\infty$-category \cite{Barwick}, but the approach, the formulations, and the proofs in our paper are different from his. In the present paper we follow Waldhausen's definition of the $K$-theory spectrum using iterations of an $S_\bullet$ construction, see Definition~\ref{def:Sbulletinfinity} for the quasicategorical variant $S_\bullet^\infty$. Barwick, on the other hand, characterizes additive theories, and defines algebraic $K$-theory as an additive theory with a universal property (the linearization of the functor ``maximal Kan subcomplex''). In his work, the $K$-theory spectrum is the ``canonical delooping of the `maximal Kan subcomplex' functor.'' Barwick handles coherence with the theory of Cartesian fibrations and defines a fibrational version of $S_\bullet^\infty$, we do not use Cartesian fibrations. Our coherence results Theorem~\ref{thm:computation_of_pushouts_objectwise}, Proposition~\ref{prop:naturality_of_object-wise_pushouts}, and Corollary~\ref{cor:main_corollary_about_pushouts} are not in \cite{Barwick}.
We construct an identity-preserving, functorial choice of pushouts in a quasicategory $\calc$ along selected maps, and construct the associated natural pushout functors in diagram categories $\calc^K$. Instead of Cartesian fibrations, the main external input is Proposition~\ref{prop:RiehlVerity_Corollary_For_Pushouts}, an analogue of Riehl--Verity's \cite[Corollary~5.2.20]{RiehlVerity}. See the paragraphs Relationship to the Literature at the end of this introduction for more information and differences with the work of Blumberg-Gepner-Tabuada \cite{BlumbergGepnerTabuadaI}.

Although the details of our proofs are mostly simplicial, categorical aspects of quasicategories play a major role. Indeed, much of this paper relies upon quasicategorical versions of the basic results of category theory, as developed by Joyal~\cite{JoyalQuadern}; any other model should also have its own versions of the basic results of category theory. Thus, we expect that the broad outline of our proofs is model independent in the sense that it can be carried out in models of higher categories that come with a suitable simplicial nerve and a suitable simplicial $S_\bullet$ construction.

Waldhausen Additivity, in its most particular form, says the following. Let $\calc$ be a Waldhausen category. This is a category equipped with a distinguished zero object, a subcategory of cofibrations, and subcategory of weak equivalences satisfying various axioms, most notably: pushouts of cofibrations exist, are cofibrations, and are invariant up to weak equivalence (=gluing lemma). Let $S_2\calc$ denote the category of cofiber sequences $A \rightarrowtail C \twoheadrightarrow B$ in $\calc$, and let $s,q\colon S_2\calc \to \calc$ denote the ``subobject'' and ``quotient'' functors which return $A$ and $B$ respectively. Waldhausen  $wS_\bullet$ Additivity for $S_2\calc$ \cite[Theorem~1.4.2]{WaldhausenAlgKTheoryI} says that the map of simplicial objects in $\mathbf{Cat}$
\begin{equation} \label{equ:intro:specific_additivity}
\xymatrix{wS_\bullet (s,q) \co wS_\bullet S_2 \mathcal{C} \ar[r] & wS_\bullet \mathcal{C} \times wS_\bullet \mathcal{C}}
\end{equation}
is a weak equivalence, that is, the diagonal of its level-wise nerve is a weak homotopy equivalence of simplicial sets. See Definition~\ref{def:S_bullet} and equation~\eqref{equ:sequence_of_cofibrations} for a recollection of Waldhausen's $S_\bullet$ construction.

More generally, if $\cala$ and $\calb$ are sub Waldhausen categories of $\calc$, in place of $S_2\calc$ we may consider the category $\cale(\cala,\calc,\calb)$, which consists of cofiber sequences $A \rightarrowtail C \twoheadrightarrow B$ in $\calc$ with $A \in \cala$ and $B \in \calb$. Waldhausen $wS_\bullet$ Additivity for $\cale(\cala,\calc,\calb)$ says that
\begin{equation} \label{equ:intro:E(A,C,B)_additivity}
\xymatrix{wS_\bullet (s,q) \co wS_\bullet \cale(\cala,\calc,\calb) \ar[r] & wS_\bullet \mathcal{A} \times wS_\bullet \mathcal{B}}
\end{equation}
is a weak equivalence of simplicial objects in $\mathbf{Cat}$. Waldhausen proved in \cite[Proposition~1.3.2]{WaldhausenAlgKTheoryI} that this apparently more general statement is equivalent to the claim that \eqref{equ:intro:specific_additivity} is a weak equivalence. In Section~\ref{sec:simplified_proof} we prove directly that \eqref{equ:intro:E(A,C,B)_additivity} is a weak equivalence.

If
\begin{equation} \label{equ:intro:split_exact_sequence}
\xymatrix{\mathcal{A} \ar[r]^-i & \mathcal{E} \ar[r]^-f \ar@/^1pc/[l]^j & \mathcal{B} \ar@/^1pc/[l]^g}
\end{equation}
is a split exact sequence of Waldhausen categories (Definitions~\ref{def:exact_sequence_of_Waldhausen_categories} and \ref{def:split-exact_sequence_of_Waldhausen_categories}) that is equivalent to a standard split exact sequence associated to $\cale(\cala,\calc,\calb)$ (see Example~\ref{examp:E(A,C,B)}, Definitions~\ref{def:Waldhausen_equivalence} and \ref{def:Waldhausen_equiv_of_sequences_quasicategorical}, and Proposition~\ref{prop:Waldhausen_equiv_of_sequences_compatible_with_j's}), then the functors $j$ and $f$ induce a weak equivalence
\begin{equation} \label{equ:intro:additivity_for_split_exact_sequence}
\xymatrix{wS_\bullet(j,f)\colon wS_\bullet \mathcal{E} \ar[r] & wS_\bullet \mathcal{A} \times wS_\bullet \mathcal{B}}
\end{equation}
and the functors $i$ and $g$ induce a stable equivalence of $K$-theory spectra
\begin{equation} \label{equ:intro:additivity_for_split_exact_sequence:spectral}
\xymatrix{\bfK(i)\vee \bfK(g)\colon \bfK(\cala) \vee \bfK(\calb) \ar[r] & \bfK(\cale).}
\end{equation}
We prove this $\bfK$ Additivity in Section~\ref{sec:split_exact_sequences_of_Waldhausen_categories} as a consequence of the weak equivalence \eqref{equ:intro:E(A,C,B)_additivity}. We give sufficient conditions in Proposition~\ref{prop:universal_split_exact_sequence} for a split exact sequence in \eqref{equ:intro:split_exact_sequence} to be Waldhausen equivalent to a standard one.

The main goal of this paper is to directly prove versions of the foregoing results for {\it Waldhausen quasicategories} in Sections~\ref{sec:Additivity_for_S_bullet_infinity} and \ref{sec:split_exact_sequences_of_Waldhausen_quasicategories}. Our strategy is to first prove these classical results in such a way that the proofs carry over to the quasicategorical context, then to develop all of the necessary properties of Waldhausen quasicategories, and finally to make the proof alterations for Waldhausen quasicategories. This strategy has several benefits: readers interested solely in classical Additivity can read our proofs of the classical results without quasicategorical distraction, genuinely quasicategorical matters are clearly separated from categorical matters, and the paper can be read with a minimum of prerequisites.

{\bf Summary of Contents.} In Section~\ref{sec:simplified_proof} we recall the notion of Waldhausen category, the $S_\bullet$ construction, and the category $\cale(\cala,\calc,\calb)$ of cofiber sequences mentioned above, and
prove $wS_\bullet$ Additivity for $\cale(\cala,\calc,\calb)$. The key step, {\it upon which the entire paper depends}, is to prove in a transferable way Lemma~\ref{lem:HardLemmaMadeEasy}: $wS_\bullet$ Additivity for $\cale(\cala,\calc,\calb)$ on the level of object simplicial sets $\mathfrak{s_\bullet}:=\text{Obj}\; S_\bullet$. In the proof of Lemma~\ref{lem:HardLemmaMadeEasy}, we construct a crucial simplicial homotopy $\rho$ via a pushout in \eqref{equ:rho_n_as_pushout}. In Section~\ref{sec:simplified_proof}, we also observe that $S_n$ is a 2-functor on $\mathbf{Wald}_2$, so that it preserves split exact sequences and equivalences, alluding to the later observation that $S_n^\infty$ is an $(\infty,2)$-functor on $\mathbf{QWald}_{\infty,2}$ that also preserves split exact sequences and equivalences.

In Section~\ref{sec:split_exact_sequences_of_Waldhausen_categories} we recall split exact sequences of Waldhausen categories, the standard split exact sequence associated to $\cale(\cala,\calc,\calb)$, and show how $wS_\bullet$ Additivity for $\cale(\cala,\calc,\calb)$ implies Waldhausen $\bfK$ Additivity for standard split exact sequences. As expected, we also conclude $wS_\bullet$ Additivity and $\bfK$ Additivity for split exact sequences equivalent to standard ones, and we give sufficient conditions for a split exact sequence to be equivalent to a standard one.

In an effort to make this paper self-contained, we discuss in Section~\ref{sec:recollections} all of the notions and results we need from the theory of quasicategories.  We prove a number of results as well that are either deep in the literature or not in the literature. Topics we recall and discuss are: quasicategories, their simplicial enrichment, simplicial exponentiation, a mapping space for a quasicategory, 0-fullness and 1-fullness of subquasicategories, the homotopy category of a quasicategory, equivalences in a quasicategory, equivalences between quasicategories, adjunctions between quasicategories, constructions of Joyal equivalences and Joyal fibrations via restriction and postcomposition, and commutative squares and pushouts in a quasicategory. Section~\ref{subsec:natural_pushout_along_cofibrations} is the main quasicategorical technicality of the paper:  a construction of an identity-preserving, functorial choice of pushouts in a quasicategory $\calc$ along selected maps, and a construction of the associated natural pushout functors in diagram categories $\calc^K$. The application of Section~\ref{subsec:natural_pushout_along_cofibrations} is the transfer of the pushout argument in \eqref{equ:rho_n_as_pushout} in Lemma~\ref{lem:HardLemmaMadeEasy} to a pushout argument in the quasicategorical context in Lemma~\ref{lem:HardLemmaMadeEasy_quasicategorical}. The main goal in Section~\ref{subsec:natural_pushout_along_cofibrations} is Corollary~\ref{cor:main_corollary_about_pushouts}, and the main external input is Proposition~\ref{prop:RiehlVerity_Corollary_For_Pushouts}, an analogue of Riehl--Verity's \cite[Corollary~5.2.20]{RiehlVerity}.

In Section~\ref{sec:Additivity_for_S_bullet_infinity} we prove the quasicategorical versions of the results in Section~\ref{sec:simplified_proof}. We introduce the notion of Waldhausen quasicategory, comment on immediate consequences and differences from the classical case, and prove the equivalence with Barwick's notion of Waldhausen $\infty$-category. We recall methods of Barwick and Blumberg--Gepner--Tabuada for obtaining Waldhausen $\infty$-categories from classical Waldhausen categories, and give the examples of stable quasicategories and diagram Waldhausen quasicategories. Thus, there are a variety of examples of Waldhausen quasicategories in Examples~\ref{examp:from_classical_to_Waldhausen_quasicategory_Barwick}, \ref{examp:from_classical_to_Waldhausen_quasicategory}, \ref{examp:stable_quasicategories_are_Waldhausen_quasicategories}, and \ref{examp:diagram_Waldhausen_quasicategories}, and Theorem~\ref{thm:BlumbergGepnerTabuada_Comparisons}. We show that Waldhausen quasicategories form both an $(\infty,2)$-category $\mathbf{QWald}_{\infty,2}$ and a 2-category $\mathbf{QWald}_{2}$, and that $S_\bullet^\infty$ is an endo-$(\infty,2)$-functor and an endo-2-functor on these respective categories. The Waldhausen quasicategory of cofiber sequences has two equivalent versions: $\cale(\cala,\calc,\calb)$ with objects commutative triangles that are cofiber sequences, and $\cale^{po}(\cala,\calc,\calb)$ with objects pushout squares that realize a cofiber sequence as a quotient. The version $\cale(\cala,\calc,\calb)$ without the full pushout data is convenient for formulations, but the version $\cale^{po}(\cala,\calc,\calb)$ with full pushout data is convenient in proofs where the data is relevant for constructions. Lemma~\ref{lem:HardLemmaMadeEasy_quasicategorical} is $(S_\bullet^\infty)_\text{equiv}$ Additivity for $\cale(\cala,\calc,\calb)$ on the level of object simplicial sets $\mathfrak{s}_\bullet^\infty:=\text{Obj}\; S_\bullet^\infty$. There, the quasicategorical issues of the transferred argument of Lemma~\ref{lem:HardLemmaMadeEasy} are presented. One difference from the classical case is that for any Waldhausen quasicategory $\cald$, the inclusion $\mathfrak{s}_\bullet^\infty \cald \hookrightarrow (S_\bullet^\infty\cald)_\text{equiv}$ is always a diagonal weak equivalence of bisimplicial sets, since the weak equivalences of a Waldhausen quasicategory are the equivalences, see Proposition~\ref{prop:D_to_D(m,w)_is_Waldhausen_equivalence}. This simplifies the proof of $(S_\bullet^\infty)_\text{equiv}$ Additivity for $\cale(\cala,\calc,\calb)$ in Theorem~\ref{thm:Additivity_for_quasicategories}.

In Section~\ref{sec:K-Theory_Space_and_Spectrum} we introduce variants of Waldhausen's $K$-theory space and $K$-theory spectrum for Waldhausen quasicategories and show how to make them strictly 1-functorial. We carefully treat the matter of selecting a distinguished zero object and implications for the structure maps of the $K$-theory spectrum.

Section~\ref{sec:split_exact_sequences_of_Waldhausen_quasicategories} is the quasicategorical version of Section~\ref{sec:split_exact_sequences_of_Waldhausen_categories}. We introduce split exact sequences of Waldhausen quasicategories, the standard split exact sequence associated to $\cale(\cala,\calc,\calb)$, and show how $(S_\bullet^\infty)_\text{equiv}$ Additivity for $\cale(\cala,\calc,\calb)$ implies Waldhausen $\bfK$ Additivity for standard pointed split exact sequences of Waldhausen {\it quasi}categories. $(S_\bullet^\infty)_\text{equiv}$ Additivity and $\bfK$ Additivity of course hold for split exact sequences equivalent to standard ones (pointed sequences for $\bfK$ Additivity), and we give sufficient conditions for a split exact sequence to be equivalent to a standard one. Importantly, these conditions hold for every split exact sequence of stable quasicategories, so $\bfK$ maps split exact sequence of pointed stable quasicategories to split cofiber sequences, see Corollary~\ref{cor:stable_qcats_are_example_for_main_theorem}.

To speed up access to the main results, we have relegated some essential material (used in the main text) to three appendices. In Section~\ref{sec:appendix1} we introduce equivalences between Waldhausen categories and Waldhausen quasicategories as equivalences in the 2-categories $\mathbf{Wald}_2$ and $\mathbf{QWald}_2$ respectively, characterize these Waldhausen equivalences as those exact equivalences of (quasi)categories which reflect (weak) equivalences and cofibrations, and show that $S_\bullet$, $wS_\bullet$, $\bfK$, respectively $S_\bullet^\infty$, $(S_\bullet^\infty)_\text{equiv}$, and quasicategorical $\bfK$, send these Waldhausen equivalences to appropriate equivalences. In Section~\ref{sec:appendix_Wald_equivalences_of_ses} we turn to sequences of Waldhausen quasicategories, and
define suitable notions of morphism and equivalence. This also makes use of the 2-category $\mathbf{QWald}_2$: the squares in a morphism of sequences are only required to commute up to iso 2-cell. More precisely, a sequence is an object of the 2-category $\text{\bf 2-Cat}_{ps}([2],\mathbf{QWald}_2)$, while a morphism of sequences is a morphism  in this 2-category. The properties of ``exact'' and ``split exact'' are preserved under Waldhausen equivalence of sequences, and in the case of split exactness, a Waldhausen equivalence is automatically compatible with splittings up to iso 2-cells. In Section~\ref{sec:appendix_simplicial_homotopy}, we review the notion of simplicial homotopy and give several methods for constructing simplicial homotopies in various contexts, especially for $\mathfrak{s}_\bullet$ and $\mathfrak{s}_\bullet^\infty$ evaluated on natural transformations between exact functors, as we need for Lemmas~\ref{lem:HardLemmaMadeEasy} and \ref{lem:HardLemmaMadeEasy_quasicategorical}.

{\bf Relationship to the Literature.} Two recent articles concern Additivity for quasicategorical versions of $K$-theory: \cite{BlumbergGepnerTabuadaI} of Blumberg--Gepner--Tabuada and \cite{Barwick} of Barwick. The work of Blumberg--Gepner--Tabuada focusses on {\it stable} quasicategories (all maps are considered cofibrations for stable quasicategories), whereas in the present article we propose a quasicategorical analogue of Waldhausen category in which the class of cofibrations can be more general, see Definition~\ref{def:Waldhausen_quasicat}. Blumberg--Gepner--Tabuada use the quasicategorical $S_\bullet$ construction denoted $S_\bullet^\infty$, like we do, however, unlike the present article, they proceed from classical Waldhausen Additivity using Morita theory, spectrally enriched categories, and strictification. Our present article remains entirely in the framework of quasicategories, and does not derive quasicategorical Additivity from classical Additivity.

The article \cite{Barwick} of Barwick remains in the framework of quasicategories, like we do. Rather than beginning with a quasicategorical $S_\bullet$ construction, he characterizes additive theories, and defines algebraic $K$-theory as an additive theory with a universal property (the linearization of the functor ``maximal Kan subcomplex''). The $K$-theory spectrum is the ``canonical delooping of the `maximal Kan subcomplex' functor.'' He proves quasicategorical versions of the basic theorems of algebraic $K$-theory, and then proves that this universal definition extends the classical one. Barwick allows the class of cofibrations to be general, like we do. In fact, his notion of Waldhausen $\infty$-category is equivalent to the notion of Waldhausen quasicategory in the present paper, see Proposition~\ref{prop:Barwick_equivalence}. Barwick also uses a quasicategorical $S_\bullet$ construction, see his Section 5 for a discussion of filtered objects. He proves in Corollary 6.9.1 that suspension in the category $\mathbf{V}_\text{add}\mathbf{Wald}_\infty$ is equivalent to the realization of the Waldhausen cocartesian fibration $\mathscr{S}\calc \to N\Delta^{\text{op}}$. Barwick told us there is an equivalence between $\mathscr{S}\calc$ and the unstraightening of the simplicial quasicategory $S_\bullet^\infty \calc$, which fiberwise is just the inclusion of the top row in each $S_n^\infty \calc$. Classical Additivity (for Waldhausen categories with good factorization properties of weak equivalences) is a consequence of quasicategorical Additivity according to \cite[Corollary~10.10.3]{Barwick}, without going through $\mathscr{B}$. Barwick does not use classical Additivity to deduce \cite[Proposition~10.10]{Barwick}, of which his 10.10.3 is a Corollary. Simple maps do not have good factorization properties, so the relative nerve does not give the proper $K$-theory when simple maps are the weak equivalences, and instead Barwick uses labelled Waldhausen $\infty$-categories for that situation. See Example~\ref{examp:from_classical_to_Waldhausen_quasicategory_Barwick} for more details on how Barwick obtains Waldhausen $\infty$-categories from classical Waldhausen categories.

Both papers \cite{Barwick} and \cite{BlumbergGepnerTabuadaI} are concerned with universal characterizations of algebraic $K$-theory via Additivity, whereas the present paper is only concerned with a proof of Additivity in quasicategories. Quasicategorical variants of Waldhausen's Approximation Theorem are treated in \cite{Barwick}, \cite{BlumbergGepnerTabuadaI}, and \cite{FioreApproximation}.

\section{Proof of $wS_\bullet$ Additivity in the Classical Case} \label{sec:simplified_proof}

We recall the notion of Waldhausen category and the $S_\bullet$ construction from \cite{WaldhausenAlgKTheoryI}, and combine elements of \cite{JonesAdditivity}, \cite{McCarthyAdditivity}, and \cite{RognesTextbook} with a pushout simplicial homotopy to prove $wS_\bullet$ Additivity for $\cale(\cala, \calc, \calb)$ via Theorem~$\widehat{A}^*$.

A {\it Waldhausen category} is a category $\mathcal{C}$ equipped with a distinguished zero object $\ast$, a subcategory $co \mathcal{C}$ of {\it cofibrations}, and a subcategory $w\mathcal{C}$ of {\it weak equivalences} that satisfy Waldhausen's axioms \cite[pages 320 and 326]{WaldhausenAlgKTheoryI}. These axioms are: all isomorphisms are both cofibrations and weak equivalences, for each object $A$ the unique map $\ast \to A$ is a cofibration, the pushout of each cofibration exists and is again a cofibration, and finally, the gluing lemma holds. Cofibrations are always indicated with a feathered arrow $\rightarrowtail$.

\begin{example}[Classical Waldhausen Categories, See \cite{WaldhausenAlgKTheoryI} and Section 8.2 of \cite{RognesTextbook}] \label{examp:classical_Waldhausen_cats} \leavevmode
\begin{enumerate}
\item \label{examp:classical_Waldhausen_cats:fin_based_sets}
The category of based finite sets and based maps is a Waldhausen category with cofibrations the injections and weak equivalences the bijections.
\item \label{examp:classical_Waldhausen_cats:fin_gen_Rmods}
For $R$ a ring, the category of finitely generated $R$-modules and $R$-module maps is a Waldhausen category with cofibrations the monomorphisms and weak equivalences the isomorphisms.
\item \label{examp:classical_Waldhausen_cats:fin_gen_proj_Rmods}
For $R$ a ring, the category of finitely generated {\it projective} $R$-modules and $R$-module maps is a Waldhausen category with cofibrations the monomorphisms $i\co P_1 \to P_2$ such that $P_2/P_1$ is finitely generated projective, and weak equivalences the isomorphisms.
\item \label{examp:classical_Waldhausen_cats:fin_gen_Rmods_chain_complexes}
For $R$ a ring, the category of bounded chain complexes of finitely generated $R$-modules and chain maps is a Waldhausen category with cofibrations the monomorphisms and weak equivalences the homology isomorphisms.
\item \label{examp:classical_Waldhausen_cats:fin_gen_proj_Rmods_chain_complexes}
For $R$ a ring, the category of bounded chain complexes of finitely generated {\it projective} $R$-modules and chain maps is a Waldhausen category with cofibrations the monomorphisms $i\co P_1 \to P_2$ such that $P_2/P_1$ is a bounded chain complex of finitely generated projective $R$-modules, and weak equivalences the homology isomorphisms.
\item
The category with objects based simplicial sets with finitely many nondegenerate simplices and morphisms based maps is a Waldhausen category. The cofibrations are the monomorphisms and the weak equivalences are the weak homotopy equivalences, that is, the maps with geometric realization a homotopy equivalence.
\item
The category with objects based simplicial sets with finitely many nondegenerate simplices and morphisms based maps is a Waldhausen category in another way as well. The cofibrations are the monomorphisms (as above), but the weak equivalences are the simple maps. A map of simplicial sets $f\co A \to B$ is {\it simple} if the preimage of each point of $\vert B \vert$ under $\vert f \vert \co \vert A \vert \to \vert B \vert$ is contractible.
\item
Every model category contains an example of a Waldhausen category. If $\calm$ is a model category, then its full subcategory $\calm_c$ on its cofibrant objects is a Waldhausen category. For a proof of the gluing lemma in this generality, see Goerss--Jardine \cite[pages 122-123 and 127]{GoerssJardine} or Hovey \cite[Lemma 5.2.6]{Hovey1999}.
\end{enumerate}
\end{example}

Let $\mathcal{B}$ and $\mathcal{C}$ be Waldhausen categories. A functor $\mathcal{B} \to \mathcal{C}$ is {\it exact} if it takes $\ast_\calb$ to $\ast_\calc$, maps cofibrations to cofibrations, maps weak equivalences to weak equivalences, and maps each pushout along a cofibration to a pushout along a cofibration.

A {\it sub Waldhausen category $\calb$ of a Waldhausen category $\mathcal{C}$} is a subcategory $\calb$ of $\calc$ that is a Waldhausen category in its own right, such that: (i) the distinguished zero object of $\calb$ is equal to the distinguished zero object of $\calc$, (ii) a morphism in the subcategory $\calb$ is a cofibration in the subcategory $\calb$ if and only if it is a cofibration in the larger category $\mathcal{C}$ {\it and} the quotient in $\calc$ is isomorphic to an object of $\calb$, (iii) every pushout square in the subcategory $\calb$ with one leg a $\calb$-cofibration is also a pushout square in the larger category $\calc$, and (iv) a morphism in the subcategory $\calb$ is a weak equivalence in $\calb$ if and only if it is a weak equivalence in $\calc$. All of this implies that the inclusion functor $\calb \to \calc$ is exact, though exactness of the inclusion is not equivalent to the statement that $\calb$ is a sub Waldhausen category (the inclusion must also reflect weak equivalences, and the inclusion must also reflect cofibrations for which the quotient is in the subcategory). The Waldhausen category of finitely generated projective $R$-modules in \ref{examp:classical_Waldhausen_cats:fin_gen_proj_Rmods} above is a sub Waldhausen category of finitely generated $R$-modules in \ref{examp:classical_Waldhausen_cats:fin_gen_Rmods} above. Similarly, Example \ref{examp:classical_Waldhausen_cats:fin_gen_proj_Rmods_chain_complexes} is a sub Waldhausen category of Example \ref{examp:classical_Waldhausen_cats:fin_gen_Rmods_chain_complexes}.


We next want to recall the $S_\bullet$ construction, a main ingredient in the Additivity Theorem. Let $\text{Ar}[n]$ be the {\it category of arrows} in $[n]$. It is the partially ordered set with elements $(i,j)$ such that $0 \leq i \leq j \leq n$, and with the order $(i,j) \leq (i',j')$ whenever $i \leq i'$ and $j \leq j'$. In other words, $\text{Ar}[n]$ is the subcategory of the square grid $[n] \times [n]$ that is full on the objects on the main diagonal and above.

\begin{definition}[Waldhausen's $S_\bullet$ Construction, Section 1.3 of \cite{WaldhausenAlgKTheoryI}] \label{def:S_bullet}
Let $\mathcal{C}$ be a Waldhausen category. An {\it object of the category $S_n\mathcal{C}$} is a functor $A \co \text{Ar}[n] \to \mathcal{C}$ such that
\begin{enumerate}
\item
For each $j \in [n]$, $A(j,j)$ is the distinguished zero object of $\mathcal{C}$,
\item
For each $i \leq j \leq k$, the morphism $A_{i,j} \to A_{i,k}$ is a cofibration,
\item
For each $i \leq j \leq k$, the diagram
$$\xymatrix{A(i,j) \ar@{>->}[r] \ar[d] & A(i,k) \ar[d] \\ A(j,j) \ar@{>->}[r] & A(j,k)}$$
is a pushout square in $\mathcal{C}$.
\end{enumerate}
The {\it morphisms of the category $S_n\mathcal{C}$} are natural transformations of diagrams of shape $\text{Ar}[n]$ in $\calc$.
\end{definition}

\begin{remark} \label{rem:Sdot_is simplicial_object}
The assignment $[n] \mapsto S_n\mathcal{C}$ is a functor $\Delta^{\text{op}} \to \mathbf{Cat}$, so that $S_\bullet\mathcal{C}$ is a simplicial object in $\mathbf{Cat}$. More precisely, if $A \co \text{Ar}[n] \to \mathcal{C}$ and $\alpha\co [k] \to [n]$, then $\alpha^* A$ is $A \circ (\alpha \times \alpha)$.
\end{remark}

The objects of $S_n\mathcal{C}$ are sequences of cofibrations in $\calc$
\begin{equation} \label{equ:sequence_of_cofibrations}
\xymatrix{\ast \ar@{>->}[r] & A_{0,1} \ar@{>->}[r] & A_{0,2} \ar@{>->}[r] & \cdots  \ar@{>->}[r] & A_{0,n}}
\end{equation}
with a choice of quotient $A_{i,j}= A_{0,j}/A_{0,i}$ for each $i \leq j$. For $0<i<n$, the face map $d_i\co S_n\mathcal{C} \to S_{n-1}\mathcal{C}$ composes the two morphisms $A_{0,i-1} \to A_{0,i} \to A_{0,i+1}$ in \eqref{equ:sequence_of_cofibrations}, which corresponds to removing object column $i$ and object row $i$ in the grid $A$. The face map $d_0\co S_n\mathcal{C} \to S_{n-1}\mathcal{C}$ removes object column $0$ and object row $0$ in the grid $A$, which means to quotient \eqref{equ:sequence_of_cofibrations} by $A_{0,1}$ and to remove the first $\ast$ to make
$$\xymatrix{\ast \ar@{>->}[r] & A_{1,2} \ar@{>->}[r] & A_{1,3} \ar@{>->}[r] & \cdots  \ar@{>->}[r] & A_{1,n}.}$$
The face map $d_n \co S_n\mathcal{C} \to S_{n-1}\mathcal{C}$ removes object column $n$ and object row $n$ in the grid $A$.
For $i \geq 0$, the degeneracy map $s_i\co S_n\mathcal{C} \to S_{n+1}\mathcal{C}$ replaces $A_{0,i}$ in \eqref{equ:sequence_of_cofibrations} by the identity morphism $A_{0,i} \to A_{0,i}$. This corresponds to inserting a column of trivial morphisms and a row of trivial morphisms in the grid $A$ in object column $i$ and object row $i$.

The category $S_n \mathcal{C}$ is a Waldhausen category. A natural transformation $f\co A \to B$ in $S_n\mathcal{C}$ is a {\it cofibration in $S_n\mathcal{C}$} if each induced morphism
\begin{equation} \label{equ:cofibration_in_S_n}
A_{0,j} \cup_{A_{0,j-1}} B_{0,j-1} \to B_{0,j}
\end{equation}
is a cofibration in $\mathcal{C}$ for each $1 \leq j \leq n$. Notice that for $j=1$, the map \eqref{equ:cofibration_in_S_n} is a cofibration if and only if $A_{0,1} \to B_{0,1}$ is a cofibration because $A_{0,0} \to B_{0,0}$ is the identity on the zero object. This definition of cofibration in $S_n\mathcal{C}$ requiring \eqref{equ:cofibration_in_S_n} to be a cofibration for each $1 \leq j \leq n$ implies, but is not equivalent to, the requirement that each $f_{i,j}\co A_{i,j} \to B_{i,j}$ be a cofibration in $\mathcal{C}$, see Rognes' notes \cite[Definition~8.3.10 and Lemma~8.3.12]{RognesTextbook}. It also implies that $A_{i,k} \cup_{A_{i,j}} B_{i,j} \to B_{i,k}$ is a cofibration for all $i \leq j \leq k$ in $[n]$. A natural transformation $f\co A \to B$ is a {\it weak equivalence in $S_n\mathcal{C}$} if each $f_{0,j}\co A_{0,j} \to B_{0,j}$ is a weak equivalence in $\mathcal{C}$. This is equivalent to requiring each $f_{i,j}$ to be a weak equivalence in $\mathcal{C}$, see \cite[8.3.15]{RognesTextbook}. Each structure map of $S_\bullet \mathcal{C}$ is exact, so $S_\bullet \mathcal{C}$ is a simplicial object in Waldhausen categories. We denote the simplicial set of the object sets of $S_\bullet \mathcal{C}$ by $\mathfrak{s}_\bullet \mathcal{C}$, that is, $\mathfrak{s}_n \mathcal{C}=\text{Obj}\; S_n \mathcal{C}$.

The objects of $S_2\mathcal{C}$ are pushouts of the form
\begin{equation} \label{equ:cofiber_sequence_pushout}
\begin{array}{c}
\xymatrix{A \ar@{>->}[r] \ar[d] & C \ar@{->>}[d] \\ \ast \ar[r] & B,}
\end{array}
\end{equation}
where the arrow $\twoheadrightarrow$ always indicates the projection to the quotient. A sequence of morphisms in $A \rightarrowtail C \twoheadrightarrow B$ in $\calc$ is called a {\it cofiber sequence} when the square \eqref{equ:cofiber_sequence_pushout} is a pushout. Since $\ast$ is a zero object, we see that the category of cofiber sequences is isomorphic to $S_2\mathcal{C}$, and we make the identification without comment.\footnote{In the context of Waldhausen quasicategories, we do not have such an isomorphism between $S_2^\infty$ and the quasicategory of cofiber sequences, rather merely an equivalence, see Notations~\ref{not:E(A,C,B)_quasicategorical} and \ref{not:E'(A,C,B)_quasicategorical}.}

The ``subobject'' and ``quotient'' functors $s,q\co S_2\mathcal{C} \to \mathcal{C}$, which assign to the cofiber sequence $A \rightarrowtail C \twoheadrightarrow B$ the objects $A$ and $B$ respectively, are exact.

\begin{remark}[The 2-Category $\mathbf{Wald}_2$ and the 2-Functor $S_n$] \label{rem:Sn_is_a_2-functor}
Waldhausen categories, exact functors, and natural transformations between them form a 2-category denoted $\mathbf{Wald}_2$. The hom category $\mathbf{Wald}_2\big((\calc,co\calc,w\calc),(\cald,co\cald,w\cald)\big)$ is the full subcategory of $\mathbf{Cat}(\calc, \cald)$ on the functors $\calc \to \cald$ that are exact. Consequently, $\mathbf{Wald}_2$ inherits its 2-category structure from  $\mathbf{Cat}$, the 2-category of small categories, functors, and natural transformations. Moreover, $S_n\co \mathbf{Wald}_2 \to \mathbf{Wald}_2$ is a 2-functor because $\mathbf{Cat}(\text{Ar}[n],-)\co\mathbf{Cat} \to \mathbf{Cat}$ is a 2-functor. In particular, for $A\co \text{Ar}[n] \to \calc$ in $S_n\calc$, exact functors $F,G\co\calc \to \cald$, and a natural transformation $\alpha\co F \Rightarrow G$ we have
$$S_nF:=F \circ A \hspace{1in} S_nG:=G \circ A  \hspace{1in} S_n\alpha:=\alpha \ast A.$$
Consequently, $S_n$ maps a natural isomorphism to a natural isomorphism, a Waldhausen equivalence to a Waldhausen equivalence (see Definition~\ref{def:Waldhausen_equivalence}), and an adjunction of Waldhausen categories to an adjunction of Waldhausen categories.\footnote{By {\it adjunction of Waldhausen categories} we simply mean an adjunction of the underlying categories.}
\end{remark}

The first version of Waldhausen Additivity says that after restricting to weak equivalences in the $S_\bullet$ construction, the Waldhausen category of cofiber sequences in $\calc$ and the product Waldhausen category $\calc \times \calc$
produce weakly equivalent simplicial objects in $\mathbf{Cat}$. Notice that both $wS_\bullet$ and its object simplicial set $\mathfrak{s}_\bullet$ preserve products,\footnote{Recall that the 2-category $\mathbf{Cat}$ is enriched Cartesian closed, so that $(-)^{\text{Ar}[n]}$ is a right adjoint and preserves limits.} so we freely write $wS_\bullet (s,q)$ to mean $(wS_\bullet s, wS_\bullet q)$, and we write $\mathfrak{s}_\bullet(s,q)$ to mean $(\mathfrak{s}_\bullet s,\mathfrak{s}_\bullet q)$.

\begin{theorem}[Waldhausen $wS_\bullet$ Additivity for $S_2\calc$, Theorem 1.4.2 of \cite{WaldhausenAlgKTheoryI}] \label{thm:Waldhausen_Additivity}
Let $\mathcal{C}$ be a Waldhausen category. Then the ``subobject'' and ``quotient'' functors induce a weak equivalence of simplicial objects in $\mathbf{Cat}$
$$\xymatrix{wS_\bullet (s,q) \co wS_\bullet S_2 \mathcal{C} \ar[r] & wS_\bullet \mathcal{C} \times wS_\bullet \mathcal{C}}.$$
That is, the diagonal of its level-wise nerve is a weak homotopy equivalence of simplicial sets.
\end{theorem}

Waldhausen deduces the Additivity Theorem from a weak equivalence on the level of object simplicial sets: his Lemma 1.4.3 says that $\mathfrak{s}_\bullet(s,q)\colon \mathfrak{s}_\bullet S_2 \calc \to \mathfrak{s}_\bullet \calc \times \mathfrak{s}_\bullet \calc$ is a weak homotopy equivalence of simplicial sets. His Lemma 1.4.3 proof concerning object simplicial sets included an application of Quillen's Theorem B and a certain technical Sublemma. In \cite{McCarthyAdditivity}, McCarthy gave a new proof of Waldhausen's Lemma 1.4.3 for object simplicial sets in the spirit of Quillen's Theorem~A without using Waldhausen's Sublemma, but left verification of the simplicial homotopy identities to the reader. Later in \cite{JonesAdditivity}, Jones--Kim--Mhoon--Santhanam--Walker--Grayson proved Theorem~$\widehat{A}^*$  (recalled below in Theorem~\ref{thm:A_hat_star}), which generally allows one to convert Theorem~B style proofs into Theorem~A style proofs. However, their proof of the weak equivalence on the level of object simplicial sets in \cite[Lemma 2.4]{JonesAdditivity} uses results of Waldhausen's technical Sublemma. Rognes, on the other hand, in \cite[Lemma~8.5.11]{RognesTextbook}, applies Theorem~$\widehat{A}^*$ to $(\mathfrak{s}_\bullet s, \mathfrak{s}_\bullet q)$ and avoids Waldhausen's technical Sublemma. In a totally different approach, Grayson uses edgewise subdivision to prove Additivity for cofibration sequences of exact functors\footnote{Additivity for cofibration sequences of exact functors is equivalent to Additivity in the forms of Theorem~\ref{thm:Waldhausen_Additivity} and Theorem~\ref{thm:Additivity_(equivalent_formulation)} by \cite[Proposition~1.3.2]{WaldhausenAlgKTheoryI}.} in \cite{GraysonAdditivity}.

In the present paper, like Rognes, we use Theorem~$\widehat{A}^*$ of \cite{JonesAdditivity} {\it without} Waldhausen's Sublemma, to prove $\mathfrak{s}_\bullet$ Additivity in Lemma~\ref{lem:HardLemmaMadeEasy}, but for $\cale(\cala, \calc, \calb)$ rather than $S^2\calc$: the subobject and the quotient in the cofiber sequences are in sub Waldhausen categories $\cala$ and $\calb$ rather than $\calc$. The novelty of the present paper is our construction of a simplicial homotopy via pushouts that makes the simplicial homotopy identities apparent, {\it even in the quasicategorical context}, so that Theorem~$\widehat{A}^*$ can also be applied in the quasicategorical context. Our simplicial homotopy is inspired by the simplicial homotopy $\Gamma\simeq \text{Id}_{S_\bullet F\vert \mathcal{C}^2(-,[n])}$ in \cite{McCarthyAdditivity}.

The Additivity Theorem~\ref{thm:Additivity_(equivalent_formulation)} for $wS_\bullet$ for cofiber sequences with subobject in $\cala$ and quotient in $\calb$ follows from the object version in Lemma~\ref{lem:HardLemmaMadeEasy}. In Corollary~\ref{cor:Additivity_GeneralAndSpectral}~\ref{cor:Additivity_GeneralAndSpectral:(i)} we prove Waldhausen Additivity for split exact sequences of Waldhausen categories that are Waldhausen equivalent to standard ones, and we do the spectral version in Corollary~\ref{cor:Additivity_GeneralAndSpectral}~\ref{cor:Additivity_GeneralAndSpectral:(ii)}.

We follow Waldhausen's notations for the category of cofiber sequences in $\calc$ with the subobject and the quotient in sub Waldhausen categories $\cala$ and $\calb$.

\begin{notation}[$\mathcal{E}(\mathcal{A},\mathcal{C},\mathcal{B})$] \label{not:E(A,C,B)}
Let $\mathcal{C}$ be a Waldhausen category and $\mathcal{A}$ and $\mathcal{B}$ sub Waldhausen categories. We denote by $\mathcal{E}(\mathcal{A},\mathcal{C},\mathcal{B})$ the category of cofiber sequences in $\mathcal{C}$ of the form
$$\xymatrix{A \ar@{>->}[r] & C \ar@{->>}[r] & B}$$
with $A \in \mathcal{A}$ and $B \in \mathcal{B}$. A morphism consists of three morphisms in $\cala$, $\calc$, and $\calb$ respectively, making the two relevant squares commute. The category $\cale(\calc, \calc, \calc)$ is isomorphic to $S_2\calc$.
\end{notation}

\begin{proposition}[$\mathcal{E}(\mathcal{A},\mathcal{C},\mathcal{B})$ is a Waldhausen Category in \cite{WaldhausenAlgKTheoryI}] \label{E(A,C,B)isWaldhausenCategory}
Let $\mathcal{C}$ be a Waldhausen category and $\mathcal{A}$ and $\mathcal{B}$ sub Waldhausen categories. Then $\mathcal{E}(\mathcal{A},\mathcal{C},\mathcal{B})$ is a Waldhausen category with levelwise weak equivalences, and cofibrations the diagram morphisms from $A_1 \rightarrowtail C_1 \twoheadrightarrow B_1$ to $A_2 \rightarrowtail C_2 \twoheadrightarrow B_2$ such that $A_1 \to A_2$ is a cofibration in $\cala$ and the induced map
\begin{equation} \label{equ:cofibration_in_E(A,C,B)}
C_1 \cup_{A_1} A_2 \to C_2
\end{equation}
is a cofibration in $\calc$, see \cite[Definition~8.3.10 and Lemma 8.3.12]{RognesTextbook}. The projections to $\mathcal{A}$, $\mathcal{C}$, and $\mathcal{B}$ are exact. In particular, the three components of a cofibration are also cofibrations.
\end{proposition}

Before proving Lemma~\ref{lem:HardLemmaMadeEasy}, we recall Theorem~$\widehat{A}^*$ and left fibers in simplicial sets. Theorem~$\widehat{A}^*$ is a simplicial version of Quillen's Theorem~A adapted for the situation of products.

\begin{definition}[Left Fiber in Simplicial Sets, pages 336-337 of \cite{WaldhausenAlgKTheoryI}] \label{def:left_fiber}
For any simplicial set map $f\co X \to Y$ and any $y \in Y_m$, the {\it left fiber} $f/(m,y)$ is defined as the following pullback in simplicial sets.
$$\xymatrix{f/(m,y) \ar[r]^-{\pi} \ar[d] \ar@{}[dr]|{\text{pullback}} & X \ar[d]^f \\ \Delta[m] \ar[r]_y & Y}$$
By the universal property of the pullback and several applications of the Yoneda Lemma, an $n$-simplex $\Delta[n] \to f/(m,y)$ is the same as an $n$-simplex $x$ of $X$ together with a morphism $\alpha\co [n] \to [m]$ in $\Delta$ such that $\alpha^* y = f(x)$.
\end{definition}

\begin{theorem}[Simplicial Version of Quillen's Theorem $A$, Lemma~1.4.A of \cite{WaldhausenAlgKTheoryI}]
If $f/(m,y)$ is contractible for every $(m,y)$, then $f$ is a weak homotopy equivalence.
\end{theorem}

\begin{theorem}[Theorem~$\widehat{A}^*$ of Jones--Kim--Mhoon--Santhanam--Walker--Grayson\footnote{In the statement of their theorem in \cite{JonesAdditivity}, Jones--Kim--Mhoon--Santhanam--Walker--Grayson use the term ``homotopy equivalence'' rather than ``weak equivalence''. In the current paper we prefer to use the term ``weak homotopy equivalence'' for maps of simplicial sets which induce a homotopy equivalence after geometric realization, in order to distinguish them from simplicial homotopy equivalences.} in \cite{JonesAdditivity}] \label{thm:A_hat_star}
Let $(f,g)\co X \to Y \times T$ be a map of simplicial sets. If the composite
$$\xymatrix{f/(m,y) \ar[r]^-{\pi} & X \ar[r]^g & T}$$ is a weak homotopy equivalence for all $m \geq 0$ and all $y \in Y_m$, then $(f,g)$ is a weak homotopy equivalence of simplicial sets.
\end{theorem}

The following has Waldhausen's \cite[Lemma 1.4.3]{WaldhausenAlgKTheoryI} as the special case $\cala=\calb=\calc$.

\begin{lemma}[Waldhausen $wS_\bullet$ Additivity for $\cale(\cala,\calc,\calb)$ on Object Simplicial Sets] \label{lem:HardLemmaMadeEasy}
Let $\mathcal{C}$ be a Waldhausen category and $\cala$ and $\calb$ sub Waldhausen categories. Then the ``subobject'' and ``quotient'' functors induce a weak homotopy equivalence of simplicial sets
$$\xymatrix{\mathfrak{s}_\bullet (s,q) \co \mathfrak{s}_\bullet \cale(\cala,\calc,\calb) \ar[r] & \mathfrak{s}_\bullet \cala \times \mathfrak{s}_\bullet \calb}.$$
\end{lemma}
\begin{proof}
We begin this proof as in the notes of Rognes \cite[Lemma~8.5.11]{RognesTextbook}.
Let $(f,g)$ be $(\mathfrak{s}_\bullet s, \mathfrak{s}_\bullet q) \co \mathfrak{s}_\bullet \cale(\cala,\calc,\calb) \to \mathfrak{s}_\bullet \cala \times \mathfrak{s}_\bullet \calb$ and fix $A \in \mathfrak{s}_m \cala$. We would like to prove that the composite
\begin{equation} \label{equ:r}
\xymatrix{f/(m,A) \ar[r]^-{\pi} \ar@/_1pc/[rr]_r & \mathfrak{s}_\bullet \cale(\cala,\calc,\calb) \ar[r]^-g & \mathfrak{s}_\bullet \calb},
\end{equation}
which we denote by $r$, is a weak homotopy equivalence of simplicial sets.
An $n$-simplex $(e,\alpha)$ in $f/(m,A)$ is an element $e$ in $\mathfrak{s}_n\cale(\cala, \calc,\calb)$ together with a morphism $\alpha \co [n] \to [m]$ in $\Delta$ such that $\alpha^\ast A$ is the top grid of $e$. More specifically, $e$ is a diagram $e\co [2] \times \text{Ar}[n] \to \mathcal{C}$ of the form
\begin{equation} \label{equ:element_in_left_fiber}
e=\left( \begin{array}{c}\xymatrix{\alpha^*A \ar@{>->}[d] \\ C \ar@{->>}[d] \\B } \end{array} \right) =
\begin{array}{c}
\xymatrix{\ast \ar@{>->}[r] \ar@{>->}[d] & A(\alpha0,\alpha1) \ar@{>->}[r] \ar@{>->}[d] &  A(\alpha0,\alpha2) \ar@{>->}[r] \ar@{>->}[d] &  \cdots \ar@{>->}[r]  & A(\alpha0,\alpha n) \ar@{>->}[d] \\
\ast \ar@{>->}[r] \ar@{->>}[d] & C_1 \ar@{>->}[r] \ar@{->>}[d] & C_2 \ar@{>->}[r] \ar@{->>}[d] & \cdots \ar@{>->}[r] & C_n \ar@{->>}[d] \\
\ast \ar@{>->}[r] & B_1 \ar@{>->}[r] & B_2 \ar@{>->}[r] & \cdots \ar@{>->}[r] & B_n}
\end{array},
\end{equation}
with $C \in \mathfrak{s}_n\calc$ and $B \in \mathfrak{s}_n B$. In the right diagram for $e$, we did not indicate the chosen quotients, and maps between them, for readability. The $i$-th face of the $n$-simplex $(e,\alpha)$ in $f/(m,A)$ is
$$d_i(e,\alpha)=\Big(e \circ (\text{Id}_{[2]} \times \delta^i_n \times \delta^i_n),\, \alpha \circ \delta^i_n\Big)$$
where $\delta^i_n\co [n-1] \to [n]$ is the weakly increasing map that skips $i$. Notice the top row of $d_i(e,\alpha)$ is $$(\alpha \circ \delta^i_n)^*A=A \circ (\alpha \times \alpha)\circ (\delta^i_n \times \delta^i_n)=d_i\Big(\alpha^*A\Big).$$
The degeneracies of $(e,\alpha)$ are defined similarly.

We proceed with some explanation of the diagram \eqref{equ:element_in_left_fiber}. To recognize the general shape of the diagram $e\in \mathfrak{s}_n\cale(\cala, \calc,\calb)$, we are using the isomorphisms of categories $(\calc^{[2]})^{\text{Ar}[n]} \cong \calc^{[2] \times \text{Ar}[n]} \cong (\calc^{\text{Ar}[n]})^{[2]}$. The rows of diagram \eqref{equ:element_in_left_fiber} are sequences of cofibrations and the columns are cofiber sequences. To see that all 3 rows of maps consist entirely of cofibrations in $\cala$, $\calc$, and $\calb$ respectively, recall that an element of $\mathfrak{s}_n\cale(\cala, \calc,\calb)$ is a sequence of $n$ cofibrations in $\cale(\cala, \calc,\calb)$ with chosen quotients, and that the projections to $\cala$, $\calc$, and $\calb$ are exact by Proposition~\ref{E(A,C,B)isWaldhausenCategory}. By a similar argument, all the maps between chosen quotients parallel to these horizontal morphisms in \eqref{equ:element_in_left_fiber} form sequences of cofibrations in $\cala$, $\calc$, and $\calb$ respectively. For later, we also need the observation that the vertical maps from level 0 to level 1 together form a cofibration $\alpha^\ast A \to C$ in $S_n \calc$: in diagram \eqref{equ:element_in_left_fiber} we have a sequence of cofibrations in $\cale(\cala, \calc,\calb)$ if and only if each induced map $C_i \cup_{A(\alpha0,\alpha i)} A(\alpha0,\alpha (i+1)) \to C_{i+1}$ is a cofibration in $\calc$ as in equation \eqref{equ:cofibration_in_E(A,C,B)}, but since these are exactly the induced maps of equation \eqref{equ:cofibration_in_S_n}, this is the case if and only if the maps from level 0 to level 1 together form a cofibration $\alpha^\ast A \to C$ in $S_n \calc$. In fact, we have an isomorphism of Waldhausen categories $S_n\cale(\cala,\calc,\calb)\cong\cale(S_n\cala,S_n\calc,S_n\calb)$, though we do not need this full isomorphism.

We consider every diagram such as that in \eqref{equ:element_in_left_fiber} as a functor $[2] \times [n] \times [n] \to \calc$ which has levels 0, 1, and 2 of functors $[n] \times [n] \to \calc$ that are $\ast$ on the main diagonal and below. Every morphism in the grid between the default $\ast$'s is of course an identity.

The composite $r$ maps the $n$-simplex $e$ above to $B\in \mathfrak{s}_n\calb$, the bottom level of diagram \eqref{equ:element_in_left_fiber}. We prove that $r$ is a simplicial deformation retraction and therefore a weak homotopy equivalence of simplicial sets. That is, we display a map $\iota\co \mathfrak{s}_\bullet \calb \to f/(m,A)$ such that $r \circ \iota = \text{Id}_{\mathfrak{s}_\bullet \calb}$ and we show $\text{Id}_{f/(m,A)}$ is simplicially homotopic to a map $(\rho_n^0)_n$, which is simplicially homotopic to $\iota \circ r$. Then we are finished by Theorem~$\widehat{A}^*$, as $A$ was an arbitrary simplex in $\mathfrak{s}_\bullet \cala$.

For $B \in \mathfrak{s}_n\calb$, we define $\iota(B)$ to be the diagram
\begin{equation} \label{equ:iota(B)}
\iota(B)=\left( \begin{array}{c} \xymatrix{\ast \ar@{>->}[d] \\ B \ar@{=}[d] \\ B } \end{array}, \begin{array}{c} \text{const}_m \end{array} \right)=
\left(
\begin{array}{c}
\xymatrix{\ast \ar@{>->}[r] \ar@{>->}[d] & \ast \ar@{>->}[r] \ar@{>->}[d] & \ast \ar@{>->}[r] \ar@{>->}[d] & \cdots \ar@{>->}[r] & \ast \ar@{>->}[d] \\
\ast \ar@{>->}[r] \ar@{=}[d] & B_1 \ar@{>->}[r] \ar@{=}[d] & B_2 \ar@{>->}[r] \ar@{=}[d] & \cdots \ar@{>->}[r] & B_n \ar@{=}[d] \\
\ast \ar@{>->}[r] & B_1 \ar@{>->}[r] & B_2 \ar@{>->}[r] & \cdots \ar@{>->}[r] & B_n}
\end{array}, \begin{array}{c} \text{const}_m \end{array}
\right)
\end{equation}
where $\text{const}_m \co [n] \to [m]$ maps each $j \in [n]$ to $m$. We do not include the chosen quotients in the notation for readability. The top row has chosen quotients all $\ast$. We have $\text{const}_m^*A$ is equal to the top grid of $\iota(B)$ in \eqref{equ:iota(B)} since $A(m,m)=\ast$.

What we have said in this proof so far is in \cite[Lemma~8.5.11]{RognesTextbook} for the special case $\cala=\calb=\calc$. Our method to prove that $\text{Id}_{f/(m,A)}$ and $\iota \circ r$ are simplicially homotopic is inspired by a construction of McCarthy \cite{McCarthyAdditivity}. Our construction via pushouts, which is new, will make the simplicial identities apparent (even for chosen quotients). We use a modern reformulation of the notion of simplicial homotopy, as in \cite[Lemma~6.4.8]{RognesTextbook}. Recall that in this modern formulation of a simplicial homotopy $h\colon f \simeq g$, we have $(h_n^{n+1})_n=f$ and $(h_n^0)_n=g$. In other words, the last maps of the homotopy form the first simplicial map, and the zero-th maps of the homotopy form the second simplicial map. See Proposition~\ref{prop:simplicial_homotopy_description} for further details.

Throughout this proof, we use a selected functorial choice of pushouts in $\calc$ along cofibrations, and this choice preserves identities.\footnote{\label{footnote:functorial_choice_of_pushout}Let $\calc$ be a Waldhausen category. Let $\calj$ be the free standing commutative square $[1] \times [1]$, and let $\cali$ be the subcategory of $\calj$ consisting of the top horizontal map, the left vertical map, and the 3 necessary identities, i.e. $\cali$ has the shape $\bullet \leftarrow \bullet \rightarrow \bullet$. Then the restriction functor $\mathbf{Cat}(\calj,\calc) \to \mathbf{Cat}(\cali,\calc)$ restricts to an equivalence of categories from the full subcategory of $\mathbf{Cat}(\calj,\calc)$ spanned by the pushouts with left vertical map a cofibration to the full subcategory of $\mathbf{Cat}(\cali,\calc)$ spanned by the diagrams with left vertical map a cofibration. An inverse equivalence to this equivalence is a functorial choice of pushouts in $\calc$ along cofibrations. To modify any such functorial choice into one which preserves identities, we can redefine it on pushouts with either leg an identity to be the respective square with vertical identities as in \eqref{equ:identity_pushout}, and conjugate diagram morphism images by the relevant isomorphisms to also redefine the pushout functor on morphisms.} ``Preserves identities'' means the selected pushout functor performs the assignments in \eqref{equ:identity_pushout} when either leg is an identity. This functorial choice of pushout that preserves identities will crucially be used object-wise in the formation of the pushout $\rho^j_n(e,\alpha)$ in \eqref{equ:rho_n_as_pushout} to guarantee compatibility with face and degeneracy maps.

To show $\text{Id}_{f/(m,A)}$ and $\iota \circ r$ are homotopic, we construct a simplicial homotopy $\rho$ from $\text{Id}_{f/(m,A)}$ to a simplicial map\footnote{This simplicial map $(\rho_n^{0})_n$ is in some sense ``isomorphic'' to $\iota \circ r$.} $(\rho_n^{0})_n$, whose value on \eqref{equ:element_in_left_fiber} is pictured in \eqref{equ:rho_n+1_grid}, and then construct another simplicial homotopy $\theta$ from $(\rho_n^{0})_n$ to $\iota \circ r$. Concerning the first simplicial homotopy $\rho$, for each $n \geq 0$ and $0\leq j \leq n+1$ we will define three functions
\begin{equation} \label{equ:lambda_mu_nu_source_target}
\xymatrix{\lambda_n^j,\; \mu_n^j, \; \nu_n^j\; \co \left(f/(m,A)\right)_n \ar[r] & \left(f/(m,A)\right)_n}
\end{equation}
that form homotopies between different maps $f/(m,A)\to f/(m,A)$. Then we will build the pushout $\rho_n^j$ of these three functions $\lambda_n^j,\; \mu_n^j, \; \nu_n^j$ using a functorial choice of pushouts along cofibrations in $\calc$ that preserves identities. These $\rho_n^j$'s together form a homotopy between the two pushed out maps $f/(m,A)\to f/(m,A)$, the first of which is $\text{Id}_{f/(m,A)}$, and the second of which is $(\rho_n^{0})_n$, pictured in \eqref{equ:rho_n+1_grid}.

{\bf Definition of Simplicial Homotopies $\lambda$, $\mu$, and $\rho$.} One ingredient we will need for $\mu$ is a simplicial homotopy $NZ\co \Delta[m] \times \Delta[1] \to \Delta[m]$. Consider the functors $[m] \to [m]$ given by identity and ``constant $m$,'' denoted $\text{Id}_{[m]}$ and $\text{const}_m$, and let the functor $Z\colon [m] \times [1] \to [m]$ be the natural transformation $\text{Id}_{[m]} \Rightarrow \text{const}_m $ which in component $j$ is simply $j \to m$. The nerve $NZ\colon \Delta[m] \times \Delta[1] \to \Delta[m]$ is then a simplicial homotopy, see Example~\ref{examp:simplicial_homotopy_from_natural_transformation}. Its functions $(NZ)^j_n \colon \Delta([n],[m])\to\Delta([n],[m])$ as in Proposition~\ref{prop:simplicial_homotopy_description} and Example~\ref{examp:simplicial_homotopy_from_natural_transformation} are
$$(NZ)^j_n \begin{pmatrix}
0 & 1 & \cdots & (j-1) & j & \cdots & n \\ a_0 & a_1 & \cdots & a_{j-1} & a_{j} & \cdots & a_n  \end{pmatrix}=\begin{pmatrix}
0 & 1 & \cdots & (j-1) & j & \cdots & n \\ a_0 & a_1 & \cdots & a_{j-1} & m & \cdots & m  \end{pmatrix}.$$

We can now define $\lambda_n^j,\; \mu_n^j, \; \nu_n^j$ in \eqref{equ:lambda_mu_nu_source_target}.

$$\lambda_n^j(e,\alpha)=\left(\begin{array}{c} \xymatrix{\alpha^\ast A \ar@{=}[d] \\ \alpha^\ast A \ar@{->>}[d] \\ \ast } \end{array}, \begin{array}{c} \alpha \end{array} \right) \hspace{.5in} \mu_n^j(e,\alpha)=\left(\begin{array}{c} \xymatrix{\left((NZ)^j_n\alpha\right)^\ast A \ar@{=}[d] \\ \left((NZ)^j_n\alpha\right)^\ast A \ar@{->>}[d] \\ \ast } \end{array}, \begin{array}{c} (NZ)^j_n\alpha \end{array} \right)   $$
$$\nu_n^j(e,\alpha)=\left(\begin{array}{c}  \xymatrix{\alpha^\ast A \ar@{>->}[d] \\ C \ar@{->>}[d] \\ B } \end{array}, \begin{array}{c} \alpha \end{array} \right)\hspace{.5in} \phantom{\mu_n^j(e,\alpha)=\left(\begin{array}{c} \xymatrix{\left((NZ)^j_n\alpha\right)^\ast A \ar@{=}[d] \\ \left((NZ)^j_n\alpha\right)^\ast A \ar@{->>}[d] \\ \ast } \end{array}, \begin{array}{c} (NZ)^j_n\alpha \end{array} \right)}$$
To be explicit, we also write out $\lambda_n^j(e,\alpha)$ and $\mu_n^j(e,\alpha)$.
\begin{equation} \label{equ:lambda_rows}
\aligned
&\lambda_n^j(e,\alpha)=\\
&\begin{array}{c}
\xymatrix@C=1pc{\ast \ar@{>->}[r] \ar@{>->}[d] & A(\alpha0,\alpha1) \ar@{>->}[r] \ar@{>->}[d] &  \cdots \ar@{>->}[r] &  A(\alpha0,\alpha(j-1)) \ar@{>->}[r] \ar@{>->}[d] & A(\alpha0,\alpha j) \ar@{>->}[r] \ar@{>->}[d] & A(\alpha0,\alpha(j+1)) \ar@{>->}[r] \ar@{>->}[d] & \cdots \ar@{>->}[r] & A(\alpha0,\alpha n) \ar@{>->}[d] \\
\ast \ar@{>->}[r] \ar@{->>}[d] & A(\alpha0,\alpha1) \ar@{>->}[r] \ar@{->>}[d] &  \cdots \ar@{>->}[r] &  A(\alpha0,\alpha(j-1)) \ar@{>->}[r] \ar@{->>}[d] & A(\alpha0,\alpha j) \ar@{>->}[r] \ar@{->>}[d] & A(\alpha0,\alpha(j+1)) \ar@{>->}[r] \ar@{->>}[d] & \cdots \ar@{>->}[r] & A(\alpha0,\alpha n) \ar@{->>}[d] \\
\ast \ar@{>->}[r] & \ast \ar@{>->}[r] & \cdots \ar@{>->}[r] &  \ast \ar@{>->}[r] & \ast \ar@{>->}[r] & \ast \ar@{>->}[r] & \cdots \ar@{>->}[r] & \ast}
\end{array}
\endaligned
\end{equation}
\begin{equation} \label{equ:mu_rows}
\aligned
& \mu_n^j(e,\alpha)= \\
& \begin{array}{c}
\xymatrix@C=1pc{\ast \ar@{>->}[r] \ar@{>->}[d] & A(\alpha0,\alpha1) \ar@{>->}[r] \ar@{>->}[d] &  \cdots \ar@{>->}[r] &  A(\alpha0,\alpha(j-1)) \ar@{>->}[r] \ar@{>->}[d] & A(\alpha0,m) \ar@{=}[r] \ar@{>->}[d] & A(\alpha0,m) \ar@{=}[r] \ar@{>->}[d] & \cdots \ar@{=}[r] & A(\alpha0,m) \ar@{>->}[d] \\
\ast \ar@{>->}[r] \ar@{->>}[d] & A(\alpha0,\alpha1) \ar@{>->}[r] \ar@{->>}[d] &  \cdots \ar@{>->}[r] &  A(\alpha0,\alpha(j-1)) \ar@{>->}[r] \ar@{->>}[d] & A(\alpha0,m) \ar@{=}[r] \ar@{->>}[d] & A(\alpha0,m) \ar@{=}[r] \ar@{->>}[d] & \cdots \ar@{=}[r] & A(\alpha0,m) \ar@{->>}[d] \\
\ast \ar@{>->}[r] & \ast \ar@{>->}[r] & \cdots \ar@{>->}[r] &  \ast \ar@{>->}[r] & \ast \ar@{>->}[r] & \ast \ar@{>->}[r] & \cdots \ar@{>->}[r] & \ast}
\end{array}
\endaligned
\end{equation}
The chosen quotients in levels 0 and 1 of $\mu^j_n(e,\alpha)$ are the respective ones of $A$. In particular, levels 0 and 1 in row $j$ and below are entirely $\ast = A(\alpha0, m)/A(\alpha0, m)$. For $j=0$, we have $\mu^0_j(e,\alpha)$ is entirely $\ast$'s, that is, the zero sequence with $n+1$ entries in levels 0, 1, and 2, all with chosen quotients $\ast/\ast=\ast$. Notice that the function $\mu_n^j$ for $1 \leq j \leq n+1$ is similar to $\lambda_n^j$, but assigns to $(e,\alpha)$ the sequence $\alpha^* A$ in levels 0 and 1 up to and including column $j-1$, then we have $A(\alpha 0, m)$.

We do not draw $\nu_n^j(e,\alpha)$, as it is just $(e,\alpha)$ and we already drew $e$ in \eqref{equ:element_in_left_fiber}.

{\bf Justification of Simplicial Homotopy Identities for $\lambda$, $\mu$, and $\rho$.} Next we justify why $\lambda$, $\mu$, and $\nu$ satisfy the simplicial homotopy identities of Proposition~\ref{prop:simplicial_homotopy_description}. Notice that $\lambda^j_n(e,\alpha)$ does not depend on $j$, so that $\lambda$ defines the ``identity'' simplicial homotopy $f/(m,A)\times \Delta[1] \to f/(m,A)$ on the simplicial map which assigns to each $(e,\alpha)$ two copies of $\alpha^* A$ followed by a grid of $\ast$'s. See the discussion after Proposition~\ref{prop:simplicial_homotopy_description} and in Example~\ref{examp:simplicial_homotopy_from_natural_transformation} for an explanation of why an ``identity'' homotopy on a simplicial map $f$ simply has all $h^j_n$ for fixed $n$ equal to $f_n$.

The simplicial homotopy identities for $\mu$ are inherited from those of $NZ$. Namely, define for the moment $h^j_n(\alpha^*A):= \left((NZ)^j_n(\alpha)\right)^*A$. We conclude the face simplicial homotopy identities for $0 \leq i \leq j-1$ for $h$ from those of $NZ$ as follows.
$$\aligned d_i (NZ)^j_n(\alpha) &= (NZ)^{j-1}_{n-1}(d^i \alpha) \\ (NZ)^j_n(\alpha) \circ \delta^i_n &= (NZ)^{j-1}_{n-1}(\alpha \circ \delta^i_n) \\
\Big((NZ)^j_n(\alpha) \circ \delta^i_n\Big)^* A &= \Big((NZ)^{j-1}_{n-1}(\alpha \circ \delta^i_n)\Big)^* A \\
d_i \circ h^j_n(\alpha^*A) &=  h^{j-1}_{n-1} \circ d_i (\alpha^*A).
\endaligned$$
The face simplicial homotopy identities for $j \leq i \leq n$, and the degeneracy simplicial homotopy identities, for $h$ can be similarly verified. Thus, $\mu$ defines a simplicial homotopy $f/(m,A)\times \Delta[1] \to f/(m,A)$ from the map which assigns to each $(e,\alpha)$ two copies of $\alpha^*A$ followed by a grid of $\ast$'s, to the map which is entirely $\ast$ in levels 0, 1, and 2.

The function $\nu_n^j$ assigns to $(e,\alpha)$ simply $(e,\alpha)$, for all $0 \leq j \leq n+1$. Thus, $\nu$ defines the identity simplicial homotopy $f/(m,A)\times \Delta[1] \to f/(m,A)$ on the identity $\text{Id}_{f/(m,A)}$.


{\bf Definition of Simplicial Homotopy $\rho$.}
We now form the diagram part of $\rho_n^j(e,\alpha)$ for $0 \leq j \leq n+1$ as the pushout
\begin{equation}  \label{equ:rho_n_as_pushout}
\begin{array}{c}
\xymatrix{ \lambda_n^j(e,\alpha) \ar[r] \ar@{}[dr]|{\mathrm{p.o.}} \ar@{>->}[d] & \mu_n^j(e,\alpha) \ar@{>->}[d] \\
\nu_n^j(e,\alpha) \ar[r] & \rho_n^j(e,\alpha)}
\end{array}
\end{equation}
of diagrams $[2] \times \text{Ar}[n]\to \calc$. More specifically we would like to form this pushout object-wise in the Waldhausen category $S_n\cale(\cala,\calc,\calb)$ using our selected functorial choice of pushouts along cofibrations in $\calc$ that preserves identities. The definition of the left vertical map in the pushout diagram is clear, the top horizontal map in levels 0 and 1 is defined as the composite natural transformation
\begin{equation} \label{equ:top_horizontal_map_i}
\xymatrix{\text{Ar}[n] \times [1] \ar[r] & \text{Ar}[m] \ar[r]^-A & \cala},
\end{equation}
induced from the natural transformation $\alpha \Rightarrow (NZ)_n^j(\alpha)$, namely the top map of \eqref{equ:rho_n_as_pushout} in levels 0 and 1 is
\begin{equation} \label{equ:top_horizontal_map_ii}
(p,q) \mapsto A \Big( (\alpha p, \alpha q) \to \left((NZ)_n^j(\alpha) p, (NZ)_n^j(\alpha) q \right) \Big).
\end{equation}
The left vertical map in \eqref{equ:rho_n_as_pushout} is a cofibration in $S_n\cale(\cala,\calc,\calb)$: to check that a map is a cofibration in $S_n\cale(\cala,\calc,\calb)$, one checks that the induced map in \eqref{equ:cofibration_in_S_n} is a cofibration in $\cale(\cala, \calc, \calb)$, but that reduces to verifying that $A(\alpha0,\alpha j)\cup_{A(\alpha 0,\alpha(j-1))} C_{j-1}\to C_j$ is a cofibration, which it is because $\alpha^\ast A \to C$ is a cofibration in $S_n\calc$ as remarked near the beginning of this proof. Since the left vertical map in \eqref{equ:rho_n_as_pushout} is a cofibration in $S_n\cale(\cala,\calc,\calb)$, it is also an object-wise cofibration between diagrams $[2] \times \text{Ar}[n]\to \calc$. Thus, the pushout in \eqref{equ:rho_n_as_pushout} exists and can be formed object-wise in $\calc^{[2] \times \text{Ar}[n]}$.

We can be sure that the constructed pushout object $\rho^j_n(e,\alpha)$ in $\calc^{[2] \times \text{Ar}[n]}$ actually has level 0 in $\cala$ and level 2 in $\calb$ because our functorial choice of pushouts along cofibrations in $\calc$ preserves identities and in levels 0 and 2 we are pushing out objectwise along identities in $\cala$ and $\calb$, namely the left vertical map in \eqref{equ:rho_n_as_pushout} is object-wise the identity in level 0, and the top horizontal map in \eqref{equ:rho_n_as_pushout} is object-wise the identity in level 2.

We can also be sure that the constructed pushout object $\rho^j_n(e,\alpha)$ in $\calc^{[2] \times \text{Ar}[n]}$ is in $S_n\cale(\calc,\calc,\calc)$, as the left vertical map in \eqref{equ:rho_n_as_pushout} is a cofibration in the Waldhausen category $S_n\cale(\calc,\calc,\calc)$.

Finally, we see that our constructed $\rho^j_n(e,\alpha)$ is in $\mathfrak{s}_n\cale(\cala,\calc,\calb)$. Notice that it would not be sufficient to skip the foregoing and simply say that a pushout in \eqref{equ:rho_n_as_pushout} exists in the Waldhausen category $S_n\cale(\cala,\calc,\calb)$, because we need the pushout for each $n$ to be compatible with face and degeneracy maps, as we will soon see.

We can even explictly write down $\rho^j_n(e,\alpha)$ (not indicating the chosen quotients) and its $\Delta$ morphism.
For $1 \leq j \leq n+1$, we have from \eqref{equ:rho_n_as_pushout}
\begin{equation} \label{equ:rho_grid}
\aligned
& \rho_n^j(e,\alpha)= \\
& \begin{array}{c}
\xymatrix@C=.5pc{\ast \ar@{>->}[r] \ar@{>->}[d] & A(\alpha0,\alpha1) \ar@{>->}[r] \ar@{>->}[d] &  \cdots \ar@{>->}[r] &  A(\alpha0,\alpha(j-1)) \ar@{>->}[r] \ar@{>->}[d] & A(\alpha0,m) \ar@{=}[r] \ar@{>->}[d] & A(\alpha0,m) \ar@{=}[r] \ar@{>->}[d] & \cdots \ar@{=}[r] & A(\alpha0,m) \ar@{>->}[d] \\
\ast \ar@{>->}[r] \ar@{->>}[d] & C_1 \ar@{>->}[r] \ar@{->>}[d] & \cdots \ar@{>->}[r] &  C_{j-1} \ar@{>->}[r] \ar@{->>}[d] & \underset{A(\alpha0,\alpha j)}{C_j\bigcup A(\alpha0,m)} \ar@{>->}[r] \ar@{->>}[d] & \underset{A(\alpha0,\alpha (j+1))}{C_{j+1}\bigcup A(\alpha0,m)} \ar@{>->}[r] \ar@{->>}[d] & \cdots \ar@{>->}[r] & \underset{A(\alpha0,\alpha n)}{C_n\bigcup A(\alpha0,m)} \ar@{->>}[d] \\
\ast \ar@{>->}[r] & B_1 \ar@{>->}[r] & \cdots \ar@{>->}[r] &  B_{j-1} \ar@{>->}[r] & B_j \ar@{>->}[r] & B_{j+1} \ar@{>->}[r] & \cdots \ar@{>->}[r] & B_n}
\end{array}
\endaligned
\end{equation}
for the part $[2] \times \{0\} \times [n] \to \calc$. The chosen quotients for $\rho_n^j(e,\alpha)$ in \eqref{equ:rho_grid} are the functorial pushouts of the respective chosen quotients, as prescribed by \eqref{equ:rho_n_as_pushout}. We do not need to know the chosen quotients explicitly. From \eqref{equ:rho_grid} we see $\rho^{n+1}_n(e,\alpha)$ is $(e,\alpha)$. The $\Delta$ morphism of $\rho^j_n(e,\alpha)$ is $(NZ)^j_n(\alpha)$ because: it is defined via the pushout in \eqref{equ:rho_n_as_pushout}, while the $\Delta$ morphisms of $\lambda^j_n(e,\alpha)$, $\mu^j_n(e,\alpha)$, and $\nu^j_n(e,\alpha)$ are respectively $\alpha$, $(NZ)^j_n(\alpha)$, and $\alpha$, and the map $\lambda^j_n(e,\alpha) \to \nu^j_n(e,\alpha)$ is identity on the $\Delta$ morphisms, so on the $\Delta$ morphisms we are pushing out along identity in \eqref{equ:rho_n_as_pushout}.

For $j=0$, we have from \eqref{equ:rho_n_as_pushout} the following picture.
\begin{equation} \label{equ:rho_n+1_grid}
\begin{array}{c}
\begin{array}{c}
\rho_n^{0}(e,\alpha)=
\end{array}
\begin{array}{c}
\xymatrix{\ast \ar@{>->}[r] \ar@{>->}[d] & \ast \ar@{>->}[r] \ar@{>->}[d] & \ast\ar@{>->}[r] \ar@{>->}[d] & \cdots \ar@{>->}[r] & \ast \ar@{>->}[d] \\
\ast \ar@{>->}[r] \ar@{->>}[d] & \Bigg(\underset{A(\alpha0,\alpha1)}{C_1\;\bigcup\;\ast}\Bigg) \ar@{>->}[r] \ar@{->>}[d] & \left(\underset{A(\alpha0,\alpha2)}{C_2 \;\bigcup\; \ast}\right) \ar@{>->}[r] \ar@{->>}[d] & \cdots \ar@{>->}[r] & \left(\underset{A(\alpha0,\alpha n)}{C_n \;\bigcup\; \ast}\right) \ar@{->>}[d] \\
\ast \ar@{>->}[r] & B_1 \ar@{>->}[r] & B_2 \ar@{>->}[r] & \cdots \ar@{>->}[r] & B_n}
\end{array}
\end{array}
\end{equation}
The $\Delta$ morphism of $\rho_n^{0}(e,\alpha)$ is $\text{const}_m$.

{\bf Justification of Simplicial Homotopy Identities for $\rho$.} Now it is easy to justify why $\rho$ satisfies the simplicial homotopy identities. Recall pushouts of diagrams are formed object-wise using the functorial choice of pushouts along cofibrations, so the pushout construction in \eqref{equ:rho_n_as_pushout} is compatible with face and degeneracy maps.\footnote{The compatibility with face and degeneracy maps will need extra justification in the quasicategorical version later.} For instance, $d_i \rho_n^j(e,\alpha)$ is the pushout of $d_i \lambda_n^j(e,\alpha)$, $d_i \mu_n^j(e,\alpha)$, and $d_i \nu_n^j(e,\alpha)$. Thus, the simplicial homotopy identities for $\rho$ now follow immediately from the simplicial homotopy identities for $\lambda$, $\mu$, and $\nu$, which we have already justified.\footnote{Notice that this is a significant simplification, because the direct verification of all simplicial homotopy identities for \eqref{equ:rho_grid}, one by one, for all chosen quotients, is very tedious.}

Thus, $\rho$ is a simplicial homotopy $f/(m,A)\times \Delta[1] \to f/(m,A)$ from the map $\text{Id}_{f/(m,A)}$ to $(\rho^{0}_n)_n$ (an evaluated component of $(\rho^{0}_n)_n$ is pictured in \eqref{equ:rho_n+1_grid}). Our next step is a simplicial homotopy $\theta$ from $(\rho^{0}_n)_n$ to $\iota \circ r$.

{\bf Simplicial Homotopy $\theta\co(\rho^{0}_n)_n \simeq \iota \circ r$.} Since each column of $\rho^j_n(e,\alpha)$ is a cofiber sequence, we can in particular say that all the vertical maps of \eqref{equ:rho_n+1_grid} between level 1 and level 2 are isomorphisms. These isomorphisms provide an isomorphism from diagram \eqref{equ:rho_n+1_grid} to the diagram $\iota r(e,\alpha)$ in \eqref{equ:iota(B)}. Composing like in \eqref{equ:psi^2_4} provides us with a simplicial homotopy $\theta$ from $(\rho^{0}_n)_n$ to $\iota \circ r$ with functions
$$
\aligned
& \theta_n^j(e,\alpha)= \\
& \begin{array}{c}
\xymatrix{\ast \ar@{>->}[r] \ar@{>->}[d] & \ast \ar@{>->}[r] \ar@{>->}[d] &  \cdots \ar@{>->}[r] &  \ast \ar@{>->}[r] \ar@{>->}[d] & \ast \ar@{>->}[r] \ar@{>->}[d] & \ast \ar@{>->}[r] \ar@{>->}[d] & \cdots \ar@{>->}[r] & \ast \ar@{>->}[d] \\
\ast \ar@{>->}[r] \ar@{->>}[d] & \left(C_1\cup_{A_1}\ast\right) \ar@{>->}[r] \ar@{->>}[d]^\cong & \cdots \ar@{>->}[r] &  \left(C_{j-1}\cup_{A_{j-1}} \ast \right) \ar@{>->}[r] \ar@{->>}[d]^\cong & B_j \ar@{>->}[r] \ar@{=}[d] & B_{j+1} \ar@{>->}[r] \ar@{=}[d] & \cdots \ar@{>->}[r] & B_n \ar@{=}[d] \\
\ast \ar@{>->}[r] & B_1 \ar@{>->}[r] & \cdots \ar@{>->}[r] &  B_{j-1} \ar@{>->}[r] & B_j \ar@{>->}[r] & B_{j+1} \ar@{>->}[r] & \cdots \ar@{>->}[r] & B_n}
\end{array}
\endaligned
$$
with $\Delta$ morphism $\text{const}_m$. See Appendix 2, in particular Example~\ref{examp:nat_transf_induces_s_bullet_homotopy}, for justification of why $\theta$ is a simplicial homotopy.

{\bf Conclusion.} Finally, we have $\text{Id}_{f/(m,A)}$ is simplicially homotopic to $(\rho^{0}_n)_n$ via $\rho$, and $(\rho^{0}_n)_n$ is simplicially homotopic to $\iota \circ r$ via $\theta$. Thus $\vert \text{Id}_{f/(m,A)} \vert$ is homotopic to $\vert \iota \circ r \vert$, while $r \circ \iota = \text{Id}_{\mathfrak{s}_\bullet \calb}$, so that $r$ is a weak homotopy equivalence of simplicial sets. Theorem~$\widehat{A}^*$ now implies $\mathfrak{s}_\bullet (s,q) \co \mathfrak{s}_\bullet \cale(\cala,\calc,\calb) \to \mathfrak{s}_\bullet \cala \times \mathfrak{s}_\bullet \calb$ is a weak homotopy equivalence of simplicial sets.
\end{proof}

Waldhausen proved in \cite[Proposition~1.3.2~(1)]{WaldhausenAlgKTheoryI} that the following Additivity Theorem is equivalent to the special case $\cala=\calb=\calc$ in Theorem~\ref{thm:Waldhausen_Additivity} above. On the other hand, we conclude the following Additivity Theorem directly from Lemma~\ref{lem:HardLemmaMadeEasy}, which was already formulated for $\cale(\cala,\calc,\calb)$.

\begin{theorem}[Waldhausen $wS_\bullet$ Additivity for $\cale(\cala,\calc,\calb)$] \label{thm:Additivity_(equivalent_formulation)}
Let $\mathcal{C}$ be a Waldhausen category and $\mathcal{A}$ and $\mathcal{B}$ sub Waldhausen categories. Then the map of simplicial objects in $\mathbf{Cat}$
$$\xymatrix{wS_\bullet(s,q)\colon wS_\bullet \mathcal{E}(\mathcal{A},\mathcal{C},\mathcal{B}) \ar[r] & wS_\bullet \mathcal{A} \times wS_\bullet \mathcal{B}}$$
induced by ``subobject'' and ``quotient'' functors is a weak equivalence. That is, the diagonal of its level-wise nerve is a weak homotopy  equivalence of simplicial sets.
\end{theorem}
\begin{proof}
We conclude this from the weak equivalence of object simplicial sets in Lemma~\ref{lem:HardLemmaMadeEasy} in exactly the same way that Waldhausen, on page 336 of \cite{WaldhausenAlgKTheoryI}, concludes his Theorem 1.4.2 (recalled here in Theorem~\ref{thm:Waldhausen_Additivity}) from object simplicial sets in the case $\cala=\calb=\calc$ of Lemma~\ref{lem:HardLemmaMadeEasy}.
\end{proof}

\section{Classical Additivity and Split Exact Sequences of Waldhausen Categories} \label{sec:split_exact_sequences_of_Waldhausen_categories}

We explain in this section how Waldhausen $wS_\bullet$ Additivity for $\cale(\cala,\calc,\calb)$, as recalled in Theorem~\ref{thm:Additivity_(equivalent_formulation)}, implies Waldhausen $\bfK$ Additivity for standard split exact sequences in Theorem~\ref{thm:K-Additivity_for_standard_split_exact_sequence}, that is, the Waldhausen $K$-theory functor $\bfK$ takes standard split exact sequences to split cofiber sequences. Standard split exact sequences are constructed from $\cale(\cala,\calc,\calb)$ in Example~\ref{examp:E(A,C,B)}. We also show $wS_\bullet$ Additivity and $\bfK$ Additivity hold for split exact sequences Waldhausen equivalent to standard ones in Corollary~\ref{cor:Additivity_GeneralAndSpectral} (the notions of Waldhausen equivalence for categories and sequences are in Sections~\ref{sec:appendix1} and \ref{sec:appendix_Wald_equivalences_of_ses}). Sufficient conditions for a split exact sequence to be Waldhausen equivalent to a standard one are in Proposition~\ref{prop:universal_split_exact_sequence}. We also develop the most basic properties of the notion of split exact sequences of Waldhausen categories. Analogous results for Waldhausen {\it quasi}categories (in particular also stable quasicategories) are in Section~\ref{sec:split_exact_sequences_of_Waldhausen_quasicategories}. Blumberg--Gepner--Tabuada discuss and relate various kinds of exact sequences of stable quasicategories, spectral categories, and triangulated categories in \cite[Section~5]{BlumbergGepnerTabuadaI}.  The notions of exact sequence and split exact sequence of Waldhausen quasicategories in this paper in the stable case are the same as in \cite{BlumbergGepnerTabuadaI}.

Exact sequences of various kinds have been prominent in algebraic $K$-theory for decades, for instance in work of Grothendieck, Quillen, Waldhausen, Thomason-Trobaugh \cite{ThomasonTrobaugh}, and Neeman \cite{NeemanBook,NeemanOnThomasonTrobaugh}. Recent work on a universal characterization of quasicategorical algebraic $K$-theory \cite[Definition~6.1]{BlumbergGepnerTabuadaI}, \cite{Barwick} has distinguished Waldhausen Additivity as a key property a functor on small stable quasicategories may possess: a functor on small stable quasicategories {\it satisfies additivity} if it sends split exact sequences to split cofiber sequences.

\begin{definition}[Exact Sequence of Waldhausen Categories] \label{def:exact_sequence_of_Waldhausen_categories}
Let $\mathcal{A}$, $\cale$, and $\calb$ be Waldhausen categories.
A sequence of exact functors
\begin{equation} \label{equ:exact_seq_of_Waldhausen_categories}
\xymatrix{\mathcal{A} \ar[r]^i & \mathcal{E} \ar[r]^f & \mathcal{B}}
\end{equation}
is called {\it exact} if
\begin{enumerate}
\item
the composite $f\circ i$ is the distinguished zero object $\ast$ of $\mathcal{B}$,
\item
the exact functor $i\colon \mathcal{A}\to \mathcal{E}$ is fully faithful, and
\item \label{condition:iii}
the restricted functor $f\vert_{\cale / \cala} \colon \cale / \cala \to \mathcal{B}$ is an equivalence of categories. Here $\cale / \cala$ is the full subcategory of $\mathcal{E}$ on the objects $E \in \cale $
such that $\cale (i(A),E)$ is a point for all $A \in \cala$.
\end{enumerate}
\end{definition}

\begin{definition}[Split Exact Sequence of Waldhausen Categories] \label{def:split-exact_sequence_of_Waldhausen_categories}
An exact sequence of Waldhausen categories and exact functors as in equation \eqref{equ:exact_seq_of_Waldhausen_categories} is called {\it split} if there exist exact functors
$$\xymatrix{\mathcal{A}  & \mathcal{E}  \ar[l]_-j & \mathcal{B} \ar[l]_-g}$$
right adjoint to $i$ and $f$ respectively, such that the unit $\text{Id}_{\mathcal{A}} \to ji$ and the counit $fg \to \text{Id}_{\mathcal{B}}$ are natural isomorphisms.
\end{definition}

\begin{remark} \label{rem:f_and_g_are_inverses}
In a split exact sequence of Waldhausen categories and exact functors, the functor $g$ is actually an inverse equivalence to $f\vert_{\cale / \cala}$.  We have $\cale(i(A),g(B))\cong\calb(f(i(A)),B)=\calb(\ast,B)=\text{pt}$ for all $A$ in $\cala$, so that $g$ goes into $\cale/\cala$. The counit $(f\vert_{\cale / \cala})g \to \text{Id}_\calb$ is a natural isomorphism by hypothesis, and the unit $\text{Id}_{\cale / \cala} \to g(f\vert_{\cale / \cala})$ is a natural isomorphism since the left adjoint $f\vert_{\cale / \cala}$ is fully faithful. See also \cite[Lemma~9.26]{FiorePaoli2010}.
\end{remark}

\begin{remark} \label{rem:jg_is_isomorphic_to_zero}
In a split exact sequence of Waldhausen categories and exact functors, the composite $j \circ g$ is naturally isomorphic to the distinguished zero object. The functor $g$ goes into $\cale/\cala$ by Remark~\ref{rem:f_and_g_are_inverses}, so for all $A \in \cala$ we have $\text{pt}=\cale(i(A),g(B))\cong\cala(A,jg(B))$, so $jg(B)$ is a terminal object of $\cala$, and isomorphic to $\ast$.
\end{remark}

\begin{remark}
The 2-functor $S_n$ sends a split exact sequence of Waldhausen categories to a split exact sequence of Waldhausen categories by Remark~\ref{rem:Sn_is_a_2-functor}.
\end{remark}

\begin{example}(Standard Split Exact Sequence from $\mathcal{E}(\cala,\calc,\calb)$) \label{examp:E(A,C,B)}
Any Waldhausen category $\mathcal{C}$ with selected sub Waldhausen categories $\mathcal{A}$ and $\mathcal{B}$ produces a split exact sequence as follows. Let $\cale$ be the Waldhausen category $\mathcal{E}(\mathcal{A},\mathcal{C},\mathcal{B})$ in Notation~\ref{not:E(A,C,B)}.
We define exact functors
\begin{equation} \label{equ:E(A,C,B)}
\xymatrix{\mathcal{A} \ar[r]^-i & \mathcal{E}(\cala,\calc,\calb) \ar[r]^-q \ar@/^1.3pc/[l]^s & \mathcal{B} \ar@/^1.3pc/[l]^g}
\end{equation}
by
\begin{equation} \label{equ:canonical_split_exact_functors_categorical}
\begin{array}{c}
\aligned
i(A) &= \left( \begin{array}{c} \xymatrix{A \ar@{>->}[r]^{=} & A \ar@{->>}[r] & \ast} \end{array} \right) \\
s\left( \begin{array}{c} \xymatrix{A \ar@{>->}[r]^{\phantom{=}} & C \ar@{->>}[r] & B} \end{array} \right)&=A \\
q\left( \begin{array}{c} \xymatrix{A \ar@{>->}[r]^{\phantom{=}} & C \ar@{->>}[r] & B} \end{array} \right)&=B \\
g(B)&=\left(\begin{array}{c} \xymatrix{\ast \ar@{>->}[r] & B \ar@{->>}[r]^{=} & B}\end{array} \right).
\endaligned
\end{array}
\end{equation}

Clearly, $q\circ i = \ast$. The unit and counit for the adjunction $i \dashv s$ are $\eta_A = 1_A \colon A \to siA$ and
$$\begin{array}{c} \varepsilon_{ABC}\colon is(ACB) \to ACB \end{array} \;\; \;\;
\begin{array}{c} \xymatrix{A \ar@{>->}[r]^{=} \ar[d]_{=} & A \ar@{->>}[r] \ar[d]_{m} & \ast \ar[d] \\ A \ar@{>->}[r]_{m} & C \ar@{->>}[r] & B}
\end{array}$$
Since $s\varepsilon_A=1_A$ and $\varepsilon_{iA}=1_{iA}$, the triangle identities clearly hold. Since the unit is a natural isomorphism, the left adjoint $i$ is fully faithful.

The unit and the counit for the adjunction $q \dashv g$ are
$$\begin{array}{c} \eta_{ABC}\colon ACB \to gq(ACB) \end{array} \;\; \;\;
\begin{array}{c} \xymatrix{A \ar@{>->}[r] \ar[d] & C \ar@{->>}[r]^{n} \ar[d]^{n} & B \ar[d]^{=} \\
\ast \ar@{>->}[r] & B \ar@{->>}[r]^{=} & B
}
\end{array}$$
and $\varepsilon_B=1_B \colon qgB \to B$. Since $\eta_{gB}=1_{gB}$ and $q\eta_{ACB}=1_B$, the triangle identities clearly hold. Since the counit is a natural isomorphism, the right adjoint $g$ is fully faithful.

The subcategory $\cale / \cala$ is full on the cofiber sequences of the form
\begin{equation} \label{equ:BinE}
\xymatrix{\ast' \ar@{>->}[r] & C' \ar@{->>}[r]^-\cong & B'}
\end{equation}
with $\ast'$ a zero object of $\calc$ in $\cala$, the quotient map an isomorphism, and $B' \in \calb$, as follows.
The condition that there be only one map
\begin{equation} \label{equ:one_map}
\begin{array}{c}
\xymatrix{A \ar@{>->}[r]^= \ar[d] & A \ar@{->>}[r] \ar[d] & \ast \ar[d] \\ A' \ar@{>->}[r] & C' \ar@{->>}[r] & B'}
\end{array}
\end{equation}
for each $A \in \cala$ implies that $A'$ is a terminal object of $\cala$ (the middle map is determined by the left map), so $A'$ is then isomorphic to the distinguished zero object of $\cala$ (hence also to that of $\calc$), so $A'$ is a zero object of $\calc$ in $\cala$.  The morphism $C' \twoheadrightarrow B'$ is an isomorphism, as it is a pushout of the isomorphism $A' \to \ast$. Because of the form \eqref{equ:BinE} of objects in $\cale / \cala$, we now see that $q\vert_{\cale / \cala} \colon \cale / \cala \to \mathcal{B}$ is fully faithful (the right vertical map in a morphism of objects \eqref{equ:BinE} uniquely determines the middle vertical map via the isomorphisms). Clearly, $q\vert_{\cale / \cala}$ is surjective on objects, and we have $q\vert_{\cale / \cala}$ is an equivalence of categories. Finally, the sequence \eqref{equ:E(A,C,B)} is split exact.

Incidentally, we can also see directly that $g\colon \calb \to \cale/\cala$ is essentially surjective: the isomorphism $C' \twoheadrightarrow B'$ in \eqref{equ:BinE} is part of an isomorphism from \eqref{equ:BinE} to $\ast \rightarrowtail B' = B'$. From above, we already know $g$ is fully faithful, hence an equivalence.
\end{example}

\begin{proposition}[Sufficient Conditions for Waldhausen Equivalence with Standard Split Exact Sequence] \label{prop:universal_split_exact_sequence}
Suppose a split exact sequence of Waldhausen categories and exact functors
$$\xymatrix{\mathcal{A} \ar[r]^-i & \mathcal{E} \ar[r]^-f \ar@/^1pc/[l]^j & \mathcal{B} \ar@/^1pc/[l]^g}$$
has the following three properties.
\begin{enumerate}
\item \label{item:i:prop:universal_split_exact_sequence}
Each counit component $ij(E) \to E$ is a cofibration.
\item \label{item:i':prop:universal_split_exact_sequence}
For each cofibration $E \rightarrowtail E'$ in $\cale$, the induced map $$E \cup_{ij(E)} ij(E') \to E'$$ is a cofibration in $\cale$.
\item \label{item:ii:prop:universal_split_exact_sequence}
In every cofiber sequence in $\cala$ of the form
$A_0 \rightarrowtail A_1 \twoheadrightarrow \ast$, the first map is an isomorphism.
\end{enumerate}
Then it is Waldhausen equivalent (see Definitions~\ref{def:Waldhausen_equivalence} and \ref{def:Waldhausen_equiv_of_sequences_quasicategorical} and Proposition~\ref{prop:Waldhausen_equiv_of_sequences_compatible_with_j's}) to a split exact sequence of the form \eqref{equ:E(A,C,B)} in Example~\ref{examp:E(A,C,B)}.
\end{proposition}
\begin{proof}
Let $\calc:=\cale$ and let $\cala', \calb' \subseteq \calc$ be the essential images of $\cala$ and $\calb$ in $\calc$ under $i$ and $g$ respectively.

Since $i\co \cala \to \calc$ is fully faithful by definition of exact sequence, $i$ provides an equivalence with its essential image $\cala'$. Since the counit $fg \to \text{Id}_{\mathcal{B}}$ is a natural isomorphism by split exactness, the right adjoint $g\co \calb \to \calc$ is fully faithful and provides an equivalence with its essential image $\calb'$.
\begin{equation} \label{equ:prop:universal_split_exact_sequence:essential_images}
\begin{array}{c}
\xymatrix{\mathcal{A} \ar[r]^-i \ar[d]_i  & \mathcal{E} \ar[r]^-f \ar[d]_\Phi & \mathcal{B} \ar[d]^g \\ \mathcal{A}' \ar[r] & \mathcal{E}(\cala',\calc,\calb') \ar[r] & \mathcal{B}'}
\end{array}
\end{equation}
The vertical functors $i$ and $g$ reflect weak equivalences and cofibrations because the unit $\text{Id}_{\mathcal{A}} \to ji$ and counit $fg \to \text{Id}_{\mathcal{B}}$ are natural isomorphisms and $j$ and $f$ are exact. The vertical functors $i$ and $g$ are now Waldhausen equivalences (see Proposition~\ref{prop:exact_equivalence_is_Waldhausen_equivalence}).

Let $\Phi(E)$ be
\begin{equation} \label{equ:counit_unit_E}
\xymatrix@C=3pc{ij(E) \ar@{>->}[r]^-{\text{counit}} & E \ar[r]^-{\text{unit}} & gf(E)}.
\end{equation}
We claim that \eqref{equ:counit_unit_E} is in fact a cofiber sequence. Consider the pushout $P$ in the left diagram below. \\
\begin{equation} \label{equ:prop:universal_split_exact_sequence:pushouts}
\begin{array}{c}
\xymatrix@C=3pc@R=3pc{
      ij(E) \ar@{>->}[r]^-{\text{counit}} \ar[d] \ar@{}[dr]|{\text{pushout}}  &  E \ar[d] \ar@/^1pc/[ddr]^{\text{unit}}  &  \\
      \ast \ar@{>->}[r] \ar@/_1pc/[drr]  &  P \ar@{-->}[dr]^{\exists \text{!}}  &  \\
      &  &  gf(E) }
\end{array}
\begin{array}{c} \xymatrix@C=3.5pc@R=3pc{
      jij(E) \ar@{>->}[r]^-{j(\text{counit})} \ar[d] \ar@{}[dr]|{\text{pushout}}  &  j(E) \ar[d] \ar@/^1pc/[ddr]^{j(\text{unit})}  &  \\
      \ast \ar@{>->}[r] \ar@/_1pc/[drr]  &  j(P) \ar@{-->}[dr]^{\exists \text{!}}  &  \\
      &  &  jgf(E) }
\end{array}
\end{equation}
The outer square of the left diagram commutes because $g$ goes in $\cale/\cala$ by Remark~\ref{rem:f_and_g_are_inverses}, and $j(E) \in \cala$. The right diagram is obtained by applying the exact functor $j$ to the entire left diagram.
By a triangle identity for the $i,j$ adjunction, the map $j(\text{counit})$ in the right diagram is an isomorphism, so that $\ast \rightarrowtail j(P)$ is also an isomorphism. For all $A \in \cala$, $\cale(i(A),P)\cong\cala(A,j(P))\cong\cala(A,\ast)=\text{pt}$, so that $P \in \cale/\cala$. By Remark~\ref{rem:f_and_g_are_inverses}, there exists some $Q \in \calb$ such that $P \cong g(Q)$.

Consider now the left diagram with $P$ replaced by $g(Q)$, and its right vertical map and dashed map adjusted accordingly to have a pushout and the induced map. Apply $f$ to the entire altered left diagram to obtain
\begin{equation} \label{equ:prop:universal_split_exact_sequence:pushout_image}
\begin{array}{c}
\xymatrix@C=3.5pc@R=3pc{
      \ast \ar@{>->}[r]^-{f(\text{counit})} \ar[d] \ar@{}[dr]|{\text{pushout}}  &  f(E) \ar[d] \ar@/^1pc/[ddr]^{f(\text{unit})}  &  \\
      \ast \ar@{>->}[r] \ar@/_1pc/[drr]  &  fg(Q) \ar@{-->}[dr]^{\exists \text{!}}  &  \\
      &  &  fgf(E).}
\end{array}
\end{equation}
The vertical map $f(E) \to fg(Q)$ is an isomorphism, as it is a pushout of an isomorphism. By a triangle identity for the $f,g$ adjunction, the map $f(\text{unit})$ is an isomorphism, so the dashed map $fg(Q) \dashrightarrow fgf(E)$ is also an isomorphism by 3-for-2. Its origin, the dashed map $g(Q) \dashrightarrow gf(E)$ in the altered left diagram, is in the image of the fully faithful functor $g$, and hence in $\cale/\cala$. Since $f\vert_{\cale/\cala}$ reflects isomorphisms (it is fully faithful), the map $g(Q) \dashrightarrow gf(E)$ is also an isomorphism. Finally, the outer square of the altered left diagram is isomorphic to a pushout square and we conclude \eqref{equ:counit_unit_E} is indeed a cofiber sequence.

The functor $\Phi$ defined as in \eqref{equ:counit_unit_E} maps cofibrations to cofibrations by hypothesis \ref{item:i':prop:universal_split_exact_sequence}. It sends $\ast_\calc$ to the zero object of $\cale(\cala,\calc,\calb)$ and weak equivalences to weak equivalences because $ij$ and $gf$ preserve distinguished zero objects and weak equivalences. We will know $\Phi$ sends pushouts along cofibrations to pushouts along cofibrations as soon as we see that $\Phi$ is an equivalence of categories.

We define $\Psi\colon \cale(\cala',\calc,\calb') \to \cale$ to be the projection to $\calc=\cale$. Clearly $\Psi \circ \Phi=\text{Id}_{\cale}$, so that $\Phi$ is faithful. For fullness, a morphism $\Phi(E) \to \Phi(F)$ must come from the associated map $E \to F$ by the universality of $ij(F) \rightarrowtail F$ and the universality of $E \twoheadrightarrow gf(E)$, as well as the fully faithfulness of $i$ and $g$.

To show $\Phi\colon \cale \to \cale(\cala',\calc,\calb')$ is essentially surjective, it suffices to prove that the top row of the middle diagram
\begin{equation} \label{equ:prop:universal_split_exact_sequence:Phi_essentially_surjective}
\begin{array}{c}
\xymatrix{A \ar@{-->}[d]_{\exists \text{!}}^m & i(A) \ar@{-->}[d]_{i(m)} \ar@{>->}[r] & E \ar@{->>}[r] \ar@{=}[d] & g(B) & B \\ j(E) & ij(E) \ar@{>->}[r] & E \ar@{->>}[r] & gf(E) \ar@{-->}[u]_{g(n)} & f(E) \ar@{-->}[u]^{\exists \text{!}}_n }
\end{array}
\end{equation}
is isomorphic to the bottom row of the middle diagram via the dashed arrows coming from the universality of the counit and unit. An application of $f$ to the middle diagram maps both left horizontal arrows to $\ast \rightarrowtail f(E)$ and therefore the right horizontal arrows to isomorphisms. Thus $f(g(n))$ is an isomorphism, and $g(n)$ is an isomorphism ($f$ is an equivalence of categories on $\cale/\cala$, so reflects isomorphisms there). On the other hand, an application of $j$ to the middle diagram maps both right horizontal arrows to $j(E) \twoheadrightarrow \ast'$ by Remark~\ref{rem:jg_is_isomorphic_to_zero}, and therefore $j$ maps the left horizontal arrows to isomorphisms by hypothesis \ref{item:ii:prop:universal_split_exact_sequence}. Therefore $ji(m)$ is an isomorphism, and $m$ is also an isomorphism using the unit isomorphism of the $i,j$ adjunction.

The exact equivalence $\Phi$ reflects weak equivalences and cofibrations (simply project to $\calc=\cale$), so $\Phi$ is a Waldhausen equivalence by Proposition~\ref{prop:exact_equivalence_is_Waldhausen_equivalence}.
\end{proof}

\begin{remark}[Comments on the Conditions of Proposition~\ref{prop:universal_split_exact_sequence} for Standardness]
Conditions \ref{item:i:prop:universal_split_exact_sequence} and \ref{item:i':prop:universal_split_exact_sequence} of Proposition~\ref{prop:universal_split_exact_sequence} hold in any standard split exact sequence, so we can expect to require them for a split exact sequence to be Waldhausen equivalent to a standard one.

Condition \ref{item:ii:prop:universal_split_exact_sequence} is known to hold for several kinds of categories $\cala$.
If $\cala$ is some category of $R$-modules with monic cofibrations, then \ref{item:ii:prop:universal_split_exact_sequence} holds: if $A_0 \rightarrowtail A_1 \rightarrowtail \ast$ is a cofiber sequence, then the monomorphism $A_0 \rightarrowtail A_1$ is also surjective since $A_1 / A_0$ is the zero module, so $A_0 \rightarrowtail A_1$ is an isomorphism. If $\cala$ is the category of finite pointed sets with cofibrations the monomorphisms, then \ref{item:ii:prop:universal_split_exact_sequence} holds by similar reasoning. Or, if $\cala$ is a category in which cofibrations are summand inclusions, then \ref{item:ii:prop:universal_split_exact_sequence} holds, because the other summand is the cofiber, which is trivial.

Condition \ref{item:ii:prop:universal_split_exact_sequence} in Proposition~\ref{prop:universal_split_exact_sequence} also holds if every cofiber sequence is a fiber sequence, because then any pullback of the identity $\ast = \ast$ is an isomorphism.

Our true interest lies in the quasicategorical version of Proposition~\ref{prop:universal_split_exact_sequence} and the upcoming Corollary~\ref{cor:Additivity_GeneralAndSpectral}, namely Proposition~\ref{prop:universal_split_exact_sequence:quasicategorical} and Corollary~\ref{cor:QCat_Additivity_GeneralAndSpectral}. The quasicategorical conditions \ref{item:i:prop:universal_split_exact_sequence}, \ref{item:i':prop:universal_split_exact_sequence}, and \ref{item:ii:prop:universal_split_exact_sequence} hold for a large class of interesting Waldhausen quasicategories, which forms the domain of the additive invariants of \cite{BlumbergGepnerTabuadaI}, namely the {\it stable quasicategories}. See Definition~\ref{def:stable_quasicategories}, Example~\ref{examp:stable_quasicategories_are_Waldhausen_quasicategories}, and Corollary~\ref{cor:stable_qcats_are_example_for_main_theorem}. Condition \ref{item:ii:prop:universal_split_exact_sequence} holds in every {\it exact $\infty$-category} in the sense of \cite{BarwickThmOfTheHeart}.
\end{remark}

\begin{remark}
In the proof of Proposition~\ref{prop:universal_split_exact_sequence}, the squares
\begin{equation} \label{equ:s_q_equivalence}
\begin{array}{c}
\xymatrix{\cala \ar[d]_i & \cale \ar[l]_j \ar[r]^f \ar[d]_{\Phi} & \calb \ar[d]^g \\ \cala'  & \cale(\cala',\calc,\calb') \ar[l]_-s \ar[r]^-q & \calb' }
\end{array}
\end{equation}
commute strictly, while the squares
$$\xymatrix{\cala \ar[d]_i \ar[r]^i & \cale \ar[d]_{\Phi} & \calb \ar[d]^g \ar[l]_g
\\ \cala' \ar[r] & \cale(\cala',\calc,\calb') & \calb' \ar[l] }$$
commute only up to natural isomorphism.
\end{remark}

Let $\bfK(\calc)$ denote the Waldhausen $K$-theory spectrum of the Waldhausen category $\calc$, its $n$-th space is
$\bfK(\calc)_n:= \vert wS_\bullet \cdots S_\bullet \calc\vert$
for $n\geq0$ where $S_\bullet$ appears $n$ times. The structure maps are defined in \cite[pages 329--330]{WaldhausenAlgKTheoryI}, see also \cite[Sections~8.4, 8.6, and 8.7]{RognesTextbook}. Waldhausen proved, as a consequence of $wS_\bullet$ Additivity, this is an $\Omega$-spectrum beyond the 0-th term, that is,
$\vert w\calc\vert \to \Omega \vert wS_\bullet\calc\vert$ might not be a weak equivalence of spaces, though
$\vert wS_\bullet^{(n)}\calc\vert \to \Omega \vert wS_\bullet^{(n+1)}\calc\vert$ is a weak equivalence of spaces for $n \geq 1$.

\begin{theorem}[Waldhausen $\bfK$ Additivity for Standard Split Exact Sequences] \label{thm:K-Additivity_for_standard_split_exact_sequence}
Let $\mathcal{C}$ be a Waldhausen category and $\mathcal{A}$ and $\mathcal{B}$ sub Waldhausen categories and consider the standard split exact sequence in \eqref{equ:E(A,C,B)}. Then
$$\xymatrix{\bfK(i)\vee\bfK(g)\colon \bfK(\cala) \vee \bfK(\calb) \ar[r] & \bfK(\cale(\cala, \calc, \calb))}$$
is a stable equivalence of spectra.
\end{theorem}
\begin{proof}
First observe that an application of $S_n$ to a standard split exact sequence \eqref{equ:E(A,C,B)} for $\cala,\calb \subseteq \calc$ yields the standard split exact sequence for $S_n\cala,\;S_n\calb \subseteq S_n\calc$ because
$S_n\cale(\cala,\calc,\calb)\cong\cale(S_n\cala,\;S_n\calc,\;S_n\calb)$. So we may apply Theorem~\ref{thm:Additivity_(equivalent_formulation)} to obtain weak homotopy equivalences of spaces $(\bfK(s)_n,\bfK(q)_n)$, which then assemble to a stable equivalence of spectra (level-wise equivalence of spectra implies stable equivalence).
In spectra, the wedge product and the product are formed levelwise, and the inclusion of a wedge product into a product is a stable equivalence, so $\bfK(\cala) \vee \bfK(\calb) \hookrightarrow \bfK(\cala) \times \bfK(\calb)$ is a stable equivalence.
$$\xymatrix@C=6pc{\bfK(\cala)_n \vee \bfK(\calb)_n \ar[r]^-{\bfK(i)_n\vee\bfK(g)_n} \ar@{^{(}->}@/_2pc/[rr]_-{\text{stably w.e.}} & \bfK(\cale(\cala, \calc, \calb))_n \ar[r]^-{(\bfK(s)_n,\bfK(q)_n)}_{\text{w.e. by Thm~\ref{thm:Additivity_(equivalent_formulation)}}} & \bfK(\cala)_n \times \bfK(\calb)_n}$$
By the 3-for-2 property of stable equivalences, we now have $\bfK(i)\vee\bfK(g)$ is a stable equivalence.
\end{proof}

\begin{corollary}[Waldhausen $wS_\bullet$ Additivity and $\bfK$ Additivity for Split Exact Sequences Waldhausen Equivalent to Standards] \label{cor:Additivity_GeneralAndSpectral}
Suppose a split exact sequence of Waldhausen categories and exact functors
$$\xymatrix{\mathcal{A} \ar[r]^-i & \mathcal{E} \ar[r]^-f \ar@/^1pc/[l]^j & \mathcal{B} \ar@/^1pc/[l]^g}$$
is Waldhausen equivalent to a standard split exact sequence of the form \eqref{equ:E(A,C,B)}, for instance any split exact sequence satisfying the hypotheses of Proposition~\ref{prop:universal_split_exact_sequence}. Then the following hold.
\begin{enumerate}
\item \label{cor:Additivity_GeneralAndSpectral:(i)}
The map
$$\xymatrix{wS_\bullet(j,f)\colon wS_\bullet \mathcal{E} \ar[r] & wS_\bullet \mathcal{A} \times wS_\bullet \mathcal{B}}$$
is a weak equivalence of simplicial objects in $\mathbf{Cat}$. That is, the diagonal of its level-wise nerve is a weak homotopy  equivalence of simplicial sets.
\item \label{cor:Additivity_GeneralAndSpectral:(ii)}
The functors $i$ and $g$ induce a stable equivalence of $K$-theory spectra
\begin{equation} \label{equ:cor:Additivity_GeneralAndSpectral:(ii)}
\xymatrix{\bfK(i)\vee\bfK(g)\colon \bfK(\cala) \vee \bfK(\calb) \ar[r] & \bfK(\cale).}
\end{equation}
\end{enumerate}
\end{corollary}
\begin{proof}
\begin{enumerate}
\item
Suppose we have a Waldhausen equivalence of sequences from $\cala \to \cale \to \calb$ to a standard one $\cala' \to \cale' \to \calb'$, see Definition~\ref{def:Waldhausen_equiv_of_sequences_quasicategorical}.
The diagram
$$\xymatrix@C=6pc{wS_\bullet \mathcal{E} \ar[r]^-{wS_\bullet(j,f)} \ar[d]_{\text{w.e. by Cor~\ref{cor:Waldhausen_equivalences_induce_weak_equivalences}}} & wS_\bullet \mathcal{A} \times wS_\bullet \mathcal{B} \ar[d]^{\text{w.e. by Cor~\ref{cor:Waldhausen_equivalences_induce_weak_equivalences}}} \\
wS_\bullet \mathcal{E}' \ar[r]_-{\text{w.e. by Thm~\ref{thm:Additivity_(equivalent_formulation)}}}^-{wS_\bullet(s,q)} & wS_\bullet \mathcal{A}' \times wS_\bullet \mathcal{B}'}$$
commutes up to a natural isomorphism for each $\bullet=n$ by the categorical version of Proposition~\ref{prop:Waldhausen_equiv_of_sequences_compatible_with_j's}. The right vertical map is a weak equivalence because geometric realization and $\pi_*$ preserve finite products. The geometric realization of the left-bottom composite is $\pi_\ast$-isomorphism, so the top-right composite is also a $\pi_\ast$-isomorphism (they are homotopic). Thus the top-right composite is a weak equivalence, and the top map is a weak equivalence by the 3-for-2 property.
\item
We similarly draw a commutative square involving \eqref{equ:cor:Additivity_GeneralAndSpectral:(ii)} and its standardized version, then use Theorem~\ref{thm:K-Additivity_for_standard_split_exact_sequence}, Corollary~\ref{cor:Waldhausen_equivalences_induce_weak_equivalences}, the categorical version of Proposition~\ref{prop:Waldhausen_equiv_of_sequences_compatible_with_j's}, and the 3-for-2 property of stable equivalences.
\end{enumerate}
\end{proof}

\begin{remark}
The $K$-theory spectrum can actually be made into a symmetric spectrum, see the Appendix of Geisser--Hesselholt \cite{GeisserHesselholt}, or \cite[2.2]{BlumbergMandellDerivedKoszul}, \cite[A.5.4]{BlumbergMandellAbstHomTheory}, \cite[Section~2]{RognesRankFiltration}. The $K$-theory spectrum functor can also be made into a multi-functor, see Section~2.3.2 of the thesis of Zakharevich \cite{ZakharevichThesis}.
\end{remark}

\section{Recollections about Quasicategories} \label{sec:recollections}

We summarize the results about simplicial sets, quasicategories, Kan complexes, and groupoids that we will
freely use from Joyal's papers \cite{JoyalQCatsAndKanComplexes} and \cite{JoyalQuadern}, Lurie's books \cite{LurieHigherToposTheory} and \cite{LurieHigherAlgebra}, and the classical theory of simplicial sets.
We also prove some new results. To make the present paper as self-contained as possible, we have attempted to use only the elementary aspects of quasicategories. The main new technical result is a natural pushout functor along cofibrations that preserves identities, see Section~\ref{subsec:natural_pushout_along_cofibrations}, which relies on a corollary of Riehl--Verity. Throughout the paper we refer back to this summary in order to streamline later arguments.

\subsection{Quasicategories, Mapping Spaces, and Fullness} \label{subsec:Quasicategories}

A {\it quasicategory} is a simplicial set $X$ in which every inner horn admits a filler. That is, for any $0<k<n$ and any map $\Lambda^k[n] \to X$, there exists a map $\Delta[n] \dashrightarrow X$ such that the diagram
$$\xymatrix{\Lambda^k[n] \ar[r] \ar@{^{(}->}[d] & X  \\ \Delta[n] \ar@{-->}[ur]_{\exists} & }$$
commutes. The notion is originally due to Boardman and Vogt \cite{BoardmanVogt} under the name of {\it weak Kan complex}, and has been developed extensively by Joyal \cite{JoyalQuadern} and Lurie \cite{LurieHigherToposTheory}, \cite{LurieHigherAlgebra}. Any Kan complex, for example the singular complex of a topological space, is a quasicategory. The nerve of any category is also a quasicategory. An {\it object of a quasicategory $X$} is a vertex of $X$, that is, an element of $X_0$. A {\it map} or {\it morphism in $X$} is a 1-simplex $f$, its {\it source} is $d_1f$ and its {\it target} is $d_0f$. If $X$ is a quasicategory and $W$ is a simplicial subset of $X$ which is 0-full (see below) in $X$ on its set of vertices $W_0$, then $W$ is also a quasicategory \cite[page 275]{JoyalQuadern}. A {\it sub quasicategory} of a quasicategory is a simplicial subset which is also a quasicategory.

If $X$ is a quasicategory and $A$ is a simplicial set, then the ordinary simplicial mapping space $X^A=\text{Map}(A,X)$ is a quasicategory \cite[Corollary 2.19]{JoyalQuadern}. As usual, the {\it $n$-simplices of the ordinary simplicial mapping space} are
$$(X^A)_n:=\text{Map}(A,X)_n:=\mathbf{SSet}(A \times \Delta[n],X).$$
If $X$ is a quasicategory and $A$ is a simplicial set, we write $\text{Fun}(A,X)$ for $X^A$ and call $\text{Fun}(A,X)$ the {\it quasicategory of functors from $A$ to $X$}. The 0-simplices of $\text{Fun}(A,X)$ are the {\it functors} from $A$ to $X$, these are simply maps of simplicial sets $A \to X$. The 1-simplices of $\text{Fun}(A,X)$ are called {\it natural transformations}. These are the maps of simplicial sets $A \times \Delta[1] \to X$.
The mapping space construction in simplicial sets (and in quasicategories) is right adjoint to the product. The natural bijection in detail is the map
\begin{equation} \label{equ:Fun_adjunction}
\begin{array}{c}
\xymatrix@R=.5pc{\mathbf{SSet}(T \times U, V) \ar[r]^-{(-)^\dagger} & \mathbf{SSet}(T,V^U) \\
L \ar@{|->}[r] & L^\dagger}
\end{array}
\end{equation}
where $L^\dagger$ assigns to an $n$-simplex $t \in T_n$ the map $L^\dagger(t)\co U \times \Delta[n] \to V$, which on an $m$-simplex $(u,q)\in U_m\times \Delta([m],[n])$ is
\begin{equation} \label{equ:Fun_adjunction_formula}
L^\dagger(t)(u,q):=L_m(q^*t,u).
\end{equation}
More diagrammatically, $L^\dagger$ is
$$L^\dagger =\Bigg(\big( \xymatrix@C=1pc{\overline{t}\co \Delta[n] \ar[r] & T} \big) \xymatrix@C=1.5pc{\ar@{|->}[r] & } \big( \xymatrix@C=2.7pc{U \times \Delta[n] \ar[r]^-{\text{Id}_U \times \overline{t}} &  U \times T \ar[r]^\cong & T \times U \ar[r]^-L & V }\big) \Bigg).$$
The bijection in \eqref{equ:Fun_adjunction} is natural in the simplicial sets $T$, $U$, and $V$.

The strict 1-category of small quasicategories is enriched in quasicategories via the mapping space $\text{Fun}(X,Y)$, and thus forms an $(\infty,2)$-category denoted by $\mathbf{QCat}_{\infty,2}$. More precisely, this simplicial enrichment of $\mathbf{QCat}$ is inherited from the simplicial enrichment of $\mathbf{SSet}$ to $\mathbf{SSet}_\text{simp}$ via the internal homs $C^B$ as follows. From \cite[page 83]{GoerssJardine}, the enriched composition in $\mathbf{SSet}_\text{simp}$
\begin{equation} \label{equ:enriched_composition}
\text{comp}\co C^B \times B^A \to C^A
\end{equation}
assigns to $m$-simplices $g$ and $f$ the map $\text{comp}_m(g,f)\co A \times \Delta[m] \to C$, which on a $k$-simplex $(a,p)\in A_k \times \Delta([k],[m])$ is
\begin{equation} \label{equ:enriched_composition_formula}
\text{comp}_m(g,f)(a,p):=g_k(f_k(a,p),p).
\end{equation}
More diagrammatically,
$$\text{comp}_m(g, f) =\Big(\xymatrix@C=3.5pc{A \times \Delta[m] \ar[r]^-{\text{Id}_A \times \text{diag}} & A \times \Delta[m] \times \Delta[m] \ar[r]^-{f \times \text{Id}_{\Delta[m]}} & B \times \Delta[m] \ar[r]^-g & C} \Big).$$
Now $\mathbf{QCat}_{\infty,2}$ is the sub simplicially enriched category of $\mathbf{SSet}_\text{simp}$ that is full on objects that are quasicategories. When $X$ and $Y$ are quasicategories, then $Y^X=\text{Fun}(X,Y)$ is also a quasicategory, so the simplicial enrichment of $\mathbf{QCat}_{\infty,2}$ is actually an enrichment in $(\infty,1)$-categories, and we can properly call $\mathbf{QCat}_{\infty,2}$ an $(\infty,2)$-category.

For later use, we prove that the mapping space provides us with simplicial functors $(-)^A$ and $C^{(-)}$, derive formulas for these simplicial functors on the $n$-simplices of hom simplicial sets, and conclude that any simplicial homotopy $\alpha$ induces simplicial homotopies $\alpha^A$ and $C^\alpha$.
\begin{proposition}[Exponentiation is a Simplicial Functor in Each Variable] \label{prop:exp_is_simplicial} \leavevmode
\begin{enumerate}
\item \label{prop:exp_is_simplicial:i}
For any simplicial set $A$, ``exponentiation to $A$'' is a simplicially enriched functor $(-)^A\co \mathbf{SSet}_\text{\rm simp} \to \mathbf{SSet}_\text{\rm simp}$. Its map of hom simplicial sets
\begin{equation} \label{equ:making_exponentiation_to_A_simplicial}
\xymatrix{C^B \ar[r] & \big(C^A\big)^{\big( B^A \big)}}
\end{equation}
assigns to an $n$-simplex $g$ the $n$-simplex $g^A\co B^A \times \Delta[n] \to C^A$ defined by
\begin{equation} \label{equ:making_exponentiation_to_A_simplicial_formula}
g^A(f,q):=\text{\rm comp}_m(q^* g,f)
\end{equation}
on each $m$-simplex $(f,q) \in \big(B^A\big)_m \times \Delta([m],[n])$.
\item \label{prop:exp_is_simplicial:ii}
The restriction of ``exponentiation to $A$'' to quasicategories is a simplicially enriched functor $(-)^A\co \mathbf{QCat}_{\infty,2} \to \mathbf{QCat}_{\infty,2}$  which is an $(\infty,2)$-functor.
\item \label{prop:exp_is_simplicial:iii}
For any simplicial set $C$, ``exponentiation with base $C$'' is a contravariant simplicially enriched functor $C^{(-)}\co \mathbf{SSet}_\text{\rm simp} \to \mathbf{SSet}_\text{\rm simp}$. Its map of hom simplicial sets
\begin{equation} \label{equ:making_exponentiation_base_C_simplicial}
\xymatrix{B^A \ar[r] & \big(C^A\big)^{\big( C^B \big)}}
\end{equation}
assigns to an $n$-simplex $f$ the $n$-simplex $C^f\co C^B \times \Delta[n] \to C^A$ defined by
\begin{equation} \label{equ:making_exponentiation_base_C_simplicial_formula}
C^f(g,q):=\text{\rm comp}_m(g,q^*f)
\end{equation}
for each $m$-simplex $(g,q) \in \big(C^B\big)_m \times \Delta([m],[n])$.
\item \label{prop:exp_is_simplicial:iv}
If $C$ is a quasicategory, the restriction of ``exponentiation with base $C$'' to quasicategories is a contravariant simplicially enriched functor $C^{(-)}\co \mathbf{QCat}_{\infty,2} \to \mathbf{QCat}_{\infty,2}$ which is an $(\infty,2)$-functor.
\end{enumerate}
\end{proposition}
\begin{proof} \leavevmode
\begin{enumerate}
\item
We take $L$ in the adjunction \eqref{equ:Fun_adjunction} and \eqref{equ:Fun_adjunction_formula} to be $\text{comp}\co C^B \times B^A \to C^A$, the enriched composition in $\mathbf{SSet}_\text{simp}$ indicated in \eqref{equ:enriched_composition} and \eqref{equ:enriched_composition_formula}. From this we obtain
the map of simplicial sets $\text{comp}^\dagger \co C^B \to \big(C^A\big)^{\big( B^A \big)}$ in \eqref{equ:making_exponentiation_to_A_simplicial}. Then for $g\co B \times \Delta[n] \to C$ and an $m$-simplex $(f,q) \in \big(B^A\big)_m \times \Delta([m],[n])$ we have
$$g^A(f,q)=\text{comp}^\dagger(g)(f,q)=\text{comp}_m(q^*g,f).$$
\item
This follows directly from \ref{prop:exp_is_simplicial:i} and the fact that $X^A$ is a quasicategory when $X$ is.
\item
We take $L$ in the adjunction \eqref{equ:Fun_adjunction} and \eqref{equ:Fun_adjunction_formula} to be $\text{comp}\circ \sigma \co B^A \times C^B  \to C^A$, the enriched composition composed with the transposition $\sigma=(12)$. From this we obtain the map of simplicial sets $(\text{comp}\circ\sigma)^\dagger\co B^A \to \big(C^A\big)^{\big( C^B \big)}$ in \eqref{equ:making_exponentiation_base_C_simplicial}. Then for $f\co A \times \Delta[n] \to B$ and an $m$-simplex $(g,q) \in \big(C^B\big)_m \times \Delta([m],[n])$ we have
$$C^f(g,q)=(\text{comp}\circ\sigma)^\dagger(f)(g,q)=(\text{comp}\circ\sigma)_m(q^*f,g)=\text{comp}_m(g,q^*f).$$
\item
This follows directly from \ref{prop:exp_is_simplicial:iii} and the fact that $C^X$ is a quasicategory when $C$ is.
\end{enumerate}
\end{proof}
\begin{corollary}[Exponentiations of Simplicial Homotopies are Simplicial Homotopies] \label{cor:exponentiating_simplicial_homotopies}
If $\alpha\co D \times \Delta[1] \to E$ is a simplicial homotopy from $\alpha_0$ to $\alpha_1$, then $\alpha^A$ is a simplicial homotopy from $\alpha_0^A$ to $\alpha_1^A$, and $C^\alpha$ is a simplicial homotopy from $C^{\alpha_0}$ to $C^{\alpha_1}$.
\end{corollary}
\begin{proof}
Both \eqref{equ:making_exponentiation_to_A_simplicial} and \eqref{equ:making_exponentiation_base_C_simplicial} are maps of simplicial sets, so the direction of the simplicial homotopies $\alpha^A$ and $C^\alpha$ are the same as the direction for $\alpha$, even though $C^{(-)}$ is a contravariant simplicial functor.
\end{proof}

Having now completed the discussion of simplicial mapping spaces for the enriched categories $\mathbf{SSet}_\text{simp}$ and $\mathbf{QCat}_{\infty,2}$, we next discuss a mapping space of morphisms {\it in a single quasicategory} $X$. For any two objects $x$ and $y$ of a quasicategory $X$, we have the {\it mapping space $X(x,y)$ of the quasicategory $X$}. This simplicial set is defined as the following pullback.
\begin{equation} \label{equ:mapping_space_of_a_quasicategory}
\begin{array}{c}
\xymatrix{X(x,y) \ar[r] \ar[d] \ar@{}[dr]|{\text{pullback}} & X^{\Delta[1]} \ar[d]^{(s,t)} \\ \ast \ar[r]_-{(x,y)} & X \times X }
\end{array}
\end{equation}
By Yoneda and the universal property of the pullback, an $n$-simplex of $X(x,y)$ is a map $\alpha\co\Delta[n] \times \Delta[1] \to X$ such that $\alpha(-,0)$ is $x = x = \cdots = x$ and $\alpha(-,1)$ is $y = y = \cdots = y$.
We use Joyal's notation $X(x,y)$, see \cite[page 158]{JoyalQuadern}. Instead of $X(x,y)$, Lurie writes $\text{Hom}_S(x,y)$, see \cite[page 28]{LurieHigherToposTheory}. If $\calc$ is a category, the mapping space $(N\calc)(x,y)$ is the same as $\calc(x,y)$ viewed as a discrete simplicial set. The mapping spaces $X(x,y)$ are Kan complexes, see \cite[Proposition~6.13]{CisinskiMoerdijk} and \cite[Corollary 17.2.2]{RiehlCategoricalHomotopyTheory}. Kan complexes are $\infty$-groupoids, so in this sense quasicategories are models for $(\infty,1)$-categories.

One drawback of this choice of mapping space $X(x,y)$ as the pullback in \eqref{equ:mapping_space_of_a_quasicategory} is that there is no composition
\begin{equation} \label{equ:no_composition_for_qc_mapping_spaces}
\xymatrix{X(y,z) \times X(x,y) \ar[r] & X(x,z)}.
\end{equation}
{\it Simplicial categories} and the left adjoint $\mathfrak{C}$ to $N^{\text{simp}}$ (recalled in Section~\ref{subsec:ho_of_a_qcat}) can be used to work around this in some situations. For a treatment of mapping spaces, see Dugger--Spivak \cite{DuggerSpivakMapping}.

Suppose $A$ and $B$ are simplicial sets and $A \subseteq B$. Then $A$ is {\it 0-full in $B$} if and only if any $n$-simplex of $B$ is in $A$ exactly when all its vertices lie in $A$. Similarly, $A$ is {\it 1-full in $B$} if and only if any $n$-simplex of $B$ is in $A$ exactly when all its edges lie in $A$.\footnote{A map of simplicial sets $f \colon A \to B$ is {\it $n$-full} if the unit naturality square
$$\xymatrix{A \ar[r] \ar[d]_f & \text{cosk}_n\text{tr}_n A \ar[d]^{\text{cosk}_n\text{tr}_n f} \\ B \ar[r] & \text{cosk}_n\text{tr}_n B}$$
is a pullback \cite[Definition B.0.11]{JoyalQuadern}. Here $\text{cosk}_n$ is the right adjoint to the $n$-th truncation functor $\text{tr}_n$. A simplicial subset of a simplicial set is {\it $n$-full} if the inclusion functor is $n$-full. The map $f$ is $n$-full if and only if it has the right lifting property with respect to the inclusion $\partial\Delta[m] \hookrightarrow \Delta[m]$ for all $m > n$ and each such lifting problem has a unique solution, see \cite[Proposition~B.0.12 and Definition~C.0.22]{JoyalQuadern}. Thus, $n$-full implies $(n+1)$-full.} 1-full does not imply 0-full, but 0-full implies 1-full. As previously remarked, if $X$ is a quasicategory and $A \subseteq X$ is 0-full, then $A$ is a quasicategory. The analogous statement for 1-full sub simplicial sets of a quasicategory is false, however we can draw the desired conclusion if we additionally require that composites exist in $A$.

\begin{proposition} \label{prop:1_full+composites_implies_quasicategory}
Let $X$ be a quasicategory and $A \subseteq X$ a sub simplicial set. Suppose that $A$ is 1-full in $X$ and every inner horn $\Lambda^1[2] \to A$ admits a filler in $A$.  Then $A$ is a quasicategory.
\end{proposition}
\begin{proof}
Let $n \geq 3$ and consider an inner horn $\sigma\co\Lambda^k[n] \to A$ as an inner horn in the quasicategory $X$ via the inclusion. Then we can extend $\sigma$ to $\overline{\sigma}\co \Delta[n] \to X$. But the edges of $\Delta[n]$ are the same as the edges of $\Lambda^k[n]$, since $n \geq 3$. Hence the edges of $\overline{\sigma}$ are the same as the edges of $\sigma$, which are in $A$, so $\overline{\sigma}$ has image in $A$ by 1-fullness.
\end{proof}

\subsection{The Homotopy Category of a Quasicategory and Equivalences {\it in} a Quasicategory} \label{subsec:ho_of_a_qcat}

Let $X$ be a quasicategory and $X(x,y)$ the Kan complex defined by \eqref{equ:mapping_space_of_a_quasicategory}. Two vertices of $X(x,y)$ (morphisms of $X$) are {\it homotopic} if there is a 1-simplex of $X(x,y)$ from one to the other; this is an equivalence relation on $X(x,y)$. The path components of $X(x,y)$ are the homotopy classes of 1-morphisms in $X$ from $x$ to $y$. Since $X$ is a quasicategory, any two morphisms in $X$ are homotopic in the above sense if and only if they are {\it left homotopic}, which is the case if and only if they are {\it right homotopic}, see \cite[Definition~1.8 and Lemma~1.9]{JoyalQuadern} and \cite{BoardmanVogt}. The vertices of $X$ and the homotopy classes of morphisms of $X$ form the
{\it Boardman--Vogt homotopy category $hoX$ of the quasicategory $X$}.

Another way to describe the homotopy category $hoX$ of a quasicategory $X$ is via the left adjoint $\tau_1$ to the fully faithful nerve functor $N \co \mathbf{Cat} \to \mathbf{SSet}$. The left adjoint $\tau_1$ sends a
simplicial set $A$ to the quotient of the free category on the graph $(A_0,A_1, d_0,d_1)$ by the relations
determined by 2-simplices \cite{GabrielZisman}. If $X$ is a quasicategory, then $\tau_1X$ coincides with the
Boardman--Vogt homotopy category $hoX$ and $\pi_0X(x,y)=\left(\tau_1X\right)(x,y)$ \cite[Proposition 1.11, page 213]{JoyalQuadern}. A useful fact about the functor $\tau_1$ is that it preserves finite products \cite[B.0.15]{JoyalQuadern}.

Simplicial sets, simplical categories, and categories are connected by the adjoint functors below, as used in \cite[2.16]{DuggerSpivakMapping} and in other papers on this topic.
$$\xymatrix@C=3pc{\mathbf{SSet} \ar@/^2.5pc/[rr]^-{\tau_1}
\ar@<.5ex>[r]^-{\mathfrak{C}} &  \ar@<.5ex>[l]^-{N^{\text{simp}}} \mathbf{SimpCat} \ar@<.5ex>[r]^-{\pi_0} &
\ar@<.5ex>[l]^-{\text{disc}} \ar@/^2.5pc/[ll]^-{N} \mathbf{\mathbf{Cat}}}$$
Here $N$ is the nerve functor of Grothendieck, the {\it categorification} $\tau_1$ is its left adjoint, $N^\text{simp}$ is the {\it homotopy coherent nerve} of Cordier-Porter \cite{Cordier}, \cite{CordierPorter}, and $\mathfrak{C}$ is its left adjoint \cite[Section 1.1.5]{LurieHigherToposTheory}, sometimes called {\it rigidification}. For any simplicial set $X$, the objects of $\mathfrak{C}(X)$ are the vertices of $X$.  The functor $\pi_0$ forms the homotopy category $\pi_0\calc$ of a simplicially enriched category $\calc$, which has $\left(\pi_0\mathcal{C}\right)(a,b)=\pi_0\left(\calc(a,b)\right)$ as above. The functor $\text{disc}$ means to consider a category as a simplicial category with discrete hom simplicial sets. The equalities $N=N^{\text{simp}}\text{disc}$ and $\tau_1=\pi_0\mathfrak{C}$ hold.

A morphism $f$ in a
quasicategory $X$ is said to be an {\it equivalence} if its image in the homotopy category $\tau_1 X$ is invertible.\footnote{Joyal prefers to call such a morphism an {\it isomorphism} in \cite{JoyalQuadern}: it is the quasicategorical version of the categorical notion of isomorphism.} The equivalences in a quasicategory $X$ satisfy the 3-for-2 property. That is, if $x\colon \Delta[2] \to X$ is a 2-simplex in $X$ such that two of the morphisms $\sigma^*_{0,1}(x)$, $\sigma^*_{1,2}(x)$, $\sigma^*_{0,2}(x)$ are equivalences, then so is the third. This follows from the 3-for-2 property of isomorphisms in $\tau_1X$ and the fact that a 2-simplex in $X$ gives rise to a commutative triangle in $\tau_1X$, and every equation of the form $[g][f]=[h]$ in the homotopy category of $X$ comes from a 2-simplex with the apparent boundary.

A quasicategory $X$ is a Kan complex if and only if $\tau_1X$ is a groupoid, see Joyal \cite[Corollary 1.4]{JoyalQCatsAndKanComplexes}  and \cite[Theorem 4.14]{JoyalQuadern}.

The inclusion of small groupoids into small categories, $\mathbf{Grpd}\to \mathbf{Cat}$, admits a right adjoint $(-)_\text{iso}$. It assigns to a category $\mathcal{C}$ the maximal groupoid contained in $\mathcal{C}$. The maximal groupoid $\mathcal{C}_\text{iso}$ has the same objects as $\mathcal{C}$, but the morphisms are all isomorphisms in $\mathcal{C}$. Similarly, the inclusion of small Kan complexes into small quasicategories, $\mathbf{Kan} \to \mathbf{QCat}$, admits a right adjoint denoted $(-)_\mathrm{equiv}$.\footnote{Blumberg--Gepner--Tabuada write $(-)_\text{iso}$ instead of $(-)_\mathrm{equiv}$ in \cite{BlumbergGepnerTabuadaI}.} For a quasicategory $X$, the simplicial set $X_\mathrm{equiv}$ is the maximal Kan complex contained in $X$ \cite[Theorem 4.19]{JoyalQuadern}.

The Kan complex $X_\mathrm{equiv}$ is 1-full in the quasicategory $X$. An $n$-simplex $x\co \Delta[n] \to X$ is in $X_\mathrm{equiv}$ if and only if the morphism $\sigma_{i,j}^*(x)$ is an equivalence in $X$ for all $0 \leq i < j \leq n$, where $\sigma_{i,j}\co [1] \to [n]$ is the relevant injection with image $\{i,j\}$,  \cite[Corollary 1.5]{JoyalQCatsAndKanComplexes} and \cite[Lemma 4.18]{JoyalQuadern}. In particular, the 1-simplices of $X_{\mathrm{equiv}}$ are the equivalences in $X$. Thus, $X_\mathrm{equiv}$ is the  subcomplex of $X$ 1-full on the morphisms in $X$ which are invertible in the homotopy category $hoX$.

\begin{lemma} \label{lem:iso_commutes_with_Nerve}
If $\mathcal{C}$ is a category, then $\left(N\mathcal{C}\right)_\mathrm{equiv}=N\left(\mathcal{C}_\text{\rm iso}\right)$.
\end{lemma}
\begin{proof}
Since $\left(N\mathcal{C}\right)_\mathrm{equiv}$ is the maximal Kan subcomplex of $N\mathcal{C}$, it contains the Kan subcomplex $N\left(\mathcal{C}_\text{\rm iso}\right)$.

A 1-simplex in $N\mathcal{C}$ is an equivalence in $N\mathcal{C}$ if and only if its image is an isomorphism in $\tau_1 N \mathcal{C} \cong \mathcal{C}$, so the 1-simplices of $\left(N\mathcal{C}\right)_\mathrm{equiv}$ and $N\left(\mathcal{C}_\text{\rm iso}\right)$ coincide.
If $x$ is an $n$-simplex of $\left(N\mathcal{C}\right)_\mathrm{equiv}$, then it is a path of $n$ morphisms in $\mathcal{C}$ in which the composites $\sigma_{i,j}^*(x)$ are equivalences in $N\mathcal{C}$ and thus isomorphisms in $\mathcal{C}$ for all $0 \leq i < j \leq n$. In particular, each of the $n$ morphisms is invertible in $\mathcal{C}$ and $x$ is in $N \left( \mathcal{C}_\text{iso} \right)$, so that $\left(N\mathcal{C}\right)_\mathrm{equiv}\subseteq N\left(\mathcal{C}_\text{\rm iso}\right)$.
\end{proof}

If $A$ is a simplicial set such that $\tau_1 A$ is a groupoid, then every map $A \to X$ factors through the inclusion $X_{\mathrm{equiv}} \subseteq X$ \cite[Proposition 4.21]{JoyalQuadern}.

The maximal Kan subcomplex $X_{\mathrm{equiv}}$ of the quasicategory $X$ is more closely related to groupoids in the following straightforward extension of \cite[Proposition 4.22]{JoyalQuadern}. Let $J[n]$ denote the nerve of the groupoid with objects $0,1, \dots, n$ and a unique isomorphism from any object to another.

\begin{proposition} \label{prop:extending_an_n-simplex_of_equivalences}
Let $X$ be a quasicategory. An $n$-simplex $x\co \Delta[n] \to X$ is in $X_\mathrm{equiv}$ if and only if it can be extended to a map $J[n] \to X$.
\end{proposition}
\begin{proof}
Suppose $x$ is in $X_\mathrm{equiv}$. Then we have the following diagram, in which the right vertical arrow is a fibration. $$\xymatrix{\Delta[n] \ar[r]^x \ar@{^{(}->}[d] & X_{\mathrm{equiv}} \ar[d] \\ J[n] \ar[r] & \ast}$$
The inclusion $\Delta[n] \hookrightarrow J[n]$ is a cofibration by definition, and a weak equivalence (both $\Delta[n]$ and $J[n]$ are contractible, as they are nerves of categories with a terminal object). Therefore the lifting problem
$\overline{x}\co J[n] \to X_\mathrm{equiv}$ can be solved.

Suppose now $x$ is in $X$ and an extension $\overline{x} \co J[n] \to X$ exists. The category $\tau_1 \left( J[n] \right)$ is isomorphic to a groupoid via the counit of the adjunction $\tau_1 \dashv N$, so $\overline{x} \co J[n] \to X$ factors through the inclusion $X_{\mathrm{equiv}} \subseteq X$.
\end{proof}

\begin{corollary} \label{cor:nat_transf_is_nat_equiv_iff_comps_equivs}
Let $X$ and $Y$ be quasicategories and $\alpha\colon X \times \Delta[1] \to Y$ a natural transformation. Then the following are equivalent.
\begin{enumerate}
\item \label{item(i)}
$\alpha$ is a natural equivalence.
\item \label{item(ii)}
$\alpha$ is invertible in the homotopy category $\tau_1\text{\rm Fun}(X,Y)=\tau_1(Y^X)$.
\item \label{item(iii)}
$\alpha$ extends to a map $X \times J[1] \to Y$.
\item \label{item(iv)}
Each component $\alpha_x:=\alpha(s_0(x), \text{Id}_{[1]})$ is an equivalence in $Y$.
\end{enumerate}
\end{corollary}
\begin{proof}
By definition, $\alpha$ is a natural equivalence if and only if it is an equivalence in $\text{Fun}(X,Y)$, which is the case if and only if $\alpha \colon \Delta[1] \to \text{Fun}(X,Y)$ extends to $\overline{\alpha} \colon J[1] \to \text{Fun}(X,Y)$. Thus, \ref{item(i)}, \ref{item(ii)}, and \ref{item(iii)} are equivalent.

The equivalence \ref{item(i)} $\Leftrightarrow$ \ref{item(iv)} is Joyal's \cite[Theorem~5.14]{JoyalQuadern}, which is quite involved to prove. We merely intuitively motivate the equivalence \ref{item(iii)} $\Leftrightarrow$ \ref{item(iv)} instead.
By the functoriality and naturality of \eqref{equ:Fun_adjunction}, $\alpha$ extends to $\overline{\alpha}$ if and only if the corresponding map $\alpha^X\colon X \to \text{Fun}(\Delta[1],Y)$ admits a lift $\overline{\alpha^X}\colon X \to \text{Fun}(J[1],Y)$ along the restriction map $i^*\colon \text{Fun}(J[1],Y) \to \text{Fun}(\Delta[1],Y)$. Assuming the restriction $i^*$ is $0$-full\footnote{One could expect that $i^*\colon \text{Fun}(J[1],Y) \to \text{Fun}(\Delta[1],Y)$ is 0-full because we know $Y_{\mathrm{equiv}} \hookrightarrow Y$ is 1-full and Proposition~\ref{prop:extending_an_n-simplex_of_equivalences}. Notice that the statement ``$i^*$ is 0-full'' is a consequence of the equivalences \ref{item(i)} $\Leftrightarrow$ \ref{item(iv)} and \ref{item(i)} $\Leftrightarrow$ \ref{item(iii)}, for if $\gamma\colon \Delta[1] \times \Delta[n] \to Y$ in the codomain has all of its vertices (``components'') $\gamma_i \colon \Delta[1] \times \{i\} \to Y$ as equivalences, then $\gamma$ is a natural equivalence, and will admit an extension to $J[1] \times \Delta[n] \to Y$ by the equivalences \ref{item(iv)} $\Leftrightarrow$ \ref{item(i)} $\Leftrightarrow$ \ref{item(iii)}.}, the existence of a lift along $i^*$ is the same as the existence of a lift on the level of $0$-simplices, which is the same as requiring each component $\alpha_x$ to be an equivalence in $Y$. Thus, assuming the intuition that $i^*$ is 0-full, we have \ref{item(iii)} $\Leftrightarrow$ \ref{item(iv)}.
\end{proof}

\subsection{Adjunctions and Equivalences between Quasicategories} \label{subsec:Adjunctions_and_Equivalences_between_Quasicategories}

We also need the notions of adjoint functors and equivalences between quasicategories. These are most easily phrased in terms of Joyal's 2-category of simplicial sets as in \cite{JoyalQuadern}, denoted by $\mathbf{SSet}^{\tau_1}$. Objects are simplicial sets, and morphisms are maps of simplicial sets. The 2-cells are given by the definition of hom categories $\mathbf{SSet}^{\tau_1}(A,B):=\tau_1(B^A)$. If $X$ and $Y$ are quasicategories, then a 2-cell between functors $j,k\colon X \to Y$ is the homotopy class of a functor $\alpha\co X \times \Delta[1] \to Y$ with $\alpha_0=j$ and $\alpha_1=k$.

\begin{proposition}[Proposition 1.27 of \cite{JoyalQuadern}] \label{prop:tau1_is_a_2-functor}
The functor $$\xymatrix{\tau_1 \colon \mathbf{SSet}^{\tau_1} \ar[r] & \mathbf{Cat}}$$ is a 2-functor.
Hence it takes an equivalence in $\mathbf{SSet}^{\tau_1}$ to an equivalence of categories, an adjunction to an adjunction, and a left respectively right adjoint to a left respectively right adjoint.
\end{proposition}

An {\it adjunction between quasicategories $X$ and $Y$} is an adjunction between $X$ and $Y$ in the 2-category $\mathbf{SSet}^{\tau_1}$. More explicitly, an adjunction consists of functors $f\colon X\to Y$ and $g \colon Y \to X$, functors $\eta\colon X \times \Delta[1] \to X$ and $\varepsilon\colon Y \times \Delta[1] \to Y$ with $\eta_0=\text{Id}_X$, $\eta_1=g \circ f$, $\varepsilon_0=f \circ g$, and $\varepsilon_1=\text{Id}_Y$, such that the triangle identities hold in the 2-category $\mathbf{SSet}^{\tau_1}$:
\begin{equation} \label{equ:qcat_triangle_identities}
(g\ast[\varepsilon]) \odot ([\eta] \ast g)=\text{Id}_g \;\;\;\;\; ([\varepsilon] \ast f) \odot (f \ast [\eta])=\text{Id}_f.
\end{equation}
Here $[\eta]$ and $[\varepsilon]$ indicate the morphisms in the categories $\tau_1(X^X)$ and $\tau_1(Y^Y)$ associated to the morphism $\eta$ and $\varepsilon$ in $\text{Fun}(X,X)$ and $\text{Fun}(Y,Y)$. An adjunction between quasicategories induces an ordinary adjunction between their homotopy categories. The nerve of an adjunction between categories is an adjunction between quasicategories.

Like the notion of adjunction between quasicategories, the notion of equivalence between quasicategories may also be phrased in terms of the 2-category $\mathbf{SSet}^{\tau_1}$. A functor $f \colon X \to Y$ between quasicategories is an {\it equivalence of quasicategories} if it is an equivalence in the 2-category $\mathbf{SSet}^{\tau_1}$. That is to say, if there exist a functor $g \colon Y \to X$, a morphism $\eta\co \text{Id}_X \to g \circ f$ in $\text{Fun}(X,X)$, and a morphism $\varepsilon \co f \circ g \to \text{Id}_Y$ in $\text{Fun}(Y,Y)$ which induce invertible 2-cells $[\eta]$ and $[\varepsilon]$ in $\mathbf{SSet}^{\tau_1}$. In other words, $\eta$ and $\varepsilon$ are natural equivalences. Every equivalence between quasicategories is a weak homotopy equivalence between the underlying simplicial sets. A simplicial set map $N\calc \to N\cald$ is an equivalence of quasicategories if and only if it is an equivalence of categories $\calc \to \cald$.

The quasicategorical notions of adjunction and equivalence between nerves of categories coincide with the respective categorical notions because for categories $\calc$ and $\cald$ the canonical simplicial set map $N(\cald^\calc)\to (N\cald)^{(N\calc)}$ is an isomorphism by \cite[Proposition~B.0.16]{JoyalQuadern}.

A functor $f \colon X \to Y$ between quasicategories is said to be {\it fully faithful} if the associated map $X(a, b) \to Y(fa, fb)$ is a weak homotopy equivalence for all objects $a$ and $b$ of $X$, recall $X(a,b)$ is defined via the pullback in equation \eqref{equ:mapping_space_of_a_quasicategory}. The functor $f$ is said to be {\it essentially surjective} if $\tau_1 f \colon \tau_1 X \to \tau_1 Y$ is essentially surjective.

\begin{proposition} \label{prop:g_fullyfaithful_iff_counit_equivalence:quasicategories}
Suppose $f\colon X \to Y$ and $g\co Y \to X$ are adjoint functors between quasicategories with unit  $\eta\co 1_{X} \Rightarrow gf$ and counit $\varepsilon\co fg \Rightarrow 1_Y$. Then the right adjoint $g$ is fully faithful if and only if every component of the counit $\varepsilon$ is an equivalence. Dually, the left adjoint $f$ is fully faithful if and only if every component of the unit $\eta$ is an equivalence.
\end{proposition}

\begin{proposition} \label{prop:qcat_equiv_iff_fullyfaithful_ess_surj}
A functor between quasicategories is an equivalence in $\mathbf{SSet}^{\tau_1}$ if and only if it is fully faithful and essentially surjective.
\end{proposition}

Every equivalence of quasicategories $f \colon X \to Y$ induces an ordinary equivalence of categories $\tau_1 f \colon \tau_1X \to \tau_1Y$. Consequently, if $\sigma$ is a 1-morphism in $X$, then $\sigma$ is an equivalence in $X$ if and only if $f(\sigma)$ is an equivalence in $Y$. That is, every equivalence of quasicategories both preserves and reflects equivalences {\it in} quasicategories.

It is well known that every equivalence in a 2-category is part of an adjoint equivalence. Thus, if $f$ is an equivalence of quasicategories, then $g$, $\eta$, and $\varepsilon$ may be chosen to make $f$ a left or right adjoint.

\subsection{The Joyal Model Structure on $\mathbf{SSet}$, its Relationship to the Kan Structure, and Constructions of Joyal Fibrations and Joyal Equivalences via Restriction}

Joyal proved in \cite{JoyalQuadern} his construction of a model structure on $\mathbf{SSet}$ in which the fibrant objects are precisely the quasicategories. Rather than going into details, we only describe Joyal fibrations and Joyal equivalences between quasicategories, compare with the Kan structure on simplicial sets, and state the main propositions on how to build Joyal fibrations and Joyal equivalences from restriction. The goal here is to merely assemble the facts needed for the present project.

A {\it Joyal cofibration} is a monomorphism in $\mathbf{SSet}$, that is, a map of simplicial sets which is injective in each degree. A {\it Joyal equivalence between quasicategories} is an equivalence between quasicategories in the 2-category $\mathbf{SSet}^{\tau_1}$, as described in Section~\ref{subsec:Adjunctions_and_Equivalences_between_Quasicategories}. A {\it Joyal fibration between quasicategories} is a map between quasicategories which has the right lifting property with respect to each $\Lambda^k[n] \hookrightarrow \Delta[n]$ for $n\geq 2$ and $0 < k < n$, and with respect to the inclusion $\{0\} \hookrightarrow \{0 \cong 1\}$.

Joyal cofibrations are the same as the cofibrations in the usual Kan model structure on simplicial sets, so we merely refer to them as {\it cofibrations}. The trivial fibrations in the Joyal model structure and in the Kan model structure are the same, since they are exactly the maps with the right lifting property with respect to the same cofibrations, so we use the term {\it trivial fibrations} to refer to the trivial fibrations of both the Joyal model structure and the Kan model structure, as Joyal elegantly does in \cite{JoyalQuadern}. Every Joyal equivalence is a Kan weak equivalence, but not vice versa. Every Kan fibration is a Joyal fibration, but not vice versa. See \cite[Proposition~6.15 and Corollary~6.16]{JoyalQuadern}.

\begin{proposition}[Corollary 4.28 of \cite{JoyalQuadern}] \label{prop:Joyal_Fibrations_btw_Kan_cxs_are_Kan}
Let $X$ and $Y$ be Kan complexes. Then every Joyal fibration $X \to Y$ is a Kan fibration.
\end{proposition}

\begin{proposition}[Construction of Joyal Equivalences and Joyal Fibrations \\ via Restriction and Postcomposition] \leavevmode \label{prop:restriction_morphisms}
\begin{enumerate}
\item \label{prop:restriction_morphisms:i}
Let $A$ and $B$ simplicial sets and $u \co A \to B$ a map of simplicial sets. Then $u$ is a Joyal equivalence (=weak categorical equivalence) if and only if for every quasicategory $X$ the restriction map $X^u\co X^B \to X^A$ is a Joyal equivalence between quasicategories. \cite[Proposition~2.27]{JoyalQuadern}
\item \label{prop:restriction_morphisms:ii}
Let $X$ be a quasicategory, $A$ and $B$ simplicial sets, and $u \co A \to B$ a monomorphism. Then the restriction map $X^u\co X^B \to X^A$ is a Joyal fibration between quasicategories. See \cite[Theorem~5.13]{JoyalQuadern} and \cite[Theorem~6.6]{JoyalQuadern}.
\item \label{prop:restriction_morphisms:iii}
Let $f\co X \to Y$ be a Joyal fibration between quasicategories, and $A$ a simplicial set. Then $f^A\co X^A \to Y^A$ is a Joyal fibration between quasicategories. See \cite[Theorem~5.13]{JoyalQuadern} and \cite[Theorem~6.6]{JoyalQuadern}.
\item
Let $u \co A \to B$ be a monomorphism that is also a Joyal equivalence. Then the restriction map $X^u\co X^B \to X^A$ is a trivial fibration.
\end{enumerate}
\end{proposition}

\subsection{Commutative Squares and Pushouts in a Quasicategory}

Even though a quasicategory $X$ is not equipped with a choice of composition, we may still speak of commutative squares in $X$, which are actually homotopy commutative squares in $X$. A {\it commutative square} in a quasicategory $X$ is a functor $\Delta[1] \times \Delta[1] \to X$. Since $\tau_1$ preserves products, a commutative square in $X$ gives rise to a truly commutative square in the homotopy category $\tau_1X$. The converse is also true.

\begin{lemma} \label{lem:commutative_squares_in_homotopy_category=commutative_squares_in_quasicategory}
Let $X$ be a quasicategory. Every commutative square $[g][f]=[k][j]$ in its homotopy category $\tau_1X$ comes from a commutative square in $X$ with boundary $g$, $f$, $k$, and $j$.
\end{lemma}
\begin{proof}
Since $X$ is a quasicategory, every morphism in $\tau_1X$ is a homotopy class $[f]$ of a morphism $f$ in $X$. Moreover, we have $[g][f]=[h]$ in $\tau_1X$ if and only if the boundary $\partial \Delta[2] \to X$ determined by $g$, $f$, and $h$ extends to a functor $\Delta[2] \to X$, see \cite[pages 212-213]{JoyalQuadern} which reference Boardman--Vogt \cite{BoardmanVogt}. Thus, if we have a commutative square $[g][f]=[k][j]$ in $\tau_1X$, there is an $[h]$ in $\tau_1X$ equal to both $[g][f]$ and $[k][j]$, and the corresponding two maps $\partial \Delta[2] \to X$ can be filled. Since $\Delta[1] \times \Delta[1]$ is 2-skeletal, this defines a functor $\Delta[1] \times \Delta[1] \to X$, which in turn induces the commutative square $[g][f]=[k][j]$.
\end{proof}

For each natural transformation $\alpha\co X \times \Delta[1] \to Y$ and each $f \in X_1$, the usual naturality square is a commutative square in $Y$.

We next recall the notion of pushout in a quasicategory $X$. Instead of recalling the general definition of colimit in a quasicategory from \cite[Definition~4.5]{JoyalQCatsAndKanComplexes} and \cite[page 159]{JoyalQuadern}, we work out pushouts explicitly. For more on pushouts, see \cite[Section~4.2.2]{LurieHigherToposTheory}.

An object $i$ of a quasicategory $X$ is {\it initial} if for any object $x$ in $X$ the map $X(i,x) \to \text{pt}$ is a weak homotopy equivalence. If $i$ is initial in the quasicategory $X$, then it is initial in the category $\tau_1X$ in the usual sense because
$$\text{pt} = \pi_0 X(i,x) = \pi_0 \left( (\mathfrak{C}X)(i,x) \right) = (\pi_0\mathfrak{C}X) (i, x) = (\tau_1X)(i,x),$$
since $X(i,x)$ and $(\mathfrak{C}X)(i,x)$ are connected by a natural zig-zag of weak homotopy equivalences \cite[Corollary~5.3]{DuggerSpivakMapping} and $\pi_0 \mathfrak{C} =\tau_1$. Any two initial objects $i$ and $i'$ of $X$ are equivalent: the simplicial sets $X(i,i')$, $X(i',i)$, $X(i,i)$, $X(i',i')$ are all weakly equivalent to a point, so every $f \in X(i,i')$ and $g \in X(i',i)$ satisfy $[g][f]=[\text{Id}_{i}]$ and $[f][g]=[\text{Id}_{i'}]$ as $\pi_0X(x,y)=\left(\tau_1X\right)(x,y)$.
Moreover, the 0-full sub quasicategory of $X$ on the initial objects is either a contractible Kan complex or the empty simplicial set by \cite[page 159]{JoyalQuadern} or \cite[Proposition 1.2.12.9]{LurieHigherToposTheory}.

If $C$ is a category and $i$ is an object of $C$, then $i$ is initial in the category $C$ if and only if $i$ is initial in the quasicategory $NC$, as $(NC)(i,x)$ is $C(i,x)$ viewed as a discrete simplicial set.

Let $\cali$ be the nerve of the category $b \leftarrow a \rightarrow c $ and $F\colon \cali \to X$ a functor. The quasicategory of commutative squares in $X$ which restrict to $F$ is the following pullback $X_{F/}$ in $\mathbf{SSet}$.
\begin{equation} \label{equ:X_over_F}
\begin{array}{c}
\xymatrix{X_{F/} \ar[r] \ar[d] \ar@{}[dr]|{\text{pullback}} & X^{\Delta[1] \times \Delta[1]} \ar[d]^{\text{incl}^\ast=\text{ res}}  \\ \ast \ar[r]_-{F} & X^\cali}
\end{array}
\end{equation}
The slice category used by Lurie is slightly different, but equivalent to this one $X_{F/}$ of Joyal, see \cite[Proposition~4.2.1.5]{LurieHigherToposTheory} and \cite[Proposition~2.4.13]{RiehlVerity}.
A commutative square in $X$ is a {\it pushout square of the diagram $F\colon \cali \to X$} if it is an initial object in the quasicategory $X_{F/}$. Since initial objects are unique up to equivalence, any two pushout squares of $F$ are equivalent. In particular, the lower right corner objects of two pushout squares of $F$ are equivalent. Moreover, the 0-full sub quasicategory of $X_{F/}$ on the pushouts (initial objects) is either a contractible Kan complex or the empty simplicial set.

In general, pushouts in $X$ are {\it not} the same as ordinary pushouts in $\tau_1X$. However pushouts in the nerve of a category $NC$ are the same as pushouts in the category $C$. This is because: pushouts in $C$ are defined precisely as initial objects in the categorical analogue of the pullback in \eqref{equ:X_over_F}, nerve preserves pullbacks, nerve is fully faithful,  $NC^{ND}\cong N(C^D)$ for any category $D$, in particular for $[1] \times [1]$ and $b \leftarrow a \rightarrow c $, and initial objects in a category are the same as initial objects in its nerve.

Though we do not need this, we remark that if $\calc$ is a category enriched in Kan complexes, then homotopy pushouts in $\calc$ correspond to pushouts in $N^\text{simp}\calc$, see \cite[Theorem 4.2.4.1, page 258]{LurieHigherToposTheory}.

In an ordinary category, the pushout of any isomorphism along any morphism exists and is also an isomorphism. We have the following analogue in a quasicategory.
\begin{lemma} \label{lem:pushout_along_equivalence}
A pushout of any equivalence along any morphism in a quasicategory $X$ exists, and is an equivalence in $X$.
\end{lemma}
\begin{proof}
(Sketch)
Suppose $j\co A \to C$ is an equivalence in $X$ and $f\co A \to B$ is any morphism. Let $fj^{-1}$ denote any filler for the horn determined by $f$ and any pseudo inverse to the equivalence $j$. Then the outer square pictured below is a pushout.

Suppose the inner square below is a pushout square in $X$.
$$\xymatrix@C=3pc@R=3pc{
      A \ar[r]^-{f} \ar[d]_j \ar@{}[dr]|{\text{pushout}}  &  B \ar[d]^{j'} \ar@/^1pc/[ddr]^{\text{Id}_B}  &  \\
      C \ar[r]_{f'} \ar@/_1pc/[drr]_{fj^{-1}}  &  P \ar@{-->}[dr]^{m}  &  \\
      &  &  B }$$
Since pushouts are unique up to equivalence, there is an equivalence $m$ which makes the relevant triangles commute in the homotopy category of $X$. By the 3-for-2 property of equivalences, $j'$ is then also an equivalence.
\end{proof}

\subsection{A Natural Pushout Functor Along Cofibrations in Waldhausen Quasicategories}
\label{subsec:natural_pushout_along_cofibrations}

In Footnote \ref{footnote:functorial_choice_of_pushout} concerning functorial pushouts along cofibrations in a category in the proof of Lemma~\ref{lem:HardLemmaMadeEasy}, we explained how to make a functorial choice of pushouts in a category preserving identity morphisms so that the object-wise formation of the pushout in \eqref{equ:rho_n_as_pushout} is compatible with face and degeneracy maps. In the quasicategorical version of Lemma~\ref{lem:HardLemmaMadeEasy}, namely Lemma~\ref{lem:HardLemmaMadeEasy_quasicategorical}, we need further justification for this compatibility for two reasons: 1) in a quasicategory, morphisms such as $B_1 \to B_2 \to B_3$ do not uniquely determine composites $B_1 \to B_3$, so drawings such as \eqref{equ:lambda_rows}, \eqref{equ:mu_rows}, and \eqref{equ:rho_grid} are ambiguous and constructions must also be done on higher simplices, and 2) in quasicategory theory, the universal property of pushouts does not uniquely determine a morphism, so connecting morphisms and higher simplices between objects and chosen quotients in \eqref{equ:rho_n_as_pushout} are not uniquely determined by objects and morphisms.

In this section,\footnote{We thank Emily Riehl and Dominic Verity for sketching Proposition~\ref{prop:RiehlVerity_Corollary_For_Pushouts} in an email and for explaining their Theorem~1.1 of \cite{RiehlVerityCompleteness}, as recalled here in Theorem~\ref{thm:Riehl--Verity_ExtraOrdinary_Naturality} and its subsequent discussion.} we explain how to construct for a quasicategory $\calc$ a functorial choice of pushouts along selected maps called ``cofibrations'' in $\calc$ in a way that preserves identity morphisms. That is, we show how to functorially extend each diagram $C \leftarrowtail A \to B$ in $\calc$ to a pushout diagram in $\calc$ in a way that also does the following.

\begin{equation} \label{equ:identity_pushout}
\begin{array}{c}
\xymatrix{A \ar[r] \ar@{>->}[d]_{\text{Id}_{A}} & B \\ A  & }
\end{array}
\mapsto
\begin{array}{c}
\xymatrix{A \ar[r] \ar@{>->}[d]_{\text{Id}_{A}} \ar@{}[dr]|{\text{p.o.}} & B \ar@{>->}[d]^{\text{Id}_{B}} \\ A \ar[r] & B}
\end{array}
\hspace{.25in}\text{and}\hspace{.25in}
\begin{array}{c}
\xymatrix{A \ar[r]^{\text{Id}_{A}} \ar@{>->}[d] & A \\ C & }
\end{array}
\mapsto
\begin{array}{c}
\xymatrix{A \ar[r]^{\text{Id}_{A}} \ar@{>->}[d] \ar@{}[dr]|{\text{p.o.}} & A\phantom{.} \ar@{>->}[d] \\ C \ar[r]_{\text{Id}_{C}} & C.}
\end{array}
\end{equation}
We then prove that this induces natural pushout functors in diagram categories $\calc^K$, all so that the quasicategorical version of pushout \eqref{equ:rho_n_as_pushout} is compatible with face and degeneracy maps. Our main goal is Corollary~\ref{cor:main_corollary_about_pushouts}. Of course, much of the discussion here applies to more general colimits in the situation with all maps cofibrations, but we restrict our attention to pushouts along cofibrations because of the application at hand. The main external input is Proposition~\ref{prop:RiehlVerity_Corollary_For_Pushouts}, which is an analogue of Riehl--Verity's \cite[Corollary~5.2.20]{RiehlVerity}.

\begin{notation} \label{not:cofibrations_etc}
In this section, we work with adjunctions in the 2-category of simplicial sets $\mathbf{SSet}^{\tau_1}$ as recalled in Section~\ref{subsec:Adjunctions_and_Equivalences_between_Quasicategories} (usually we have adjunctions in its sub 2-category of quasicategories). In this section, $\calc$ is a quasicategory equipped with a distinguished class of morphisms called cofibrations and denoted with feathered arrows $\rightarrowtail$. We do not assume any Waldhausen structure, instead our only assumption on the class of cofibrations in $\calc$ is that every identity morphism in $\calc$ (=degenerate 1-simplex in $\calc$) is a cofibration. Let $\calj$ be $\Delta[1] \times \Delta[1]$, the nerve of the free standing commutative square $[1] \times [1]$. Let $\cali$ be the nerve of the subcategory of $[1] \times [1]$ determined by the top horizontal map and the left vertical map. Let $\calc^{\cali}_{co}$ be the 0-full sub quasicategory of $\calc^{\cali}$ on the diagrams of the form $C_3 \leftarrowtail C_1 \to C_2$, that is, on diagrams with the left vertical map a cofibration. Similarly, we denote by $\calc^{\calj}_{co}$ the 0-full sub quasicategory of $\calc^{\calj}$ on the commutative squares in which the left vertical map is a cofibration. The quasicategory $\calc^{\calj}_{co, po}$ is the 0-full sub quasicategory of $\calc^{\calj}_{co}$ on pushout squares with left vertical map a cofibration. The inclusion $\cali \hookrightarrow \calj$ induces restriction maps $\text{res}\co \calc^{\calj} \to \calc^{\cali}$ and  $\text{res}_{co} \co \calc^{\calj}_{co} \to \calc^{\cali}_{co}$. Finally, let $\calc^\cali_\text{Id}$ be the sub simplicial set of $\calc^\cali_{co}$ that is the union of the images of
$$\xymatrix@R=.25pc{\calc^{\Delta[1]} \ar[r] & \calc^\cali_{co} \\ f \ar@{|->}[r] & \big( \overset{\text{\rm Id}}{\leftarrowtail} \, \overset{f}{\rightarrow} \big)} \hspace{.5in} \text{and} \hspace{.5in} \xymatrix@R=.25pc{\calc^{\Delta[1]}_{co} \ar[r] & \calc^\cali_{co} \\ g \ar@{|->}[r] & \big( \overset{g}{\leftarrowtail} \, \overset{\text{\rm Id}}{\rightarrow} \big)}.$$
Here $\calc^{\Delta[1]}_{co}$ is the sub quasicategory of $\calc^{\Delta[1]}$ that is 0-full on the cofibrations.
\end{notation}

A first observation about these newly notated quasicategories is that both restrictions $\text{res}\co \calc^{\calj} \to \calc^{\cali}$ and  $\text{res}_{co} \co \calc^{\calj}_{co} \to \calc^{\cali}_{co}$ are fibrations between quasicategories in Joyal's model structure on $\mathbf{SSet}$. Recall that a {\it fibration between quasicategories} in Joyal's model structure is a map between quasicategories which has the right lifting property with respect to each $\Lambda^k[n] \hookrightarrow \Delta[n]$ for $n\geq 2$ and $0 < k < n$, and with respect to the inclusion $\{0\} \hookrightarrow \{0 \cong 1\}$. The restriction $\text{res}\co \calc^\calj \to \calc^\cali $ is such a fibration by \cite[Theorem~6.6]{JoyalQuadern} recalled in Proposition~\ref{prop:restriction_morphisms}~\ref{prop:restriction_morphisms:ii}, see also \cite[Recall 2.2.8]{RiehlVerity}. The restriction $\text{res}_{co} \co \calc^{\calj}_{co} \to \calc^{\cali}_{co}$ is also such a fibration as it is the pullback along $\calc^{\cali}_{co}\hookrightarrow \calc^{\cali}$ of the fibration $\text{res}\co\calc^\calj \to \calc^\cali$.

A consequence of \cite[Corollary~5.2.20]{RiehlVerity} is that a quasicategory $\calc$ admits limits of diagrams of shape $X$ if and only if the restriction $\calc^{\mathbf{1} \star X} \to \calc^X$ admits a right adjoint right inverse, i.e., a right adjoint for which the counit is an identity 2-cell in $\mathbf{SSet}^{\tau_1}$. We are interested in the following analogue for pushouts along cofibrations instead of limits.

\begin{proposition}[Riehl-Verity Analogue] \label{prop:RiehlVerity_Corollary_For_Pushouts}
Let $\calc$ be a quasicategory equipped with a distinguished class of morphisms called cofibrations, which includes all the identity morphisms of $\calc$, i.e. includes the degenerate 1-simplices of $\calc$. We use Notation~\ref{not:cofibrations_etc}. Then $\calc$ admits pushouts along cofibrations if and only if the restriction $\text{\rm res}_{co} \co \calc^{\calj}_{co} \to \calc^{\cali}_{co}$ admits a left adjoint right inverse $F\co \calc^{\cali}_{co} \to \calc^{\calj}_{co}$, that is, if and only if this restriction admits a left adjoint $F$ for which the unit $\text{\rm Id}_{\calc^\cali_{co}} \Rightarrow \text{\rm res}_{co} \circ F$ is an identity 2-cell in $\mathbf{SSet}^{\tau_1}$. Moreover, if $\calc$ admits pushouts along cofibrations then the following is also true.
\begin{enumerate}
\item \label{RVi}
The pushout squares with left vertical map a cofibration are exactly the objects of $\calc^{\calj}_{co}$ that are equivalent to objects in the image of any left adjoint right inverse $F$.
\item \label{RVii}
The counit $\varepsilon\co F\circ \text{\rm res}_{co} \Rightarrow \text{Id}_{\calc^{\calj}_{co}}$ is the identity 2-cell except for the lower right corner. More precisely, $\text{\rm res}_{co}\,\varepsilon$ is the identity 2-cell on $\text{\rm res}_{co}$, and consequently any representative natural transformation for the counit has components homotopic to identity morphisms in all corners except the lower right one.
\item \label{RViii}
Let $F_{po} \co \calc^{\cali}_{co} \to \calc^{\calj}_{co,po}$ be $F$ with restricted codomain quasicategory, and $\text{\rm res}_{po} \co \calc^{\calj}_{co,po} \to \calc^{\cali}_{co}$ the restriction of $\text{\rm res}_{co}$ to the indicated sub quasicategory. Similarly, let $\varepsilon_{po}\co F_{po} \text{\rm res}_{po} \Rightarrow \text{\rm Id}_{\calc^\calj_{co,po}}$ be the restriction of $\varepsilon$ to $\calc^\calj_{co,po} \subseteq \calc^\calj_{co}$. Then $\varepsilon_{po}$ is an iso 2-cell in $\mathbf{SSet}^{\tau_1}$.
\item \label{RViv}
The sub adjunction $F_{po} \dashv \text{\rm res}_{po}$ of $F \dashv \text{\rm res}_{co}$ is an adjoint equivalence. Its unit is an identity 2-cell, and its counit $\varepsilon_{po}$ is the restriction of $\varepsilon$.
\end{enumerate}
\end{proposition}
\begin{proof}
We do not prove the main ``if and only if'' statement, nor \ref{RVi}. \\
\ref{RVii} The triangle identity $(\text{res}_{co}\,\varepsilon) \circ (\eta\,\text{res}_{co})=\text{Id}_{\text{res}_{co}}$ holds and the unit $\eta$ is an identity 2-cell, so $\text{res}_{co}\,\varepsilon$ must also be an identity 2-cell.\\
\ref{RViii} Let $e\co \calc^\calj_{co} \times \Delta[1] \to \calc^\calj_{co}$ be a representative natural transformation for the counit $\varepsilon$. Three components of $e$ are homotopic to identity morphisms by \ref{RVii}, so are therefore equivalences. When $e$ is evaluated at a pushout square, the lower right corner must also be an equivalence by the universal property of pushouts. Thus the restriction of $e$ to $\calc^\calj_{co,po}$ is a natural equivalence, and $\varepsilon_{po}$ is an iso 2-cell. \\
\ref{RViv} This follows from \ref{RViii} and the fact that the unit is an identity 2-cell.
\end{proof}

\begin{proposition} \label{prop:functorial_choice_of_pushouts_pres_id}
Let $\calc$ be a quasicategory equipped with a distinguished class of morphisms called cofibrations and denoted with feathered arrows $\rightarrowtail$.  We use Notation~\ref{not:cofibrations_etc}. Suppose that all pushouts along cofibrations exist in $\calc$ and that all identity morphisms in $\calc$ are cofibrations. We do not assume any Waldhausen axioms. Then there exists a map $\calc^{\cali}_{co} \to \calc^{\calj}_{co,po}$ which is a functorial choice of pushouts along cofibrations that preserves identity morphisms. More precisely, there is a solution to the following lifting problem.
\begin{equation} \label{equ:lifting_problem}
\begin{array}{c}
\xymatrix@C=2.5pc{\calc^\cali_{\underset{\phantom{.}}{\text{\rm Id}}} \ar[r]^-{\text{\rm id ext}} \ar@{^{(}->}[d] & \calc^{\calj}_{co,po} \ar[d]^{\text{\rm res}_{po}} \\
\calc^\cali_{co} \ar[r]_= \ar@{-->}[ur] & \calc^\cali_{co}}
\end{array}
\end{equation}
The top horizontal map ``{\rm id ext}'' is the extensions indicated in \eqref{equ:identity_pushout}.
\end{proposition}
\begin{proof}
The left vertical map in \eqref{equ:lifting_problem} is a monomorphism, so we can solve the lifting property as a consequence of the Joyal model structure on $\mathbf{SSet}$ if we can show that the right vertical map is a trivial fibration between quasicategories.

We first claim that the right vertical map in \eqref{equ:lifting_problem} is a Joyal fibration between quasicategories. Consider any diagram for $0 < k < n$ as on the left below (the right vertical map of \eqref{equ:lifting_problem}, which we want to show is a Joyal fibration, is now in the middle).
 \begin{equation} \label{equ:showing_horn lifts}
\begin{array}{c}
\xymatrix@C=2.5pc{\Lambda^k[n] \ar[r]  \ar@{^{(}->}[d] & \calc^{\calj}_{co,po} \ar[d]^(.6){\text{res}_{po}} \ar@{^{(}->}[r]^-{\text{0-full}} & \calc^{\calj}_{co} \ar[d]^{\text{res}_{co}} \\
\Delta[n] \ar[r] \ar@{-->}[urr] & \calc^\cali_{co} \ar[r]_= & \calc^\cali_{co} }
\end{array}
\end{equation}
As we already remarked, the right vertical map of \eqref{equ:showing_horn lifts} is a Joyal fibration, so the dashed lift exists. Since the vertices of $\Lambda^k[n]$ and $\Delta[n]$ are the same when $0 < k < n$, the commutativity of the top triangle shows that the dashed lift sends vertices of $\Delta[n]$ to $\calc^{\calj}_{co,po}$. But since, $\calc^{\calj}_{co,po} \hookrightarrow \calc^{\calj}_{co}$ is a 0-full inclusion, this implies that the dashed lift maps every simplex of $\Delta[n]$ to $\calc^{\calj}_{co,po}$ as well, and we have solved the lifting problem of the left square.

For solving lifting problems with the inclusion $\{0\} \hookrightarrow \{0 \cong 1\}$ on the left and the right vertical map of \eqref{equ:lifting_problem} on the right, suppose we have a natural equivalence of diagrams of the form $C_3 \leftarrowtail C_1 \to C_2$, where the first diagram extends to a pushout square. Then we find a lift of this equivalence to an equivalence in $\calc^\calj_{co}$ with source the pushout square, as the restriction $\text{res}_{co} \co \calc^\calj_{co} \to \calc^\cali_{co}$ is a Joyal fibration. But any commutative square equivalent to a pushout is also a pushout, hence the lifted equivalence is in $\calc^\calj_{co,po}$, and we have solved the lifting problem with the inclusion $\{0\} \hookrightarrow \{0 \cong 1\}$, so the right vertical map of \eqref{equ:lifting_problem} is a Joyal fibration between quasicategories.

The right vertical map of \eqref{equ:lifting_problem} is an equivalence by Proposition~\ref{prop:RiehlVerity_Corollary_For_Pushouts}~\ref{RViv}.

Finally, the right vertical map of \eqref{equ:lifting_problem} is a trivial fibration, the lifting problem of \eqref{equ:lifting_problem} is solved, and we obtain a pushout functor along cofibrations that preserves identities.
\end{proof}

\begin{proposition} \label{prop:any_functorial_po_extents_to_leftadjoint_rightinverse}
Suppose $\calc$ admits pushouts along cofibrations. If a map $F_{po}' \co \calc^{\cali}_{co} \to \calc^{\calj}_{co,po}$ is a right inverse to $\text{\rm res}_{po}\co \calc^{\calj}_{co,po} \to \calc^{\cali}_{co}$
\begin{equation} \label{equ:F_po_prime_right_inverse}
\text{\rm res}_{po} \circ F_{po}'=\text{\rm Id}_{\calc^\cali_{co}}\,,
\end{equation}
for instance the map in Proposition~\ref{prop:functorial_choice_of_pushouts_pres_id} is such a right inverse, then the composite $F'$
$$\xymatrix{\calc^{\cali}_{co} \ar[r]^{F_{po}'} \ar@/_1pc/[rr]_{F'} & \calc^{\calj}_{co,po} \ar[r]^{\text{\rm incl}} & \calc^{\calj}_{co}}$$
is a left adjoint right inverse to $\text{\rm res}_{co}$, like $F$ in Proposition~\ref{prop:RiehlVerity_Corollary_For_Pushouts}.
\end{proposition}
\begin{proof}
Recall from 2-category theory that we can replace a left adjoint by an isomorphic 1-cell to obtain another adjunction if we appropriately alter the unit and counit. Namely, if $F \dashv G$ with unit $\eta$ and counit $\varepsilon$ is an adjunction in a 2-category, and $\alpha\co F \Rightarrow F'$ is an iso 2-cell, then $F' \dashv G$ is an adjunction with unit $\eta':=(G \alpha) \odot \eta$ and counit $\varepsilon'=\varepsilon \odot (\alpha^{-1} G)$. The verification of the triangle identities is an exercise in 2-categorical pasting diagrams.

We apply this replacement result to the adjunction $F \dashv \text{res}_{co}$ in Proposition~\ref{prop:RiehlVerity_Corollary_For_Pushouts} with unit $\eta$ an identity, but first we have to construct $\alpha$. Consider the sub adjunction $F_{po} \dashv \text{\rm res}_{po}$ as in \ref{RViv}. Then composing \eqref{equ:F_po_prime_right_inverse} with $F_{po}$, and horizontally composing $\varepsilon_{po}$ with the identity 2-cell on $F_{po}'$, we obtain an iso 2-cell $\alpha:=\varepsilon_{po} \ast F_{po}'$.
$$\xymatrix@C=3.5pc{F_{po}=F_{po} \circ \text{res}_{po} \circ F_{po}' \ar@{=>}[r]^-{\varepsilon_{po} \ast F_{po}' } & F_{po}' }$$
This iso 2-cell is also an iso 2-cell $\alpha\co F \Rightarrow F'$. Since $\text{res}_{co}\;\alpha$ is an identity 2-cell by Theorem~\ref{prop:RiehlVerity_Corollary_For_Pushouts}~\ref{RVii}, the new unit $\eta':=(\text{res}_{co}\;\alpha) \odot \eta$ is also an identity 2-cell, so $F'$ is a left adjoint right inverse to $\text{res}_{co}$.
\end{proof}

Let $K$ be a simplicial set. We are now interested in forming pushouts along object-wise cofibrations in $\calc^K$. In this section we call a morphism $\beta\co K \times \Delta[1] \to \calc$ in $\calc^K$ an {\it object-wise cofibration} if for each object $k \in K$ the 1-simplex $\alpha(k,-): \Delta[1] \to \calc$ is a cofibration in $\calc$. Of course, every cofibration in $S_n^\infty \calc$ is also such an object-wise cofibration, but the converse is not true. In the main goal Corollary~\ref{cor:main_corollary_about_pushouts} we only need object-wise cofibrations. We next show that pushouts along object-wise cofibrations in $\calc^K$ can be constructed object-wise using any functorial choice from Proposition~\ref{prop:functorial_choice_of_pushouts_pres_id}, {\it and this can be done naturally in the variable $K$}. Notice first that $$(\calc^\cali_{co})^K\cong(\calc^K)^\cali_{co}\;\;\; \;\;\text{and}\;\;\;\;\; (\calc^\calj_{co})^K\cong(\calc^K)^\calj_{co}\;,$$ but only after the next theorem will we know $(\calc^\calj_{co,po})^K\cong(\calc^K)^\calj_{co,po}$.

\begin{theorem} \label{thm:computation_of_pushouts_objectwise}
Let $K$ be a simplicial set. If $\calc$ admits pushouts along cofibrations, then $\calc^K$ admits pushouts along object-wise cofibrations, and these can be computed object-wise.  More precisely, if $F_{po}\co \calc^\cali_{co} \to \calc^\calj_{co,po}$ is any functorial choice of pushouts in $\calc$ along cofibrations, then
$$\xymatrix{(F_{po})^K \co (\calc^\cali_{co})^K \ar[r] & (\calc^\calj_{co,po})^K}$$
$$(K \times \Delta[n] \overset{y}{\to} \calc_{co}^\cali) \mapsto (F_{po} \circ y)$$
is a functorial choice of pushouts in $\calc^K$.
\end{theorem}
\begin{proof}
Let $F_{po}\co \calc^\cali_{co} \to \calc^\calj_{co,po}$ be any functorial choice of pushouts in $\calc$ along cofibrations, i.e. a right inverse to $\text{res}_{po}$. Then by Proposition~\ref{prop:any_functorial_po_extents_to_leftadjoint_rightinverse}, the map $F:=\text{incl} \circ F_{po}$ is a left adjoint right inverse to $\text{res}_{co}$. Next, $(-)^K\co \mathbf{SSet}^{\tau_1} \to \mathbf{SSet}^{\tau_1}$ is a 2-functor since $\mathbf{SSet}^{\tau_1}$ is Cartesian closed, see 3.2.4 - 3.2.6 of \cite{RiehlVerity}. Since 2-functors map adjunctions to adjunctions, we have the adjunction $F^K \dashv (\text{res}_{co})^K$ with identity unit. But $(\text{res}_{co})^K$ is the same as the restriction $(\calc^K)^\calj_{co} \to (\calc^K)^\cali_{co}$, so may now apply Proposition~\ref{prop:RiehlVerity_Corollary_For_Pushouts} to conclude that $\calc^K$ admits pushouts along object-wise cofibrations, and that $F^K$ lands in $(\calc^K)^\calj_{co,po}$. In other words, computing pushouts along cofibrations object-wise produces a pushout in $\calc^K$ along an object-wise cofibration, and $(F_{po})^K$ is a $(F^K)_{po}$.
\end{proof}

\begin{proposition} \label{prop:naturality_of_object-wise_pushouts}
The object-wise computation of pushouts along object-wise cofibrations in a diagram category is natural in the source diagram. More precisely, if $\calc$ admits pushouts along cofibrations, $g\co K \to L$ is a map of simplicial sets, and $F_{po}\co \calc^\cali_{co} \to \calc^\calj_{co,po}$ is any functorial choice of pushouts along cofibrations, then the following diagram commutes.
$$\xymatrix{(\calc^\cali_{co})^L \ar[r]^{F^L_{po}} \ar[d]_{g^*} & (\calc^\calj_{co,po})^L \ar[d]^{g^*} \\
(\calc^\cali_{co})^K \ar[r]_{F^K_{po}} & (\calc^\calj_{co,po})^K}$$
\end{proposition}
\begin{proof}
Let $y\co L\times \Delta[n] \to \calc^\cali_{co}$ be an $n$-simplex in the upper left corner. Then clearly
$$\Big( F_{po} \circ y \Big) \circ (g \times \text{Id}_{\Delta[n]})=F_{po} \circ \Big(y \circ (g \times \text{Id}_{\Delta[n]})\Big)$$
by associativity.
\end{proof}

Finally, we can now conclude the corollary needed for the proof of quasicategorical additivity for object simplicial sets in Lemma~\ref{lem:HardLemmaMadeEasy_quasicategorical}. For Lemma~\ref{lem:HardLemmaMadeEasy_quasicategorical}, we will apply Corollary~\ref{cor:main_corollary_about_pushouts} with a functorial choice of pushouts along cofibrations {\it that preserves identity morphisms}, which is guaranteed to exist by Proposition~\ref{prop:functorial_choice_of_pushouts_pres_id}.

\begin{corollary} \label{cor:main_corollary_about_pushouts}
Suppose $\calc$ admits pushouts along cofibrations, we select a functorial choice of pushouts along cofibrations in $\calc$, and we use this as in Theorem~\ref{thm:computation_of_pushouts_objectwise} to object-wise compute pushouts along object-wise cofibrations in all quasicategories of diagrams in $\calc$. Let $T,U,V$ be diagrams in $\calc$ of the shape $\Delta[2] \times \Delta[n] \times \Delta[n]$. If
$$\xymatrix{T \ar[r] \ar[d] & U \\ V & }$$
is in $\calc^{\Delta[2]\times\Delta[n]\times\Delta[n]}$ and the vertical map is an object-wise cofibration, then the pushout object $P$, as a diagram in $\calc$ of shape $\Delta[2]\times\Delta[n]\times\Delta[n]$, is compatible with all face and degeneracy maps. In other words, $d_iP$ is the pushout of $d_iT$, $d_iU$, $d_iV$, while $s_iP$ is the pushout of $s_iT$, $s_iU$, $s_iV$. The face and degeneracy maps here are
$$d_iP:=P \circ \left( \text{\rm Id}_{\Delta[2]} \times (\delta^i)^* \times (\delta^i)^*  \right)\phantom{.}$$
$$s_iP:=P \circ \left( \text{\rm Id}_{\Delta[2]} \times (\sigma^i)^* \times (\sigma^i)^*  \right).$$
\end{corollary}

Theorem~\ref{thm:Riehl--Verity_ExtraOrdinary_Naturality} is a version of Corollary~\ref{cor:main_corollary_about_pushouts} in the case of all maps cofibrations, where one does not first select a functorial pushout in $\calc$ to construct pushouts in diagram quasicategories compatible with face and degeneracy maps. This was proposed to us by Riehl and Verity, and uses the following result in \cite{RiehlVerityCompleteness}. An alteration of the resulting pushout functors to ones that preserve identities in levels 0,1,2 would need further discussion. However, since we do not use Theorem~\ref{thm:Riehl--Verity_ExtraOrdinary_Naturality}, we will not work out any details of alterations.

\begin{theorem}[Riehl-Verity Theorem~1.1 in \cite{RiehlVerityCompleteness}] \label{thm:Riehl--Verity_ExtraOrdinary_Naturality}
Let $\underline{\text{\rm qCat}}_{\infty,2}$ denote the simplicially enriched category of quasicategories with
its simplicial enrichment inherited from $\mathbf{SSet}$ with usual internal hom $B^A$.
Let $X$ be a simplicial set. The quasicategorically enriched subcategory
of $\underline{\text{\rm qCat}}_{\infty,2}$ spanned by those quasicategories admitting (co)limits of shape X and those
functors preserving them is closed in $\underline{\text{\rm qCat}}_{\infty,2}$ under all projective cofibrant weighted
limits.
\end{theorem}

Their theorem means the following.  Let $X$ be a simplicial set, $\mathbf{D}$ a small simplicial category, $W\colon \mathbf{D} \to \mathbf{SSet}$ a projectively cofibrant simplicial functor called the {\it weight}, and $D\colon \mathbf{D} \to \underline{\text{\rm qCat}}_{\infty,2}$ a simplicial functor. Suppose that all quasicategories in the diagram $D$ admit colimits of shape $X$, and that all maps in the diagram $D$ preserve colimits of shape $X$ up to equivalence.
Then the {\it  simplicial limit $\{W,D\}$ of $D$ weighted by $W$} exists, it is a quasicategory that admits all colimits of shape $X$, and the legs $\{W,D\} \to Dd$ of the limit cone all preserve them. The statement also holds if the words ``colimits of shape $X$'' are replaced by ``limits of shape $X$.''

Now to obtain from their theorem a pushout functor in $\calc^{\Delta[2] \times \Delta[n] \times \Delta[n]}$ (cofibrations are all maps) compatible with face and degeneracy maps, we let the simplicial set $X$ be $\cali$ and $\mathbf{D}:=\Delta_2 \times \Delta \times \Delta$ where $\Delta_2$ is the full subcategory of $\Delta$ on the objects $[0]$, $[1]$, and $[2]$. The weight $W\colon \mathbf{D} \to \mathbf{SSet}$ is the represented functor
$$\xymatrix{W=\Delta[-] \times \Delta[-] \times \Delta[-] \colon \Delta_2 \times \Delta \times \Delta \ar[r] & \mathbf{SSet}},$$ and the diagram $D\colon \mathbf{D} \to \underline{\text{\rm qCat}}_{\infty,2}$ is constant $\calc$. If $\calc$ admits all pushouts, then Theorem~\ref{thm:Riehl--Verity_ExtraOrdinary_Naturality} implies that the limit $\{W,D\}$ is a quasicategory and admits pushouts. Theorem~5.2.12 of \cite{RiehlVerity} then guarantees the existence of a pushout functor $\{W,D\}^X \to \{W,D\}$. From this pushout functor on the limit $\{W,D\}$, one can construct pushout functors for each $\calc^{W(i,m,n)}=\calc^{\Delta[i] \times \Delta[m] \times \Delta[n]}$, natural in the variables $i$, $m$, and $n$.  The only caveat is that $\calc$ does not have all pushouts in our case, rather only pushouts along cofibrations. However, this is not a problem: Riehl--Verity's use of absolute liftings allows them to also handle situations in which some (but not all) pushouts exist in a quasicategory.

\section{Additivity for $(S^\infty_\bullet)_\text{equiv}$ of a Waldhausen Quasicategory} \label{sec:Additivity_for_S_bullet_infinity}

We introduce Waldhausen quasicategories, the $(\infty,2)$-category $\mathbf{QWald}_{\infty,2}$ of such, the 2-category $\mathbf{QWald}_{2}$ of such, and the endo-functor $S_n^\infty$ on these. Then we prove Additivity of $(S^\infty_\bullet)_\text{equiv}$ for the Waldhausen quasicategory $\cale(\cala,\calc,\calb)$ of cofiber sequences. We also give a variety of examples of Waldhausen quasicategories constructed by Barwick and Blumberg--Gepner--Tabuada.

\begin{definition}[Waldhausen Quasicategory] \label{def:Waldhausen_quasicat}
A {\it Waldhausen quasicategory} consists of a quasicategory $\calc$ with zero objects and a sub quasicategory $co \calc$, the 1-simplices of which are called {\it cofibrations} and denoted $\rightarrowtail$, such that
\begin{enumerate}
\item \label{def:Waldhausen_quasicat:(i)}
The sub quasicategory $co \calc$ is 1-full in $\calc$ and contains all equivalences in $\calc$,
\item \label{def:Waldhausen_quasicat:(ii)}
For each object $A$ of $\calc$ and any zero object $\ast$ of $\calc$, every morphism $\ast \to A$ is a cofibration,
\item \label{def:Waldhausen_quasicat:(iii)}
The pushout of a cofibration along any morphism exists, and every pushout of a cofibration along any morphism is a cofibration.
\end{enumerate}
\end{definition}

Several remarks about the definition of Waldhausen quasicategory are in order. We begin with comments about the equivalences. The sub quasicategory $co \calc$ contains all of $\calc_{\mathrm{equiv}}$ by \ref{def:Waldhausen_quasicat:(i)} because $\calc_{\mathrm{equiv}}$ is 1-full on the equivalences. A major difference between the classical and the quasicategorical settings is that a Waldhausen {\it quasi}category does not have an additional structure $w\calc$, rather the ``weak equivalences''  are the equivalences of the quasicategory, so $w\calc$ is always $\calc_{\mathrm{equiv}}$ in this paper.\footnote{The variant of Waldhausen quasicategory where there is a quasicategory $w\calc$ different from $\calc_{\mathrm{equiv}}$ is studied in \cite[Section 9]{Barwick} under the name ``labelled Waldhausen $\infty$-category.'' There, Barwick associates to each labelled Waldhausen $\infty$-category a ``virtual Waldhausen $\infty$-category.'' } Consequently, the  saturation axiom (3-for-2 property of weak equivalences) automatically holds for quasicategories, though classically it as an additional axiom beyond the axioms of Waldhausen category. Another difference from the classical setting is the status of the ``gluing lemma'': classically it is a basic axiom of Waldhausen category, but quasicategorically it automatically holds (any pushout in a quasicategory is automatically invariant under equivalence).

An important consequence of the Waldhausen quasicategory axioms is that $co\calc$ is homotopy replete. Recall that a sub quasicategory $\calr$ of a quasicategory $X$ is {\it homotopy replete} if for every commutative square in $X$ with vertical morphisms equivalences
\begin{equation} \label{equ:diagram_for_homotopy_repleteness}
\begin{array}{c}
\xymatrix{x \ar[r]^r \ar[d]_{\mathrm{equiv}} & y \ar[d]^{\mathrm{equiv}} \\ x' \ar[r]_{r'} & y'}
\end{array}
\end{equation}
we have $r \in \calr_1 \Leftrightarrow r' \in \calr_1$. See \cite[Definition~F.1.1]{JoyalQuadern} for this definition in a model category. Any commutative square of the form \eqref{equ:diagram_for_homotopy_repleteness} is a pushout, so in a Waldhausen quasicategory, by \ref{def:Waldhausen_quasicat:(iii)}, we have $r \in (co\calc)_1\Rightarrow r' \in (co\calc)_1$. For the other implication, we consider the image of \eqref{equ:diagram_for_homotopy_repleteness} in $\tau_1(\calc)$, reverse the isomorphisms, and then use Proposition~\ref{lem:commutative_squares_in_homotopy_category=commutative_squares_in_quasicategory} to obtain \eqref{equ:diagram_for_homotopy_repleteness} with $r$ and $r'$ exchanged. We conclude $co \calc$ is homotopy replete.

In a Waldhausen quasicategory, any morphism homotopic to a cofibration is also a cofibration. This follows from homotopy repleteness of $co \calc$: if $\calr$ is homotopy replete in $X$ and $f$ is a morphism in $\calr$, then any morphism $g$ homotopic to $f$ is also in $\calr$, as any left and right homotopies provide us with a commutative square of the following form.
$$\xymatrix{x \ar[r]^f \ar[d]_{1_x} \ar[dr]^f & y \ar[d]^{1_y} \\ x \ar[r]_{g} & y}$$

Another observation from the definition of Waldhausen quasicategory is that $\tau_1(co\calc)$ is naturally a subcategory of $\tau_1(\calc)$  by 1-fullness and \ref{def:Waldhausen_quasicat:(iii)}. Namely, since any morphism homotopic to a cofibration is also a cofibration, the homotopy class of a cofibration in $co\calc$ is the same as its homotopy class in $\calc$. By 1-fullness of $co\calc$, a relation $[g][f]=[h]$ between homotopy classes of cofibrations holds in $\tau_1(co\calc)$ if and only if it holds in $\tau_1(\calc)$, and we now have $\tau_1(co\calc)$ naturally embedded in $\tau_1(\calc)$. Moreover $\tau_1(co\calc)$ contains all the isomorphisms of $\tau_1(\calc)$.

A minor difference from the classical notion is that a Waldhausen quasicategory has zero objects {\it without} distinguishing one. This is in line with the philosophy of quasicategories that structure is not chosen, rather merely required to exist. It also simplifies some formulations. However, to construct the $K$-theory spectrum in Section~\ref{sec:K-Theory_Space_and_Spectrum} we will select a zero object and natural transformations between the selected zero object and the identity functors.

In the presence of \ref{def:Waldhausen_quasicat:(i)} and \ref{def:Waldhausen_quasicat:(iii)}, it is actually sufficient in Definition~\ref{def:Waldhausen_quasicat} to require a weaker condition in \ref{def:Waldhausen_quasicat:(ii)}: for each object $A$ of $\calc$, there is some zero object $\ast$ and some morphism $\ast \to A$ that is a cofibration. Namely, for any another zero object $\ast'$ and any morphism $f\co \ast' \to A$, we know that $\ast$ and $\ast'$ are zero objects of $\tau_1(\calc)$, so the diagram in $\tau_1 \calc$
$$\xymatrix{\ast' \ar[rr]^{[f]} \ar[dr]_{\text{isomorphism}} & & A \\ & \ast \ar[ur]_{[\text{cofibration}]} & }$$
commutes, so $[f]$ is a morphism in $\tau_1(co \calc)$ and $f$ is a cofibration.

Lastly, we remark that we are primarily interested in small Waldhausen quasicategories with only finite coproducts, since the Eilenberg swindle applies when there are infinite coproducts to make the $K$-theory trivial, see \cite[Proposition~8.1]{Barwick}.

Barwick's notion of Waldhausen $\infty$-category in \cite[Definition~2.7]{Barwick} is equivalent to our Definition~\ref{def:Waldhausen_quasicat}.

\begin{proposition}[Equivalence\footnote{We thank Clark Barwick for discussion of Proposition~\ref{prop:Barwick_equivalence}. In the statement of Proposition~\ref{prop:Barwick_equivalence}, we have subsumed Barwick's Definitions 1.10 and 1.11.1 of ``subcategory'' and ``pair'' into condition (i$^\prime$).} with Barwick's Definition 2.7 in \cite{Barwick}] \label{prop:Barwick_equivalence}
Let $\calc$ be a quasicategory and $\calc_\dagger$ a simplicial subset. Then $\calc$ is a Waldhausen quasicategory with cofibration quasicategory $co \calc=\calc_\dagger$ in the sense of Definition~\ref{def:Waldhausen_quasicat} if and only if it is a Waldhausen $\infty$-category $(\calc, \calc_\dagger)$ in the sense of \cite[Definition 2.7]{Barwick}, which means:
\begin{enumerate}
\item[(i$^\prime$)] \label{prop:Barwick_equivalence:(i')}
The simplicial subset $\calc_\dagger$ contains $\calc_\mathrm{equiv}$ and is a ``subcategory'' of $\calc$ in the sense that there exists a subcategory $\bfA$ of $\tau_1 \calc$ such that the diagram
\begin{equation} \label{equ:subcategory_diagram}
\begin{array}{c}
\xymatrix{\calc_\dagger \ar@{^{(}->}[r] \ar[d] & \calc \ar[d] \\ N\bfA \ar@{^{(}->}[r] & N\tau_1\calc}
\end{array}
\end{equation}
is a pullback of simplicial sets.
\item[(ii$^\prime$)] \label{prop:Barwick_equivalence:(ii')}
The quasicategory $\calc$ contains a zero object, and for any zero object $0$ and any object $x \in \calc_0$, any morphism $0 \to x$ is a morphism in $\calc_\dagger.$
\item[(iii$^\prime$)]  \label{prop:Barwick_equivalence:(iii')}
The pushout of a cofibration along any morphism exists, and every pushout of a cofibration along any morphism is a cofibration.
\end{enumerate}
\end{proposition}
\begin{proof}
Suppose $(\calc, \calc_\dagger)$ satisfies (i$^\prime$), (ii$^\prime$), and (iii$^\prime$). 
We first prove that $\calc_\dagger$ is a quasicategory. The bottom map $N\mathbf{A}\hookrightarrow N\tau_1 \calc$ in diagram \eqref{equ:subcategory_diagram} is a mid-fibration by \cite[Proposition 2.2]{JoyalQuadern} because it is the nerve of a functor, so its pullback $\calc_\dagger \hookrightarrow \calc$ in diagram \eqref{equ:subcategory_diagram} is also a mid-fibration. The composite $\calc_\dagger \hookrightarrow \calc \to \ast$ is a mid-fibration, so now $\calc_\dagger$ is a quasicategory.

We proceed to verify \ref{def:Waldhausen_quasicat:(i)}, \ref{def:Waldhausen_quasicat:(ii)}, and \ref{def:Waldhausen_quasicat:(iii)} of Definition~\ref{def:Waldhausen_quasicat}. For \ref{def:Waldhausen_quasicat:(i)}, we first observe that the inclusion $N\bfA \hookrightarrow N \tau_1 \calc$ is 1-full because an $n$-simplex in $N\tau_1 \calc$ is in $N\bfA$ if and only if each edge is in $N\bfA$ (which is the case if and only if the spine is in $N\bfA$ by definition of nerve). So for any $m > 1$ and any commutative left square below, there is a unique lift $f_1$ for the outer square, where $Sk^1 \Delta[m]$ denotes the 1-skeleton of $\Delta[m]$.
$$\xymatrix@C=3pc{Sk^1 \Delta[m] \ar[r] \ar@{^{(}->}[d] & \calc_\dagger \ar[r] \ar@{^{(}->}[d] \ar@{}[dr]|{\text{\rm pullback}} & N\bfA \ar@{^{(}->}[d] \\ \Delta[m] \ar[r] \ar@{-->}[urr]^(.7){f_1} \ar@{-->}[ur]^{f_2} & \calc \ar[r] & N \tau_1 \calc }$$
But since the right square is a pullback by (i$^\prime$), there exists a unique map $f_2$ which makes both the lower triangle of the first square and the triangle bounded by the two dashed lines commute. The upper triangle of the first square commutes by the universal property of the pullback and a diagram chase. Thus, the inclusion $\calc_\dagger \hookrightarrow \calc$ is 1-full. Clearly, $\calc_\dagger$ contains the equivalences because it contains $\calc_\mathrm{equiv}$.

The requirements \ref{def:Waldhausen_quasicat:(ii)} and (ii$^\prime$) are the same.

The requirements \ref{def:Waldhausen_quasicat:(iii)} and (iii$^\prime$) are the same.

Conversely, suppose $\calc$ is a Waldhausen quasicategory in the sense of Definition~\ref{def:Waldhausen_quasicat} with sub quasicategory $co\calc$ of cofibrations. Since $co\calc$ is 1-full, contains the equivalences, and $\calc_\mathrm{equiv}$ is 1-full, we have $\calc_\mathrm{equiv}\subseteq co \calc=:\calc_\dagger$. Let $\bfA:=\tau_1(co\calc)$.
We must verify \ref{equ:subcategory_diagram} is a pullback diagram. Suppose $X$ is a simplicial set and $F$ and $G$ are maps such that the outer square below commutes.
$$\xymatrix{X \ar@/^1pc/[drr]^G \ar@/_1pc/[ddr]_F \ar@{-->}[dr] & & \\ & co\calc \ar@{^{(}->}[r] \ar[d] & \calc \ar[d]^h \\ & N\tau_1(co \calc) \ar@{^{(}->}[r] & N\tau_1\calc}$$
All four maps in the inner square are identity functions in degree 0, so we may take $G_0 \co X_0 \to (co\calc)_0$ as the dashed arrow in degree 0. For degree 1, if $f \in X_1$, then $Ff=h(Gf)$ and $Gf$ is in the homotopy class $Ff$ of some cofibration. By the discussion after Definition~\ref{def:Waldhausen_quasicat}, $Gf$ is also a cofibration, so $G_1$ goes into $(co\calc)_1$, and we may take $G_1 \co X_1 \to (co\calc)_1$ as the dashed arrow in degree 1. Since $G_1$ goes into $(co\calc)_1$, every edge of every $G_n(\sigma)$ is in $co\calc$, and $G_n(\sigma)$ is in $co\calc$ by 1-fullness, and we may take $G$ for the dashed arrow. Uniqueness of the dashed map is clear because the top horizontal map is an inclusion. So (i$^\prime$) follows from \ref{def:Waldhausen_quasicat:(i)}.

As remarked above, (ii$^\prime$) and \ref{def:Waldhausen_quasicat:(ii)}, as well as (iii$^\prime$) and \ref{def:Waldhausen_quasicat:(iii)}, are the same.
\end{proof}

Examples of Waldhausen quasicategories include the following.

\begin{example}{(Waldhausen Quasicategories from Classical Ones)} \label{examp:from_classical_to_Waldhausen_quasicategory_Barwick}
If $\mathcal{C}$ is a classical Waldhausen category in which the weak equivalences are the isomorphisms, then $co(N\calc):=N(co\calc)$ makes $N\calc$ into a Waldhausen quasicategory. Pushouts in $N \calc$ are the same as pushouts in $\calc$. The classical Waldhausen $K$-theory spectrum of $\calc$ is the $K$-theory spectrum of $(N\calc,Nco\calc)$, as we will soon define in Definition~\ref{def:K-theory_spectrum_of_Waldhausen_Qcat}. See Example~\ref{examp:special_case of K=KN}.

Constructing a Waldhausen quasicategory out of a classical Waldhausen category $\calc$ is more subtle when the weak equivalences are not the isomorphisms. Three options are possible under certain hypotheses, as worked out by Barwick in \cite{Barwick}: (1) form the relative nerve $N(\calc, w\calc)$ and then define cofibrations as containing $co\calc$, or (2) form the ordinary nerves $(N\calc, N(co\calc), N(w\calc))$ as a ``labelled'' Waldhausen quasicategory and define another notion of $K$-theory for labelled Waldhausen quasicategories, or (3) form the ordinary nerves $(N\calc, N(co\calc), N(w\calc))$ and then invert the weak equivalences to make them the equivalences in a resulting quasicategory, and let the cofibrations become what they must.

Ad (1), suppose the weak equivalences of a classical Waldhausen category $\calc$ satisfy the 6-for-2 axiom, and its weak equivalences and trivial cofibrations are part of a three-arrow calculus of fractions. In the {\it relative nerve} $N(\calc, w\calc)$, define $coN(\calc, w\calc)$ to be the smallest 1-full sub quasicategory that is closed under the homotopy relation and contains the equivalences and the images of cofibrations from $\calc$. Then the relative nerve $N(\calc, w\calc)$ with cofibration sub quasicategory $coN(\calc, w\calc))$ is a Waldhausen quasicategory \cite[9.15]{Barwick}, and its $K$-theory is canonically equivalent to the classical Waldhausen $K$-theory of the classical Waldhausen category $(\calc,co\calc,w\calc)$ \cite[Corollary 10.10.3]{Barwick}.

Ad (2), suppose $\calc$ is a classical Waldhausen category, and assume no other hypotheses. The ordinary nerve of the classical Waldhausen category $\calc$ is a Waldhausen quasicategory equipped with an additional sub quasicategory $N(w\calc)$ which contains the equivalences and satisfies the ``gluing lemma''. Such a triple $(N\calc, N(co\calc), N(w\calc))$ is an example of a {\it labelled Waldhausen quasicategory}. The algebraic $K$-theory space of the labelled Waldhausen quasicategory $(N\calc, N(co\calc), N(w\calc))$ is {\it defined} as the space $K(\mathscr{B}(N\calc, N(w\calc)))$ in \cite[Definition~10.8]{Barwick}, where $\mathscr{B}$ is discussed in \cite[9.6 and beyond]{Barwick}. Essentially, $\mathscr{B}(N\calc, N(w\calc))$ is defined so that $\mathscr{B}_p(N\calc, N(w\calc))\subseteq \text{Fun}(\Delta[p], N\calc)$ is 0-full on those sequences $A\co\Delta[p] \to N\calc$ with every $A_i \to A_{i+1}$ an edge in $N(w\calc)$. Moreover, $K( \mathscr{B})$ means to do $K$-theory degreewise and take the colimit. The algebraic $K$-theory of the labelled Waldhausen quasicategory $(N\calc, N(co\calc), N(w\calc))$ defined as $K(\mathscr{B}(N\calc, N(w\calc)))$ is naturally equivalent to the classical Waldhausen $K$-theory of the classical Waldhausen category $(\calc, co\calc,w\calc)$ \cite[Corollary 10.10.1]{Barwick}, {\it without} hypotheses on the classical Waldhausen category $(\calc, co\calc,w\calc)$.

Ad (3), suppose $\calc$ is a classical Waldhausen category, and assume no other hypotheses. The ordinary nerves $(N\calc, N(co\calc), N(w\calc))$ form a labelled Waldhausen quasicategory. The inclusion of Waldhausen quasicategories into labelled Waldhausen quasicategories admits a left adjoint \cite[Lemma~9.14]{Barwick} which formally inverts the weak equivalences in a quasicategorical sense, and the Waldhausen quasicategory associated in this way to $(N\calc, N(co\calc), N(w\calc))$ is denoted by $(Nw\calc)^{-1}N\calc$. If the weak equivalences of the classical Waldhausen category $\calc$ satisfy the 6-for-2 axiom, and its weak equivalences and trivial cofibrations are part of a three-arrow calculus of fractions, then the Waldhausen quasicategory $(Nw\calc)^{-1}N\calc$ and the relative nerve $N(\calc, w\calc)$ are equivalent Waldhausen quasicategories \cite[Proposition~9.15]{Barwick}, and consequently the $K$-theory of the Waldhausen quasicategories $(Nw\calc)^{-1}N\calc$ and  $(N(\calc, w\calc),coN(\calc, w\calc))$ are the same (here $K$-theory means the canonical delooping of the maximal Kan subcomplex functor), and are also equivalent to the classical Waldhausen $K$-theory of the classical Waldhausen category $(\calc,co \calc, w\calc)$.

See  \cite[Sections 9 and 10]{Barwick} for details. The $K$-theories in (1), (2), and (3) all coincide when the classical Waldhausen category comes from a model category, see \cite[Corollary~10.10.3]{Barwick}.
\end{example}

Another version of part of the foregoing was suggested to us by David Gepner.

\begin{example}[Sketch of a Suggestion by D. Gepner] \label{examp:from_classical_to_Waldhausen_quasicategory}
If $\mathcal{C}$ is a classical Waldhausen category in which the weak equivalences are not the isomorphisms, then one may construct a Waldhausen quasicategory $\calc'$ with the same $K$-theory as $\calc$, though this Waldhausen quasicategory is not simply $N\calc$. Let $\calc'$ be the localization\footnote{See \cite[page 168]{JoyalQuadern}. Here this means the localization quasicategory $\calc'$ in the sense of a map $N\calc \to \calc'$ which takes the morphisms of $w\calc$ to equivalences in $\calc'$ in a universal way. This does not require the stronger sense of localization in terms of a fully faithful right adjoint as Lurie does in \cite[Section 5.2.7]{LurieHigherToposTheory}.} of $N\calc$ with respect to $N(w\calc)$, and let $co\calc'$ be the smallest 1-full sub quasicategory\footnote{We propose one should also require $co\calc'$ to be closed under the homotopy relation, and adjust the rest of the sketch appropriately.} of $\calc'$ which contains the cofibrations of $\calc$ and the equivalences of $\calc'$. A morphism of $\calc'$ is a concatenation of morphisms in $\calc$ and formally inverted morphisms of $w \calc$, while a 1-morphism of $co \calc'$ is a concatenation of cofibrations in $\calc$ and formally inverted morphisms of $w \calc$ (to see this, one uses that fact that $\tau_1\calc'$ is an equivalent category to $(\tau_1N\calc)[Nw\calc)^{-1}])$. By construction, axioms \ref{def:Waldhausen_quasicat:(i)} and \ref{def:Waldhausen_quasicat:(ii)} of Definition~\ref{def:Waldhausen_quasicat} hold. For \ref{def:Waldhausen_quasicat:(iii)}, if we have a diagram of the form $\bullet \leftarrow \bullet \rightarrowtail \bullet$ then we can construct an $m \times n$ grid of pushout squares, one row at a time, using pushouts along cofibrations in $\calc$ and Lemma~\ref{lem:pushout_along_equivalence}. A composite exists and is a pushout square. The composite edge across from the cofibration is also a cofibration by the description above. Under appropriate hypotheses, the $K$-theory of $\calc'$, which we define below, is the same as the $K$-theory of $\calc$ by a quasicategorical analogue of Waldhausen's Fibration Theorem.
\end{example}

Also relevant to this discussion of quasicategories associated to Waldhausen categories and $K$-theory comparisons, we quote the following two comparison results of Blumberg--Gepner--Tabuada. If $\calc$ is a category with weak equivalences that satisfies a homotopy calculus of two sided fractions in sense of Dwyer and Kan \cite[6.1]{DwyerKanHammock}, then the quasicategory associated to $\calc$ is $$N^\text{simp} ((L^H\calc)^\text{fib} ),$$ the homotopy-coherent nerve of a fibrant replacement of the simplicial category $L^H\calc$, which is the Hammock localization of $\calc$.

\begin{theorem}[Blumberg--Gepner--Tabuada] \label{thm:BlumbergGepnerTabuada_Comparisons} \leavevmode
\begin{enumerate}
\item
\cite[Proposition~2.10]{BlumbergGepnerTabuadaI} Let $\calc$ be a small category with a subcategory $w\calc$ of weak equivalences which satisfies a homotopy calculus of two-sided fractions in the sense of
Dwyer and Kan. Then there is a weak equivalence of simplicial sets
$$N(wC) \simeq N^\text{\rm simp} ((L^H\calc)^\text{\rm fib} )_\text{\rm equiv}.$$
\item
\cite[Theorem~7.8]{BlumbergGepnerTabuadaI}
Let $\calc$ be a sub simplicial category of a simplicial model category $\cala$. Suppose $\calc$ is small, simplicially full, admits all homotopy pushouts, has only cofibrant objects, and is a Waldhausen category with its maps that are cofibrations and weak equivalences in the model category. Then there is an equivalence of spectra
$$\bfK(\calc)\simeq  \bfK(N^\text{\rm simp}(\overline{\calc}^{\text{\rm cf}}))$$
which is natural in weakly exact functors. Here $\overline{\calc}$ is the simplicially full sub simplicial category of $\cala$ on the objects weakly equivalent via zig-zags to objects of $\calc$, and
$\overline{\calc}^{\text{\rm cf}}$ is the simplicially full sub simplicial category of $\overline{\calc}$ on the cofibrant-fibrant objects in $\overline{\calc}$. The right $K$-theory spectrum is constructed via $S_\bullet^\infty$ as in present Section~\ref{sec:K-Theory_Space_and_Spectrum}, but with every map of $N^\text{\rm simp}(\overline{\calc}^{\text{\rm cf}})$ considered a cofibration.
\end{enumerate}
\end{theorem}

Stable quasicategories, considered in \cite{LurieHigherAlgebra}, are an important class of examples of Waldhausen quasicategories.

\begin{definition}[Stable Quasicategory, 1.1.1.9 of \cite{LurieHigherAlgebra}] \label{def:stable_quasicategories}
A quasicategory is {\it stable} if it admits all finite limits and colimits, and pushout squares and pullback squares coincide.
\end{definition}

\begin{example}[Stable Quasicategories are Waldhausen Quasicategories] \label{examp:stable_quasicategories_are_Waldhausen_quasicategories}
Any stable quasicategory is a Waldhausen quasicategory with all morphisms cofibrations.
For instance, the quasicategory of modules for an $A_\infty$ ring spectrum $R$ satisfying finiteness conditions, or the quasicategory of complexes of quasicoherent $\calo_X$-modules over a quasicompact and quasiseparated scheme $X$, are stable quasicategories.

If $\calm$ is a stable model category, and $\calm^c$ denotes its full sub category on the cofibrant objects, then $N(\calm^c)[W^{-1}]$ is a stable quasicategory.

If $\calm$ is a stable simplicial model category, and $\calm^{cf}$ denotes its full sub simplicial category on the cofibrant-fibrant objects, then the simplicial nerve $N(\calm^{cf})$ is a stable quasicategory.

See \cite{LurieHigherAlgebra} and \cite[Section 2]{BlumbergGepnerTabuadaI} for details on stable quasicategories.

Genuine examples of Waldhausen quasicategories that are not stable are the nerves of Examples~\ref{examp:classical_Waldhausen_cats}~\ref{examp:classical_Waldhausen_cats:fin_based_sets}, \ref{examp:classical_Waldhausen_cats:fin_gen_Rmods}, and \ref{examp:classical_Waldhausen_cats:fin_gen_proj_Rmods}.
\end{example}

Diagrams of Waldhausen quasicategories also provide us with examples of Waldhausen quasicategories.

\begin{example}[Diagram Waldhausen Quasicategories $\cald^A$ and $\cald(m,w)$] \label{examp:diagram_Waldhausen_quasicategories}
Let $\cald$ be a Waldhausen quasicategory and $A$ any simplicial set. Then the quasicategory $\cald^A$ is a Waldhausen quasicategory in which the cofibrations are the natural transformations $A \times \Delta[1] \to \cald$ that are componentwise cofibrations. The sub simplicial set $co(\cald^A)$ is the 1-full sub simplicial set of $\cald^A$ on the componentwise cofibrations, this is a quasicategory by Proposition~\ref{prop:1_full+composites_implies_quasicategory}.
The equivalences of $\cald^A$ are cofibrations because a natural transformation is a natural equivalence if and only if its components are equivalences \cite[Theorem~5.14]{JoyalQuadern}, see Corollary~\ref{cor:nat_transf_is_nat_equiv_iff_comps_equivs}. Axioms \ref{def:Waldhausen_quasicat:(ii)} and \ref{def:Waldhausen_quasicat:(iii)} follow because colimits in a quasicategory are formed pointwise \cite[5.2.18]{RiehlVerity}, see also Section~\ref{subsec:natural_pushout_along_cofibrations}.

Consider the case $A=\Delta[m]$. In Proposition~\ref{prop:D_to_D(m,w)_is_Waldhausen_equivalence} we are interested in the Waldhausen quasicategory $\cald(m,w)$, which is the 0-full sub quasicategory of $\cald^{\Delta[m]}$ on the maps $\Delta[m] \to \cald$ which send each edge of $\Delta[m]$ to an equivalence in $\cald$. The sub quasicategory $co\cald(m,w)\subseteq \cald(m,w)$ is 1-full in $\cald(m,w)$ on the morphisms of $\cald(m,w)$ that are componentwise cofibrations in $\cald$.
In Proposition~\ref{prop:D_to_D(m,w)_is_Waldhausen_equivalence} we observe a natural bijection $\mathfrak{s}_n^\infty \cald(m,w)\cong \left[ (S^\infty_n\cald)_\text{equiv} \right]_m$, and use this to prove $\mathfrak{s}_\bullet^\infty \cald \hookrightarrow(S_\bullet^\infty \cald)_\text{\rm equiv}$ is a diagonal weak equivalence of bisimplicial sets.

We are also interested in the case $A=J[m]$, the nerve of the groupoid with objects $0,1, \dots, m$ and a unique isomorphism from any object to another. In Remark~\ref{rem:replacing_D(m,w)_with_D^J[m]} we observe that $\cald(m,w)$ is Waldhausen equivalent to $\cald^{J[m]}$.
\end{example}

The Waldhausen quasicategory $\cald(m,w)$ is a sub Waldhausen quasicategory of $\cald^{\Delta[m]}$, in the following sense.

\begin{definition}[Sub Waldhausen Quasicategory] \label{def:sub_Waldhausen_quasicategory}
A {\it sub Waldhausen quasicategory} $\mathcal{B}$ of a Waldhausen quasicategory $\calc$ is a sub quasicategory $\calb$ such that $\calb$ is a Waldhausen quasicategory in its own right and
\begin{enumerate}
\item \label{def:sub_Waldhausen_quasicategory:i}
every zero object of $\calb$ is also a zero object of $\calc$,
\item \label{def:sub_Waldhausen_quasicategory:ii}
a morphism in the sub quasicategory $\calb$ is a $\calb$-cofibration if and only if it is a $\calc$-cofibration  {\it and} the $\calc$-quotient is equivalent to an object of $\mathcal{B}$,
\item \label{def:sub_Waldhausen_quasicategory:iii}
every pushout square in the sub quasicategory $\calb$ with one leg a $\calb$-cofibration is also a pushout square in the larger quasicategory $\calc$.
\item
A morphism in $\calb$ is an equivalence in $\calb$ if and only if it is an equivalence in $\calc$.
\end{enumerate}
\end{definition}

We next define Waldhausen's $S_\bullet$ construction for Waldhausen quasicategories, denoted $S_\bullet^\infty$. Below, we include cofibrations in the definitions of $[n]$-complex and $S_\bullet^\infty \mathcal{C}$ from 1.2.2.2 and 1.2.2.5 of Lurie \cite{LurieHigherAlgebra}. See Blumberg--Gepner--Tabuada \cite[Corollary~7.7 and Theorem~7.8]{BlumbergGepnerTabuadaI} for a comparison of $S_\bullet^\infty$ with the $S_\bullet'$ construction of \cite{BlumbergMandellAbstHomTheory} in the case of a simplicial model category which admits all finite homotopy colimits (all maps in a model category are weak cofibrations, so in \cite{BlumbergMandellAbstHomTheory} the cofibrations do not play a role in the application of $S_\bullet^\infty$ to the simplicial nerve of the full sub-simplicial category on the fibrant-cofibrant objects). See also Barwick \cite[Section~5 and Corollary~6.9.1]{Barwick} for a quasicategorical version of the $S_\bullet$ construction called {\it the virtual Waldhausen $\infty$-category of totally filtered objects}.

\begin{definition}[$S_\bullet^\infty$ Construction] \label{def:Sbulletinfinity}
Let $\mathcal{C}$ be a Waldhausen quasicategory. An {\it $[n]$-complex} is a map of simplicial sets $A \co N\text{Ar}[n] \to \mathcal{C}$ such that
\begin{enumerate}
\item
For each $j \in [n]$, $A(j,j)$ is a zero object of $\mathcal{C}$, possibly different for each $j$,
\item
For each $i \leq j \leq k$, the morphism $A(i,j) \to A(i,k)$ is a cofibration,
\item
For each $i \leq j \leq k$, the diagram
$$\xymatrix{A(i,j) \ar@{>->}[r] \ar[d] & A(i,k) \ar[d] \\ A(j,j) \ar@{>->}[r] & A(j,k)}$$
is a pushout square in $\mathcal{C}$.
\end{enumerate}
Let $S^{\infty}_n\mathcal{C}$ be the sub quasicategory of $\mathcal{C}^{N\text{Ar}[n]}$ that is 0-full on the $[n]$-complexes.\footnote{In \cite{LurieHigherAlgebra} and \cite{BlumbergGepnerTabuadaI}, which have all maps cofibrations, $S^{\infty}_n\mathcal{C}$ is also called $\mathrm{Gap}([n], \calc)$.}
\end{definition}

Notice that all of the diagonal entries $A(j,j)$ can be different zero objects, and that a morphism $A \to A'$ in $S_n^\infty\calc$ has equivalences $A(j,j) \to B(j,j)$ on the diagonal, not equalities as in the classical case.

\begin{remark}
The $S_\bullet^\infty$ construction of $\mathcal{C}$ is the simplicial quasicategory $S_\bullet^\infty \mathcal{C}$ defined by
$$n \mapsto S_n^\infty \mathcal{C}\subseteq \mathbf{SSet}(N\left(\text{Ar}[n]\right) \times \Delta[-],\mathcal{C})=\mathcal{C}^{N\text{Ar}[n]}.$$
The face and degeneracy maps of the simplicial object $S_\bullet^\infty \mathcal{C}$ in the category $\mathbf{QCat}$ are induced via the category of arrows construction.  More precisely, if $A \co N\text{Ar}[n] \to \mathcal{C}$ and $\alpha\co [k] \to [n]$, then $\alpha^* A=A \circ (\alpha_* \times \alpha_*)\co N\text{Ar}[k] \to \mathcal{C}$. Compare with Remark~\ref{rem:Sdot_is simplicial_object}.
\end{remark}

\begin{remark}
For any Waldhausen quasicategory $\calc$, the quasicategory $S_n^\infty \calc$ is also a Waldhausen quasicategory with cofibrations defined analogously to the classical case. However, notice that cofibrations in $S_n^\infty\calc$ are not defined as the objectwise cofibrations (though every $S_n^\infty\calc$ cofibration is also an objectwise cofibration). This is different from $\calc^{\Delta[n]}$ in Example~\ref{examp:diagram_Waldhausen_quasicategories}, where the cofibrations {\it are} precisely the objectwise cofibrations.
\end{remark}

\begin{remark}[The $(\infty,2)$-Category $\mathbf{QWald}_{\infty,2}$ and the $(\infty,2)$-Functor $S_n^\infty$] \label{rem:infinity-2-category_of_Waldhausen_quasicategories}
The quasicategory-enriched strict 1-category of Waldhausen quasicategories and exact functors is denoted by $\mathbf{QWald}_{\infty,2}$. The quasicategory $\mathbf{QWald}_{\infty,2}((\calc,co\calc), (\cald,co\cald))$ is the sub quasicategory of $\cald^\calc$ that is 0-full on the exact functors $\calc \to \cald$. This quasicategory enrichment of $\mathbf{QWald}_{\infty,2}$ is inherited from the quasicategory enrichment of $\mathbf{QCat}_{\infty,2}$, which is a full sub simplicial category of $\mathbf{SSet}_\text{simp}$, as described in Section~\ref{subsec:Quasicategories}.  Though we will not use it, we can obtain a quasicategory of Waldhausen quasicategories by taking first the maximal Kan subcomplexes of the hom quasicategories of $\mathbf{QWald}_{\infty,2}$ and then the homotopy coherent nerve of the resulting Kan-enriched category, the resulting quasicategory of Waldhausen quasicategories coincides with $\mathbf{Wald}_\infty$ introduced in \cite[Notation~2.13]{Barwick}. Instead we prefer to work in $\mathbf{QWald}_{\infty,2}$ in order to discuss adjunctions and split exact sequences of Waldhausen quasicategories.

In this $(\infty,2)$-context we can also quickly verify that $S_n^\infty$ extends to an $(\infty,2)$-functor, that is, $S_n^\infty$ is a quasicategory-enriched strict endo-functor on a quasicategory-enriched strict 1-category. If we take $A$ in Proposition~\ref{prop:exp_is_simplicial}~\ref{prop:exp_is_simplicial:i} and \ref{prop:exp_is_simplicial:ii} to be $N\text{Ar}[n]$, then the exponentiation $(-)^{N\text{Ar}[n]}$ is a simplicially enriched endo-functor on $\mathbf{QCat}_{\infty,2}$, so an $(\infty,2)$-functor. As $S_n^\infty \calc$ is 0-full in $\calc^{N\text{Ar}[n]}$ and is itself a Waldhausen quasicategory, we also have that $S_n^\infty$ is an $(\infty,2)$-endofunctor on $\mathbf{QWald}_{\infty,2}$. We can explicitly describe $S_n^\infty$ on exact functors and natural transformations. If $f\co \calc \to \cald$ is an exact functor (see Definition~\ref{def:exact_functor}), then the functor $S^\infty_nf\co S^{\infty}_n\mathcal{C} \to S^{\infty}_n\mathcal{D}$ is the sub functor of postcomposition $f_\ast \co \text{Fun}(N\text{Ar}[n], \mathcal{C})\to \text{Fun}(N\text{Ar}[n], \mathcal{D})$ on $S_n^\infty\calc$. For a natural transformation $\alpha \co \calc \times \Delta[1] \to \cald$ of exact functors, the natural transformation $S_n^\infty \alpha \co S_n^\infty \calc \times \Delta[1] \to S_n^\infty \cald$ is in level $m$ the appropriate sub map of
$$\xymatrix{\big(\calc^{N\text{Ar}[n]}\big)_m \times \big( \Delta[1] \big)_m \ar[r] & \big( \cald^{N\text{Ar}[n]} \big)_m}$$
$$\xymatrix{\Big( A \co N\text{Ar}[n] \times \Delta[m] \to \calc, \;\; q\co [m] \to [1] \Big) \ar@{|->}[r] & \text{comp}_m(q^\ast \alpha,A) }$$
by Proposition~\ref{prop:exp_is_simplicial}~\ref{prop:exp_is_simplicial:i} and \ref{prop:exp_is_simplicial:ii}.
If $\alpha$ is a natural equivalence, then so is $S_n^\infty \alpha$.

An equivalence respectively adjunction in $\mathbf{QWald}_{\infty,2}$ is an equivalence respectively adjunction in the 2-category $\mathbf{QWald}_2$, compare with Section~\ref{subsec:Adjunctions_and_Equivalences_between_Quasicategories}.
\end{remark}

\begin{remark}[The 2-Category $\mathbf{QWald}_2$ and the 2-Functor $S_n^\infty$] \label{rem:QWald2}
Since the left adjoint $\tau_1$ to the nerve preserves finite products, we may apply it to the hom simplicial sets of $\mathbf{QWald}_{\infty,2}$ and obtain the 2-category $\mathbf{QWald}_2$ of Waldhausen quasicategories. Its hom categories are
$$\mathbf{QWald}_2((\calc,co\calc), (\cald,co\cald))=\tau_1\mathbf{QWald}_{\infty,2}((\calc,co\calc), (\cald,co\cald)).$$
An equivalence in the 2-category $\mathbf{QWald}_2$ is precisely a Waldhausen equivalence between Waldhausen quasicategories in the sense of Definition~\ref{def:Waldhausen_equivalence:qcats}. Adjunctions in $\mathbf{QWald}_2$ feature prominently in split exact sequences of Waldhausen quasicategories in Definition~\ref{def:split-exact_sequence_of_Waldhausen_quasicategories}. The $(\infty,2)$-functor $S_n^\infty$ of Remark~\ref{rem:infinity-2-category_of_Waldhausen_quasicategories} induces via $\tau_1$ on hom simplicial sets a 2-functor $S_n^\infty \co \mathbf{QWald}_2 \to \mathbf{QWald}_2$. It will always be clear from the context whether we mean by $S_n^\infty$ the $(\infty,2)$-functor or the 2-functor.
\end{remark}

\begin{example}[$NS_n$ and $S_n^\infty N$ are Equivalent when Weak Equivalences are Isomorphisms] \label{examp:special_case of K=KN}
%
If $\mathcal{C}$ is a (classical) Waldhausen category in which the weak equivalences are the isomorphisms, then its nerve $N\calc$ is a Waldhausen quasicategory with $co(N\calc)=N(co\calc)$. The inclusion $NS_n\calc \hookrightarrow S^\infty_n N\calc$ is an equivalence of quasicategories which is the nerve of an equivalence of categories, namely $S_n\calc$ includes essentially surjectively and fully faithfully into the variant of $S_n\calc$ which allows any zero objects on the diagonal, not just the distinguished one of $\calc$. This inclusion $NS_n\calc \hookrightarrow S^\infty_n N\calc$ similarly restricts to an essentially surjective and fully faithful inclusion $$NwS_n\mathcal{C}=(NS_n\mathcal{C})_\text{equiv}\hookrightarrow\left(S^\infty_n N\mathcal{C}\right)_\mathrm{equiv}$$ by Lemma~\ref{lem:iso_commutes_with_Nerve}. Notice that we are also using the facts that
$$\aligned
\mathbf{Cat}(\text{Ar}[n]\times [m], \mathcal{C})
& \cong\mathbf{SSet}(N\left(\text{Ar}[n]\times [m]\right),N\mathcal{C}) \\
& \cong\mathbf{SSet}(N\left(\text{Ar}[n]\right)\times \Delta[m],N\mathcal{C})
\endaligned$$
and pushout squares and zero objects in $N\calc$ are the same as pushout squares and zero objects in $\calc$. Thus, when the weak equivalences of a Waldhausen category $\calc$ are the isomorphisms, the nerve $N\calc$ is a Waldhausen quasicategory with $K$-theory space and spectrum (see Section~\ref{sec:K-Theory_Space_and_Spectrum}) the same as the classical $K$-theory space and spectrum of $\calc$. For instance, this applies to Examples~\ref{examp:classical_Waldhausen_cats}~\ref{examp:classical_Waldhausen_cats:fin_based_sets}, \ref{examp:classical_Waldhausen_cats:fin_gen_Rmods}, and \ref{examp:classical_Waldhausen_cats:fin_gen_proj_Rmods}, none of which have nerves stable quasicategories.
\end{example}


Moving towards Additivity, we now consider the quasicategory of cofiber sequences with subobject and quotient in specified sub quasicategories.

\begin{notation}[Quasicategorical $\cale(\cala,\calc,\calb)$]  \label{not:E(A,C,B)_quasicategorical}
Recall Notation~\ref{not:E(A,C,B)} and Example~\ref{examp:E(A,C,B)}.
Let $\calc$ be a Waldhausen quasicategory and $\cala$ and $\calb$ sub Waldhausen quasicategories.
A {\it cofiber sequence $A \rightarrowtail C \twoheadrightarrow B$ in $\calc$} is a map $\Delta[2] \to \calc$ such that $A \rightarrowtail C$ is a cofibration, and for which there exists an extension to a pushout square $\Delta[1] \times \Delta[1]\to \calc$ of the form
\begin{equation} \label{equ:cofiber_sequence_quasicategorical}
\begin{array}{c}
\xymatrix{A \ar@{>->}[r] \ar[d] \ar[dr] & C \ar[d] \\ \ast \ar[r] & B}
\end{array}
\end{equation}
with $*$ some zero object and the upper triangle the datum $\Delta[2] \to \calc$. In other words, a chosen composite is part of the data of a cofiber sequence, but the rest of the pushout square is not. The {\it quasicategory of cofiber sequences in $\calc$} is denoted $\cale(\calc, \calc, \calc)$. It is the sub quasicategory of $\calc^{\Delta[2]}$ that is 0-full on the cofiber sequences $A \rightarrowtail C \twoheadrightarrow B$. We have ``subobject'' and ``quotient'' maps $s,q\co \calc^{\Delta[2]} \to \calc$, which on an $n$-simplex $\sigma\co \Delta[2] \times \Delta[n] \to \calc$ are
$$s(\sigma)=\sigma(0,-) \hspace{.75in} \text{and} \hspace{.75in} q(\sigma)=\sigma(2,-)$$
as pictured in \eqref{equ:canonical_split_exact_functors_categorical}.
The notation $\sigma(0,-)$ is shorthand for the precomposition of $\sigma$ with the map $\Delta[n] \to \Delta[2] \times \Delta[n]$ which is the nerve of $[n] \to [2] \times [n]$, $i \mapsto (0,i)$, etc.
Then the {\it quasicategory of cofiber sequences $A \rightarrowtail C \twoheadrightarrow B$ with $A \in \cala$ and $B \in \calb$}, denoted $\cale(\cala, \calc, \calb)$, is the following pullback in $\mathbf{SSet}$.
$$\xymatrix@C=3pc{\cale(\cala, \calc, \calb) \ar[r] \ar[d] \ar@{}[dr]|{\text{pullback}} & \cale(\calc,\calc,\calc)
\ar[d]^{(s,q)} \\ \cala \times \calb \ar[r]_{\text{incl}} & \calc \times \calc }$$

A diagram morphism $(A_1,C_1,B_1)\to (A_2,C_2,B_2)$ in $\cale(\cala, \calc, \calb)$ is a {\it cofibration} if $A_1 \to A_2$ is a cofibration in $\cala$ and if any (equivalently every) ``induced map'' $C_1 \cup_{A_1} A_2 \to C_2$ is a cofibration in $\calc$, (this is well defined because any map homotopic to a cofibration is a cofibration, and because any cofibration composed with an equivalence is a cofibration). Then $co\cale(\cala, \calc, \calb)$ is defined to be the 1-full sub simplicial set of $\cale(\cala, \calc, \calb)$ on the cofibrations. The composite of two cofibrations is a cofibration by an argument similar to \cite[Lemma~1.1.1]{WaldhausenAlgKTheoryI}, so $co\cale(\cala, \calc, \calb)$ is a quasicategory by Proposition~\ref{prop:1_full+composites_implies_quasicategory}. Every equivalence of $\cale(\cala, \calc, \calb)$ is a cofibration because the pushout of an equivalence is an equivalence, so the ``induced map'' $C_1 \cup_{A_1} A_2 \to C_2$ will also be an equivalence by the 3-for-2 property of equivalences. Requirement \ref{def:Waldhausen_quasicat:(ii)} of Definition~\ref{def:Waldhausen_quasicat} is clear, and \ref{def:Waldhausen_quasicat:(iii)} holds by an argument similar to \cite[Lemma~1.1.1]{WaldhausenAlgKTheoryI}, so $\cale(\cala, \calc, \calb)$ is a Waldhausen quasicategory.
\end{notation}

\begin{remark}
In Notation~\ref{not:E(A,C,B)_quasicategorical} we did not include the pushout data in $\cale(\cala, \calc, \calb)$. We next define an equivalent version $\cale^{po}(\cala, \calc, \calb)$ which does include the pushout data in \eqref{equ:cofiber_sequence_quasicategorical} with a fixed selected zero object $\ast$. We need the pushout data to construct the left vertical map of \eqref{equ:rho_n_as_pushout:quasicategorical} and the homotopy $\theta$ in Example~\ref{examp:nat_transf_induces_s_bullet_homotopy_quasicategorical}, which is used near the end of the proof of Lemma~\ref{lem:HardLemmaMadeEasy_quasicategorical}. Of course, the restriction $\cale^{po}(\cala,\calc,\calb) \to \cale(\cala,\calc,\calb)$ is a Waldhausen equivalence.
 \end{remark}

\begin{notation}[Quasicategorical $\cale^{po}(\cala,\calc,\calb)$ with Pushout Data]  \label{not:E'(A,C,B)_quasicategorical}
Recall Notation~\ref{not:E(A,C,B)}, Example~\ref{examp:E(A,C,B)}, and Notation~\ref{not:E(A,C,B)_quasicategorical}. Let $\calc$ be a Waldhausen quasicategory with a selected zero object $\ast$ and $\cala$ and $\calb$ sub Waldhausen quasicategories.
An object of $\cale^{po}(\calc,\calc,\calc)$ is a map $\Delta[1] \times \Delta[1]\to \calc$ of the form \eqref{equ:cofiber_sequence_quasicategorical}
such that the square is a pushout square, $\ast$ is the selected zero object, and the map $A \rightarrowtail C$ is a cofibration. In particular, the data of an object includes morphisms $A \to \ast$, $\ast \to B$, a diagonal morphism, and two 2-simplices with the indicated boundaries. The quasicategory $\cale^{po}(\calc,\calc,\calc)$ is the 0-full sub quasicategory of $\calc^{\Delta[1] \times \Delta[1]}$ on the objects just defined.

An $n$-simplex of $\cale^{po}(\cala, \calc, \calb)$ is a map $\sigma\co \Delta[1] \times \Delta[1] \times \Delta[n] \to \calc$ such that $\sigma(0,0,-)$ goes into $\cala$, $\sigma(1,1,-)$ goes into $\calb$, for all $0\leq \ell \leq n$ the object $\sigma(1,0,\ell)$ is the selected zero object, and each square $\sigma(-,-,\ell)$ is a pushout square.

In more detail, consider the objects $\{(0,0)\}$ and $\{(1,1)\}$ in $[1] \times [1]$ respectively. Via nerve and precomposition, these induce ``subobject'' and ``quotient'' functors $s,q\colon \calc^{\Delta[1] \times \Delta[1]} \to \calc$, which on an $n$-simplex $\sigma\co \Delta[1] \times \Delta[1] \times \Delta[n] \to \mathcal{C}$ are defined as
$$
s(\sigma)=\sigma(0,0,-) \hspace{.75in} \text{and} \hspace{.75in} q(\sigma)=\sigma(1,1,-).
$$
Then the quasicategory $\cale^{po}(\cala, \calc, \calb)$ is the following pullback in $\mathbf{SSet}$.
$$\xymatrix@C=3pc{\cale^{po}(\cala, \calc, \calb) \ar[r] \ar[d] \ar@{}[dr]|{\text{pullback}} & \cale^{po}(\calc,\calc,\calc)
\ar[d]^{(s,q)} \\ \cala \times \calb \ar[r]_{\text{incl}} & \calc \times \calc }$$

A diagram morphism $(A_1,C_1,B_1)\to (A_2,C_2,B_2)$ in $\cale^{po}(\cala, \calc, \calb)$ is a {\it cofibration} if $A_1 \to A_2$ is a cofibration in $\cala$ and if any (equivalently every) ``induced map'' $C_1 \cup_{A_1} A_2 \to C_2$ is a cofibration in $\calc$. Then $co\cale^{po}(\cala, \calc, \calb)$ is defined to be the 1-full sub simplicial set of $\cale^{po}(\cala, \calc, \calb)$ on the cofibrations. As in Notation~\ref{not:E(A,C,B)_quasicategorical}, $co\cale^{po}(\cala, \calc, \calb)$ is a quasicategory and $\cale^{po}(\cala, \calc, \calb)$ is a Waldhausen quasicategory.
\end{notation}

In his proof of Additivity, Waldhausen uses the object part $\mathfrak{s}_\bullet$ of the $S_\bullet$ construction, as we recalled in Section~\ref{sec:simplified_proof}. We will do the same in the quasicategory context, where it means to truncate to the 0-part. Notice that both $(S_\bullet^\infty)_\text{equiv}$ and its object simplicial set $\mathfrak{s}_\bullet^\infty$ preserve products,\footnote{Recall that the simplicially enriched category $\mathbf{SSet}_\text{simp}$ is enriched Cartesian closed, so that $(-)^{N\text{Ar}[n]}$ is an enriched right adjoint and preserves enriched limits.} so we freely write $S_\bullet^\infty (s,q)_{\text{equiv}}$ to mean $((S_\bullet s)_\text{equiv}, (S_\bullet q)_{\text{equiv}} )$, and we write $\mathfrak{s}_\bullet^\infty(s,q)$ to mean $(\mathfrak{s}_\bullet^\infty s,\mathfrak{s}_\bullet^\infty q)$.

\begin{definition}[The Simplicial Set $\mathfrak{s}_\bullet^\infty \mathcal{C}$]
Let $\mathcal{C}$ be a Waldhausen quasicategory. The simplicial set $\mathfrak{s}_\bullet^\infty \mathcal{C}$ is the object part of the simplicial quasicategory $S_\bullet^\infty \mathcal{C}$.
$$\mathfrak{s}_n^\infty \mathcal{C}=(S_n^\infty \mathcal{C})_0 \subseteq \mathbf{SSet}(N\text{Ar}[n], \mathcal{C})$$
The $n$-simplices of $\mathfrak{s}_\bullet^\infty \mathcal{C}$ are the $[n]$-complexes in Definition~\ref{def:Sbulletinfinity}.
\end{definition}

\begin{lemma}[Waldhausen $(S_\bullet^\infty)_\text{equiv}$ Additivity for $\cale(\cala,\calc,\calb)$ on Object Simplicial Sets] \label{lem:HardLemmaMadeEasy_quasicategorical}
Let $\mathcal{C}$ be a Waldhausen quasicategory and $\cala$ and $\calb$ sub Waldhausen quasicategories. Then the ``subobject'' and ``quotient'' functors induce a weak homotopy equivalence of simplicial sets
\begin{equation} \label{equ:HardLemmaMadeEasy_quasicategorical}
\xymatrix{\mathfrak{s}_\bullet^\infty (s,q) \co \mathfrak{s}_\bullet^\infty \cale(\cala,\calc,\calb) \to \mathfrak{s}_\bullet^\infty \cala \times \mathfrak{s}_\bullet^\infty \calb}.
\end{equation}
\end{lemma}
\begin{proof}
The proof of Lemma~\ref{lem:HardLemmaMadeEasy} carries over to the quasicategorical context using some further justifications. We want to prove that $r$, the composite of the quasicategorical analogue of \eqref{equ:r}, is a weak equivalence.

Fix a zero object $\ast$ of $\calc$ that is in both $\cala$ and $\calb$ (if there is not a common zero object,  we select a zero object of $\calc$, add it and the necessary simplices to $\cala$ and $\calb$ to make Waldhausen equivalent quasicategories $\cala'$ and $\calb'$, prove the Lemma for $\cala'$ and $\calb'$, and then conclude the Lemma for the original $\cala$ and $\calb$ without a common zero object by the 3-for-2 property). We work with $\cale^{po}(\cala,\calc,\calb)$ from Notation~\ref{not:E'(A,C,B)_quasicategorical} and prove that
$$\xymatrix{\mathfrak{s}_\bullet^\infty (s,q) \co \mathfrak{s}_\bullet^\infty \cale^{po}(\cala,\calc,\calb) \to \mathfrak{s}_\bullet^\infty \cala \times \mathfrak{s}_\bullet^\infty \calb}$$
is a weak homotopy equivalence of simplicial sets. Then \eqref{equ:HardLemmaMadeEasy_quasicategorical} is also a weak homotopy equivalence by the 3-for-2 property of weak homotopy equivalences (recall $\cale^{po}(\cala,\calc,\calb) \to \cale(\cala,\calc,\calb)$ is a Waldhausen equivalence).

To work with the same diagrams as in \eqref{equ:element_in_left_fiber}, but with chosen composites as part of the data, we note the isomorphisms of simplicial sets $$\Big(\calc^{\Delta[1] \times \Delta[1]}\Big)^{N\text{Ar}[n]} \cong \calc^{\Delta[1] \times \Delta[1] \times N\text{Ar}[n]} \cong \Big(\calc^{N\text{Ar}[n]}\Big)^{\Delta[1] \times \Delta[1]}.$$
We now have that $(e,\alpha)\in f/(m,A)$ consists of a map $e\co \Delta[1] \times \Delta[1] \times N\text{Ar}[n] \to \calc$ and a weakly increasing map $\alpha\co[m] \to [n]$ such that $\alpha^*A$ is row $(0,0)$ of $e$.

The analogous diagrams of Lemma~\ref{lem:HardLemmaMadeEasy} are now defined in the quasicategorical context as follows. We first select, once and for all, two natural transformations $\varphi,\psi \co \Delta[1] \times \calc \to \calc$, one with $\varphi_0=\ast$ and $\varphi_1=\text{Id}_\calc$, the other with $\psi_0=\text{Id}_\calc$ and $\psi_1=\ast$, which exist because $\ast$ is a zero object in $\calc$. For $B \in \mathfrak{s}_n\calb$, we define $\iota(B)\in \mathfrak{s}^\infty_n\cale^{po}(\cala,\calc,\calb)$ in diagram \eqref{equ:iota(B)} as follows.
$$\xymatrix@C=3.5pc{\Delta[1]\times\Delta[1] \times N\text{Ar}[n] \ar[r]^-{ \text{pr}_2 \times B} \ar@/_1pc/[rr]_{\iota(B)} & \Delta[1]\times\calc \ar[r]^-\varphi & \calc}$$
Then $\iota\co \mathfrak{s}_\bullet \calb \to f/(m,A)$ is compatible with face and degeneracy maps because for instance
$$\iota(d_i B)=\iota(B \circ ((\delta^i_n)_* \times (\delta^i_n)_*)=\iota(B)\circ \Big(\text{Id}_{\Delta[1]\times\Delta[1]}\times (\delta^i_n)_* \times  (\delta^i_n)_* \Big)=d_i \iota(B).$$
For an $n$-simplex $(e,\alpha)$ in $f/(m,A)$, we define $\lambda_n^j(e,\alpha)$ in diagram \eqref{equ:lambda_rows} as follows, with $\Delta$ morphism $\alpha$.
$$\xymatrix@C=4.5pc{\Delta[1]\times\Delta[1] \times N\text{Ar}[n] \ar[r]^-{\text{pr}_1 \times (\alpha^*A)  } \ar@/_1pc/[rr]_{\lambda_n^j(e,\alpha)} & \Delta[1]\times\calc  \ar[r]^-\psi & \calc}$$
This is independent of $j$, and we see that for fixed $j$, we have that $(\lambda^j_n)_n$ is a map of simplicial sets, using the same argument as for $\iota$ above.
For an $n$-simplex $(e,\alpha)$ in $f/(m,A)$, we define $\mu_n^j(e,\alpha)$ in diagram \eqref{equ:mu_rows} as follows, with $\Delta$ morphism $(NZ)^j_n(\alpha)$.
$$\xymatrix@C=8pc{\Delta[1] \times \Delta[1] \times N\text{Ar}[n] \ar[r]^-{\text{pr}_1 \times \big(\left((NZ)^j_n(\alpha)\right)^*A\big)} \ar@/_1pc/[rr]_{\mu_n^j(e,\alpha)} & \Delta[1]\times\calc  \ar[r]^-\psi & \calc}$$
The definition of $\nu^j_n(e,\alpha)$ is simply $(e,\alpha)$.

We now make a functorial choice of pushouts along cofibrations in $\calc$ which preserves identity morphisms. This is guaranteed to exist by Proposition~\ref{prop:functorial_choice_of_pushouts_pres_id}. Anytime we form a pushout or write a pushout in $\calc$, we mean this choice. In particular, we form for each simplex $(e,\alpha)$ in $f/(m,A)$ the pushout
\begin{equation}  \label{equ:rho_n_as_pushout:quasicategorical}
\begin{array}{c}
\xymatrix{ \lambda_n^j(e,\alpha) \ar[r] \ar@{}[dr]|{\mathrm{p.o.}} \ar@{>->}[d] & \mu_n^j(e,\alpha) \ar@{>->}[d] \\
\nu_n^j(e,\alpha) \ar[r] & \rho_n^j(e,\alpha)}
\end{array}
\end{equation}
in $\calc^{\Delta[2] \times \Delta[n] \times \Delta[n]}$ object-wise using this choice. The top horizontal map in \eqref{equ:rho_n_as_pushout:quasicategorical} is defined completely analogously to \eqref{equ:top_horizontal_map_i} and \eqref{equ:top_horizontal_map_ii}. The left vertical map in \eqref{equ:rho_n_as_pushout:quasicategorical} is induced by the natural transformation on $\cale^{po}(\cala, \calc,\calb)$ sketched by
$$\begin{array}{c}
\xymatrix{A_j \ar@{>->}[d] \\ C_j \ar@{->>}[d] \\ B_j }
\end{array}
\begin{array}{c}
\xymatrix{\ar@{|->}[r] & }
\end{array}
\begin{array}{c}
\xymatrix{A_j \ar@{>->}[d] & A_j \ar@{>->}[l]_{\text{Id}} \ar@{=}[d]  \\ C_j \ar@{->>}[d] \ar@{}[dr]|{\text{p.o.}} & A_j \ar@{>->}[l] \ar[d]  \\ B_j & \ast \ar@{>->}[l] }
\end{array}
$$
where the square labelled ``p.o.'' is from the data of the vertical cofiber sequences, see \eqref{equ:cofiber_sequence_quasicategorical}. By Corollary~\ref{cor:main_corollary_about_pushouts}, these pushouts \eqref{equ:rho_n_as_pushout:quasicategorical} are compatible with face and degeneracy maps, and $\rho_n^j$ therefore satisfies the simplicial homotopy identities. We can also conclude that each $\rho^j_n(e,\alpha)$ is in $\mathfrak{s}_n\cale^{po}(\cala,\calc,\calb)$ by arguments exactly like those after diagram \eqref{equ:rho_n_as_pushout}. Thus we have a homotopy from $\text{Id}_{f/(m,A)}$ to $(\rho^0_n)_n$, depicted in \eqref{equ:rho_n+1_grid}.

The justification that the homotopy $\theta\co(\rho^{0}_n)_n \simeq \iota \circ r$ from Lemma~\ref{lem:HardLemmaMadeEasy} also works analogously in the quasicategorical case is contained in Examples~\ref{examp:nat_transf_induces_s_bullet_homotopy} and \ref{examp:nat_transf_induces_s_bullet_homotopy_quasicategorical}.

We have $\text{Id}_{f/(m,A)}\simeq (\rho^{0}_n)_n \simeq \iota \circ r$ and $r \circ \iota = \text{Id}_{\mathfrak{s}_\bullet \calb}$, so that $r$ is a weak homotopy equivalence of simplicial sets. Finally, $\mathfrak{s}_\bullet (s,q) \co \mathfrak{s}_\bullet \cale^{po}(\cala,\calc,\calb) \to \mathfrak{s}_\bullet \cala \times \mathfrak{s}_\bullet \calb$ is a weak homotopy equivalence of simplicial sets by Theorem~$\widehat{A}^*$, and also \eqref{equ:HardLemmaMadeEasy_quasicategorical} is a weak homotopy equivalence of simplicial sets by 3-for-2.
\end{proof}

Next we move from the object-set additivity for $\mathfrak{s}_\bullet^\infty$ in Lemma~\ref{lem:HardLemmaMadeEasy_quasicategorical} to the additivity of $(S_\bullet^\infty)_{\text{equiv}}$ for $\cale(\cala,\calc,\calb)$. The idea is due to Waldhausen \cite[page~336]{WaldhausenAlgKTheoryI}, but the implementation in quasicategories requires some discussion.

\begin{proposition}\label{prop:D_to_D(m,w)_is_Waldhausen_equivalence}
Let $\cald$ be a Waldhausen quasicategory and $\cald(m,w) \subseteq \cald^{\Delta[m]}$ the 0-full sub quasicategory on the maps $\Delta[m] \to \cald$ which send each edge of $\Delta[m]$ to an equivalence in $\cald$, as in Example~\ref{examp:diagram_Waldhausen_quasicategories}. Then the following hold.
\begin{enumerate}
\item \label{prop:D_to_D(m,w)_is_Waldhausen_equivalence:i}
The inclusion $i\co \cald \hookrightarrow \cald(m,w)$ defined by $x \mapsto (x=x=\cdots=x)$, or more precisely $$\big(\xymatrix@C=1pc{\Delta[n] \ar[r]^-d & \cald}\big) \xymatrix{ \ar@{|->}[r] &}  \Big(\xymatrix@C=1.5pc{ \Delta[m] \times \Delta[n] \ar[r]^-{pr_2} & \Delta[n] \ar[r]^-d  &  \cald} \Big),$$ is a Waldhausen equivalence.
\item \label{prop:D_to_D(m,w)_is_Waldhausen_equivalence:ii}
For fixed $n$ and $m$, we have a bijection of sets $\mathfrak{s}_n^\infty \cald(m,w) \cong \big[(S_n^\infty \cald)_\text{\rm equiv}\big]_m$ natural in $n$ and $m$.
\item \label{prop:D_to_D(m,w)_is_Waldhausen_equivalence:iii}
For fixed $m$, the inclusion
\begin{equation} \label{equ:D_to_D(m,w)_is_Waldhausen_equivalence:fixed_m}
\xymatrix{\mathfrak{s}_\bullet^\infty \cald \ar@/_1pc/@{^{(}->}[rr] \ar@{^{(}->}[r]^-{\mathfrak{s}_\bullet^\infty i} &
\mathfrak{s}_\bullet^\infty \cald(m,w) \ar[r]^-{\cong} &
\big[(S_\bullet^\infty \cald)_\text{\rm equiv}\big]_m
}
\end{equation}
is a weak homotopy homotopy equivalence of simplicial sets.
\item \label{prop:D_to_D(m,w)_is_Waldhausen_equivalence:iv}
The simplicial maps of \eqref{equ:D_to_D(m,w)_is_Waldhausen_equivalence:fixed_m} for varying $m$ assemble to a diagonal weak equivalence
\begin{equation} \label{equ:s_into_Sequiv_is_diag_we}
\xymatrix{ \mathfrak{s}_\bullet^\infty \cald \ar@{^{(}->}[r] & (S_\bullet^\infty \cald)_\text{\rm equiv} }
\end{equation}
of bisimplicial sets.
\end{enumerate}
\end{proposition}
\begin{proof} \leavevmode
\begin{enumerate}
\item
We prove directly that the inclusion $i\co \cald \hookrightarrow \cald(m,w)$ is an equivalence of quasicategories by identifying the inclusion $\cald \hookrightarrow \cald^{\Delta[m]}$ with $\cald^{N\rho}$ in an adjunction that restricts to an adjoint equivalence on $\cald(m,w)$.  \\
Consider the adjunction of categories on the left below, where $\iota$ is the inclusion of 0 and $\rho$ is the projection to 0.
\begin{equation} \label{equ:D_to_D(m,w)_is_Waldhausen_equivalence:adjunction}
\begin{array}{c}
\xymatrix@C=2pc{[0] \ar@/^.75pc/[r]^{\iota}  \ar@{}[r]|{\bot} & \ar@/^.75pc/[l]^{\rho} [m]}
\end{array}
\xymatrix@C=3pc{\ar@{|->}[r]^{\cald^{(-)} \circ N } &  }
\begin{array}{c}
\xymatrix@C=2pc{\cald^{\Delta[0]} \ar@{<-}@/^.75pc/[r]^{\cald^{N\iota}}  \ar@{}[r]|{\top} & \ar@{<-}@/^.75pc/[l]^{\cald^{N\rho}} \cald^{\Delta[m]}}
\end{array}
\end{equation}
We have trivial unit $\text{Id}_{[0]}=\rho \iota$ and counit the natural transformation $\varepsilon\co \iota \rho =\text{const}_0 \Rightarrow \text{Id}_{[m]}$ which in component $j$ is the unique map $0 \to j$. Applying nerve and exponentiation base $\cald$ to the adjunction of categories on the left yields the adjunction of quasicategories on the right (recall Proposition~\ref{prop:exp_is_simplicial}~\ref{prop:exp_is_simplicial:iii} and \ref{prop:exp_is_simplicial:iv} and Corollary~\ref{cor:exponentiating_simplicial_homotopies}). The adjunction of quasicategories on the right also has a trivial unit by functoriality of nerve and $\cald^{(-)}$.

We prove that the natural transformation $\cald^{N\varepsilon}\co \cald^{\Delta[m]} \times \Delta[1] \to \cald^{\Delta[m]}$ is a natural equivalence when restricted to $\cald(m,w) \subseteq \cald^{\Delta[m]}$ by computing the component of $\cald^{N\varepsilon}$ at an object $h\co \Delta[m] \to \cald$ of $\cald(m,w)$, recall Corollary~\ref{cor:nat_transf_is_nat_equiv_iff_comps_equivs}~\ref{item(i)}$\Leftrightarrow$\ref{item(iv)}. The degeneracy $s_0(h)$ is $h \circ \text{pr}_1\co \Delta[m] \times \Delta[1] \to \cald$, so the component of $\cald^{N\varepsilon}$ at $h$ is
\begin{equation} \label{equ:checking_CNalpha_natural_equivalence_i}
\cald^{N\varepsilon}(s_0(h), \text{Id}_{[1]})=\text{comp}_1(h \circ \text{pr}_1, N\varepsilon)\in \big(\cald^{\Delta[m]}\big)_1,
\end{equation}
as we see from equation \eqref{equ:making_exponentiation_base_C_simplicial_formula} by taking $A=B=\Delta[m]$, $f=N\varepsilon$, $g=h \circ \text{pr}_1$, $q=\text{Id}_{[1]}$, and $m=1$ there (but not here). To see the component \eqref{equ:checking_CNalpha_natural_equivalence_i} is an equivalence in $\cald^{\Delta[m]}$, we check that the components of \eqref{equ:checking_CNalpha_natural_equivalence_i} are equivalences in $\cald$. For this, we apply formula \eqref{equ:enriched_composition_formula} with $a$ the degeneracy $s_0(j)=\big(j=j\big)$ of any vertex $j$ of $\Delta[m]$ and $p=\text{Id}_{[1]}$.
\begin{equation} \label{equ:checking_CNalpha_natural_equivalence_ii}
\aligned
\text{comp}_1(h \circ \text{pr}_1, N\varepsilon)(s_0(j),\text{Id}_{[1]})&=\big(h \circ \text{pr}_1\big)_1((N\varepsilon)_1(s_0(j),\text{Id}_{[1]}),\text{Id}_{[1]})\\
&=\big(h \circ \text{pr}_1\big)_1(0\to j,\text{Id}_{[1]})\\
&=h(0\to j)
\endaligned
\end{equation}
But $h$ is an object of $\cald(m,w)$, so $h(0 \to j)$ is an equivalence in $\cald$. Thus $\cald^{N\varepsilon}$ is a natural equivalence on $\cald(m,w)$, and the restriction of the right adjunction in equation  \eqref{equ:D_to_D(m,w)_is_Waldhausen_equivalence:adjunction} to $\cald(m,w)\subseteq \cald^{\Delta[m]}$ is an adjoint equivalence. The restriction of $\cald^{N\rho}$, which is an equivalence of quasicategories, is $i\co\cald \hookrightarrow \cald(m,w)$.

It is fairly clear that $i\co\cald \hookrightarrow \cald(m,w)$ is exact in the sense of Definition~\ref{def:exact_functor}, and reflects cofibrations, so $i$ is a Waldhausen equivalence by Proposition~\ref{prop:exact_equivalence_is_Waldhausen_equivalence:qcats}.
\item
In $\mathfrak{s}_n^\infty$ and $S_n^\infty$ recall that any zero objects are on the main diagonal, so in $\mathfrak{s}_n^\infty \cald(m,w)$ any sequence $\ast\cong \ast' \cong \cdots \cong \ast^{(m)}$ can be a diagonal entry. The rest follows from the standard adjunction formulae $$\mathbf{SSet}(N\text{Ar}[n], \cald^{\Delta[m]})\cong\mathbf{SSet}(N\text{Ar}[n]\times\Delta[m], \cald)\cong\mathbf{SSet}(\Delta[m], \cald^{N\text{Ar}[n]}).$$
\item
From \ref{prop:D_to_D(m,w)_is_Waldhausen_equivalence:i} and Example~\ref{examp:nat_transf_induces_s_bullet_homotopy_quasicategorical} we see that $\mathfrak{s}_\bullet^\infty i$ is a simplicial homotopy equivalence.
\item
This follows from \ref{prop:D_to_D(m,w)_is_Waldhausen_equivalence:iii} by the Realization Lemma: if a map of bisimplicial sets is a level-wise weak homotopy equivalence of simplicial sets, then its diagonal is a weak homotopy equivalence of simplicial sets.
\end{enumerate}
\end{proof}

\begin{remark}[Replacing $\cald(m,w)$ by $\cald^{J[m]}$] \label{rem:replacing_D(m,w)_with_D^J[m]}
Recall $J[m]$ denotes the nerve of the groupoid with objects $0,1, \dots, m$ and a unique isomorphism from any object to another, and that a morphism $J[m] \to \cald$ is the same as a map $\Delta[m] \to \cald$ with all edges equivalences, and with selected coherent inverses (see Proposition~\ref{prop:extending_an_n-simplex_of_equivalences}). The equivalence of quasicategories $\Delta[0] \hookrightarrow  J[m]$, which comes from an equivalence of categories, exponentiates to an equivalence of quasicategories $\cald^{J[m]} \to \cald$ by Proposition~\ref{prop:restriction_morphisms}~\ref{prop:restriction_morphisms:i} or by Corollary~\ref{cor:exponentiating_simplicial_homotopies}. The restriction $\cald^{J[m]} \to \cald(m,w)$ is an equivalence of quasicategories by the 3-for-2 property of equivalences of quasicategories.
$$\xymatrix@C=2.5pc{\cald^{J[m]} \ar[r] \ar@/_1pc/[rr]_-{\text{equiv}} & \cald(m,w) \ar[r]^-{\text{equiv}} & \cald}$$
\end{remark}

\begin{remark}[Derivator $K$-Theory]
In the context of $K$-theory of {\it pointed right derivators}, Muro-Raptis proved a version of Proposition~\ref{prop:D_to_D(m,w)_is_Waldhausen_equivalence}~\ref{prop:D_to_D(m,w)_is_Waldhausen_equivalence:iv} where the weak equivalences are not required to be the weak isomorphisms, see Section 4.2 and Proposition 4.2.1 of \cite{MuroRaptis_KforPointedRightDerivators}. The idea is that the homotopy pushouts in the $S_\bullet$ construction in their context already contain enough information to make the inclusion of 0-simplices into $S_\bullet$ into a weak equivalence in $K$-theory. Their work \cite{MuroRaptis_KforPointedRightDerivators} is a follow-up on \cite{MuroRaptis_SameDerKDifferentWaldK}, in which they give an example of two differential graded algebras with the same {\it triangulated derivator $K$-theory} but different Waldhausen $K$-groups. For pointed right derivators that arise from good enough Waldhausen categories, the $K$-theory defined in \cite{MuroRaptis_KforPointedRightDerivators} agrees with Waldhausen $K$-theory. The $K$-theory of pointed right derivators is not invariant under general derivator equivalence, while Maltsiniotis' original definition of $K$-theory for triangulated derivators in \cite{Maltsiniotis_KOfTriangulatedDerivator} ``is the best approximation to Waldhausen $K$-theory by a functor that is invariant under equivalences of derivators'' as proved in \cite{MuroRaptis_KforPointedRightDerivators}. We also mention that Additivity for derivator $K$-theory of triangulated derivators (using Maltsiniotis' original definition) was proved by Cisinski-Neeman in \cite{CisinskiNeeman_AdditivityForTriangulatedDerivators}.
\end{remark}

\begin{theorem}[Waldhausen $(S_\bullet^\infty)_\text{equiv}$ Additivity for $\cale(\cala,\calc,\calb)$] \label{thm:Additivity_for_quasicategories}
Let $\mathcal{C}$ be a Waldhausen quasicategory and $\cala$ and $\calb$ sub Waldhausen quasicategories. Then $(s,q)$ induces a diagonal weak equivalence
\begin{equation} \label{equ:Additivity_for_quasicategories}
\xymatrix{S^\infty_\bullet(s,q)_\mathrm{equiv}\co (S^\infty_\bullet \cale(\cala,\calc,\calb))_\mathrm{equiv} \ar[r] & (S^\infty_\bullet \cala)_\mathrm{equiv} \times (S^\infty_\bullet \calb)_\mathrm{equiv}}
\end{equation}
of bisimplicial sets.
\end{theorem}
\begin{proof}
Consider the commutative diagram of bisimplicial set maps below.
$$\xymatrix@C=5pc@R=3pc{\mathfrak{s}^\infty_\bullet \cale(\cala,\calc,\calb) \ar[r]^{\mathfrak{s}^\infty_\bullet(s,q)} \ar@{^{(}->}[d]_{\eqref{equ:s_into_Sequiv_is_diag_we}}^{\text{w.e.}} & \mathfrak{s}^\infty_\bullet \cala \times \mathfrak{s}^\infty_\bullet \calb \ar@{^{(}->}[d]^{\eqref{equ:s_into_Sequiv_is_diag_we}}_{\text{w.e.}} \\
(S^\infty_\bullet \cale(\cala,\calc,\calb))_\mathrm{equiv} \ar[r]_-{S^\infty_\bullet(s,q)_\mathrm{equiv}} & (S^\infty_\bullet \cala)_\mathrm{equiv} \times (S^\infty_\bullet \calb)_\mathrm{equiv} } $$
The left and right vertical maps are diagonal weak equivalences by Proposition~\ref{prop:D_to_D(m,w)_is_Waldhausen_equivalence}~\ref{prop:D_to_D(m,w)_is_Waldhausen_equivalence:iv}. The top horizontal map is a diagonal weak equivalence by Lemma~\ref{lem:HardLemmaMadeEasy_quasicategorical} and the Realization Lemma. Thus, the bottom map, which is \eqref{equ:Additivity_for_quasicategories}, is a diagonal weak equivalence by the 3-for-2 property.
\end{proof}

%
%

\section{The $K$-Theory Space and $K$-Theory Spectrum of a Waldhausen Quasicategory} \label{sec:K-Theory_Space_and_Spectrum}

We introduce the Waldhausen $K$-theory space and spectrum of a Waldhausen quasicategory, and discuss how to make them 1-functorial. Throughout this section we use the notion of {\it exact functor} introduced in Definition~\ref{def:exact_functor}. Recall that Waldhausen quasicategories do not have a distinguished zero object, so we must carefully treat basepoints and homotopy independence of choices when constructing $K$-theory spaces and spectra.

\begin{definition}[$K$-Theory Space of a Waldhausen Quasicategory]  \label{def:K-theory_space_of_Waldhausen_Qcat}
Let $(\calc, co\calc)$ be a Waldhausen quasicategory and select a zero object $\ast$. The {\it Waldhausen $K$-theory space of $(\calc, co\calc,\ast)$} is
\begin{equation} \label{equ:K-theory_space_for_quasicategories}
K(\calc, co\calc,\ast):=\Omega \big\vert \big(S_\bullet^\infty \calc\big)_{\text{equiv}} \big\vert,
\end{equation}
where the loops are based at the selected zero object $\ast$ of $\calc$. Usually we simply write $K(\calc)$ for $K(\calc, co\calc,\ast)$.
\end{definition}

\begin{remark}[Making the $K$-Theory Space 1-Functorial]
Choosing a different zero object gives a homotopy equivalent $K$-theory space (recall that the diagonal entries of $S_n^\infty \calc$ are any zero objects of $\calc$, so changing the selected zero object $\ast$ has no effect on the space $\vert (S_\bullet^\infty \calc)_{\text{equiv}} \vert$ in \eqref{equ:K-theory_space_for_quasicategories}, rather only an effect on $\Omega \vert (S_\bullet^\infty \calc)_{\text{equiv}} \vert$\,). The algebraic $K$-theory space $K(-)$ is not 1-functorial in exact functors because exact functors are not required to preserve a selected zero object. However, the algebraic $K$-theory space construction is a strict 1-functor on the 1-category $\mathbf{QWald}^{\text{zero}}$ of Waldhausen quasicategories with a distinguished zero object and morphisms the exact functors strictly preserving the distinguished zero objects.
\end{remark}

To define the structure maps $\Sigma\bfK(\calc)_n \to \bfK(\calc)_{n+1}$ of the $K$-theory spectrum associated to a Waldhausen quasicategory $\calc$ (following Waldhausen's original construction \cite[pages 329--330]{WaldhausenAlgKTheoryI} for categories), we need not only a selection of a zero object $\ast_\calc$, but also a selection of natural transformations $\alpha\co \text{Id}_\calc \Rightarrow \ast_\calc$ and $\beta \co \ast_\calc \Rightarrow \text{Id}_\calc$ with component $\ast_\calc=\ast_\calc$ at the object $\ast_\calc$.

\begin{definition}[Pointed Waldhausen Quasicategory and Pointed Exact Functor]
A {\it pointed} Waldhausen quasicategory $(\calc,co\calc,\ast_\calc,\alpha^\calc,\beta^\calc)$ is a Waldhausen quasicategory $(\calc,co\calc)$ with a distinguished zero object $\ast_\calc \in \calc_0$ and a selection of two natural transformations $\alpha^\calc\co \text{Id}_\calc \Rightarrow \ast_\calc$ and $\beta^\calc \co \ast_\calc \Rightarrow \text{Id}_\calc$, both with component $\ast_\calc=\ast_\calc$ at the object $\ast_\calc$. \\ A {\it pointed} exact functor $f\co (\calc,co\calc,\ast_\calc,\alpha^\calc,\beta^\calc) \to (\cald,co\cald,\ast_\cald,\alpha^\cald,\beta^\cald)$ of pointed Waldhausen quasicategories is an exact functor strictly preserving the zero object and natural transformations, that is
$$f(\ast_\calc)=\ast_\cald, \;\;\;\;\; f\circ \alpha^\calc =\alpha^\cald \circ (f \times \text{Id}_{\Delta[1]}), \;\;\;\;\; \text{ and }  \;\;\;\;\; f\circ \beta^\calc =\beta^\cald \circ (f \times \text{Id}_{\Delta[1]}).$$
\end{definition}

\begin{remark} \label{rem:D_pointed_implies_SnD_pointed}
If $\cald$ is a pointed Waldhausen category, then $S_n^\infty \cald$ is a pointed Waldhausen category with selected zero object $\ast_{S_n^\infty \cald}$ the grid of $\ast_\cald$ and equalities, and the natural transformations are $S_n^\infty \alpha$ and $S_n^\infty \beta$ (recall $S_n^\infty$ is an $(\infty,2)$-functor from Remark~\ref{rem:infinity-2-category_of_Waldhausen_quasicategories}.)
\end{remark}

\begin{definition}[0-th Structure Map of $\bfK(\cald)$]
Let $\cald$ be a pointed Waldhausen category. The {\it 0-th structure map of the $K$-theory spectrum $\mathbf{K}(\cald)$} is the composite
\begin{equation} \label{equ:0th_structure_map}
\xymatrix@C=3pc{\sum\big\vert \cald_\text{equiv} \big\vert \ar@{^{(}->}[r] &
\sum\big\vert \big(S_1^\infty\cald\big)_\text{equiv} \big\vert \ar@{^{(}->}[r] &
\big\vert \big(S_\bullet^\infty\cald\big)_\text{equiv} \big\vert }.
\end{equation}
The first map of \eqref{equ:0th_structure_map} is the reduced suspension of the realization of the quasicategory map
\begin{equation} \label{equ:first_part_of_K-theory_spectrum_structure_map_on_qcatlevel}
\xymatrix@C=3pc{\cald_\text{equiv} \ar@{^{(}->}[r] &
\big(S_1^\infty\cald\big)_\text{equiv} }
\end{equation}
$$d \xymatrix@C=3pc{\ar@{|->}[r] & } \begin{array}{c} \xymatrix@C=1pc@R=1pc{\ast_\cald \ar[r]^{\beta_d^\cald} & d \ar[d]^{\alpha_d^\cald} \\ & \ast_\cald} \end{array}.$$
More precisely, \eqref{equ:first_part_of_K-theory_spectrum_structure_map_on_qcatlevel} sends an $n$-simplex $\gamma \co \Delta[n] \to \cald_\text{equiv}$ to the $n$-simplex $N\text{Ar}[1]\times \Delta[n] \to \cald$ which on the top face and side face respectively are
$$\xymatrix@C=4.5pc{\Delta[n] \times \Delta[1] \ar[r]^-{\beta^\cald \circ (\gamma \times \text{Id}_{\Delta[1]})} & \cald} \hspace{.5in} \text{and} \hspace{.5in} \xymatrix@C=4.5pc{\Delta[n] \times \Delta[1] \ar[r]^-{\alpha^\cald \circ (\gamma \times \text{Id}_{\Delta[1]})} & \cald}.$$
For an explanation of the second inclusion in \eqref{equ:0th_structure_map}, see \cite[Section~8.4]{RognesTextbook}.
\end{definition}

\begin{definition}[$K$-Theory Spectrum of a Pointed Waldhausen Quasicategory]  \label{def:K-theory_spectrum_of_Waldhausen_Qcat}
Let $\calc$ be a pointed Waldhausen quasicategory. The {\it Waldhausen $K$-theory spectrum of $\calc$} has $n$-th based space
\begin{equation} \label{equ:nth_space_of_K-theory_spectrum}
\bfK(\calc,co\calc,\ast_\calc)_n:= \Big| \Big(\big[S_\bullet^\infty\big]^n \calc \Big)_\mathrm{equiv}\Big|
\end{equation}
with basepoint $\ast_\calc$, where $n \geq 0$ and $\big[S_\bullet^\infty\big]^n\calc\co \Delta^{\times n} \to \mathbf{QWald}$ is the multisimplicial Waldhausen quasicategory obtained by $n$ applications of $S_\bullet^\infty$ to $\calc$. The structure map
$$\xymatrix@C=3pc{\Sigma\bfK(\calc,co\calc,\ast_\calc)_n \ar[r] & \bfK(\calc,co\calc,\ast_\calc)_{n+1}}$$
is the composite
$$\xymatrix@C=3pc{\sum\Big| \Big(\big[S_\bullet^\infty\big]^n \calc \Big)_\mathrm{equiv}\Big| \ar@{^{(}->}[r] &
\sum\Big| \Big(S_1^\infty \big[S_\bullet^\infty\big]^{n} \calc \Big)_\mathrm{equiv}\Big| \ar@{^{(}->}[r] &
\Big| \Big(S_\bullet^\infty \big[S_\bullet^\infty\big]^{n} \calc \Big)_\mathrm{equiv}\Big| }$$
where the first map on the quasicategory level is \eqref{equ:first_part_of_K-theory_spectrum_structure_map_on_qcatlevel} in each fixed multisimplicial degree, making use of Remark~\ref{rem:D_pointed_implies_SnD_pointed} to form \eqref{equ:first_part_of_K-theory_spectrum_structure_map_on_qcatlevel}. Usually we write $\bfK(\calc)$ for the spectrum $\bfK(\calc,co\calc,\ast)$, and we notationally suppress the dependence of the structure map on $\alpha$ and $\beta$.
\end{definition}

\begin{remark}[1-Functoriality of the $K$-Theory Spectrum for Pointed Exact Functors] \label{rem:making_K-theory_spectrum_functorial}
The {\it underlying space} of $\bfK(\calc)_n$ in \eqref{equ:nth_space_of_K-theory_spectrum} is 1-functorial on Waldhausen quasicategories and exact functors, without any selections of zero objects nor requirements on exact functors, as the underlying space of \eqref{equ:nth_space_of_K-theory_spectrum} makes no use of $\ast_\calc$, $\alpha$, $\beta$. However, the structure maps of the spectrum $\mathbf{K}(\calc)$ are not compatible with arbitrary exact functors because \eqref{equ:first_part_of_K-theory_spectrum_structure_map_on_qcatlevel} is not compatible with arbitrary exact functors (for instance, $f(\beta_c^\calc)$ may not be $\beta_{f(c)}^\cald$).  Nevertheless, the algebraic $K$-theory spectrum $\bfK(-)$ is a strictly 1-functorial for pointed exact functors. That is, $\bfK(-)$ is a strict 1-functor on the 1-category $\mathbf{QWald}^{\text{pointed}}$ of pointed Waldhausen quasicategories and pointed exact functors.
\end{remark}

\begin{proposition}[Replacement of an Unpointed Exact Functor by a Pointed Exact Functor] \label{prop:exact_is_natequivalent_to_pointed_exact}
Let $(\calc,co\calc,\ast_\calc,\alpha^\calc,\beta^\calc)$ and $(\cald,co\cald,\ast_\cald,\alpha^\cald,\beta^\cald)$ be pointed Waldhausen quasicategories. Let $f\co (\calc,co\calc) \to (\cald,co\cald)$ be any exact functor. Then $f$ is naturally equivalent to a pointed exact functor $g\co (\calc,co\calc,\ast_\calc,\alpha^\calc,\beta^\calc) \to (\cald,co\cald,\ast_\cald,\alpha^\cald,\beta^\cald)$ that agrees with $f$ away from $\ast_\calc$.
\end{proposition}
\begin{proof}
Let $\calc \backslash \ast_\calc$ be the 0-full sub quasicategory of $\calc$ on $\calc_0 \backslash \ast_\calc$. Any lift in the diagram below constructs $g$ and a natural equivalence from $f$ to $g$.
$$\xymatrix@C=5pc{\Big(\calc \times \{0\}\Big)\; \bigcup \; \Big((\calc\backslash \ast_\calc) \cup \text{im}(\alpha) \cup \text{im}(\beta) \Big)\times J[1] \ar[r] \ar@{^{(}->}[d]_{\text{Joyal trival cofibration}} & \cald \ar[d]^{\text{Joyal fibration}} \\ \calc \times J[1] \ar[r] \ar@{-->}[ur] & \ast }$$
The left vertical map is a Joyal trivial cofibration (=mono map that is also a Joyal equivalence) because
$\{0\} \hookrightarrow J[1]$ is a monomorphic equivalence of quasicategories, and
$$\xymatrix{(\calc\backslash \ast_\calc) \cup \text{im}(\alpha) \cup \text{im}(\beta) \ar@{^{(}->}[r] & \calc}$$
is a monomorphism. This implies the left vertical map is a Joyal trivial cofibration because the Joyal model structure is Cartesian closed, see \cite[Theorem~6.6]{JoyalQuadern} and \cite[Recall~2.2.8]{RiehlVerity}, and \cite[pages 232--235]{JoyalQuadern} for similar statements for related classes of maps. The right vertical map is a Joyal fibration, since $\cald$ is a quasicategory, and quasicategories are exactly the fibrant objects in the Joyal model structure.

To justify that $g$ is compatible with the natural transformations, we next explain what the top horizontal map is. On $\calc \times \{0\}$ it is simply $f$. On $\calc \backslash \ast_\calc \times J[1]$ it is $f\vert_{\calc \backslash \ast_\calc}$, independent of $J[1]$. On the simplicial set $\text{im}(\alpha) \times J[1]$, it assigns to the commutative square $(\alpha_A^\calc \co A \to \ast_\calc,\; 0\to 1)$ with $A \neq \ast_\calc$, the following commutative square.
$$\xymatrix{f(A) \ar[r] \ar[d]_{f(\alpha^\calc_A)} \ar[dr]^{\alpha^\cald_A} & g(A) \ar[d]^{\alpha^\cald_A=g(\alpha^\calc_A)} \\ f(\ast_\calc) \ar[r]_{\text{equiv}} & \ast_\cald = g(\ast_\calc) }$$
On the simplicial set $\text{im}(\beta) \times J[1]$, there is a similar definition.

All that remains is to justify why a functor $g$, naturally equivalent to an exact functor $f$, is also exact. The functor $g$ sends zero objects to zero objects because any object equivalent to a zero object is also a zero object, and $g$ sends cofibrations to cofibrations by homotopy repleteness of $co\cald$, see \eqref{equ:diagram_for_homotopy_repleteness}. Pushout squares are mapped to pushout squares because pushouts are invariant under equivalence.

\end{proof}

\begin{remark}[A Spectrum Map $\bfK(f)$ for an exact Kan Fibration $f$]
If one is working with a specific exact functor $f\co \calc \to \cald$ that is a Kan fibration, then one can obtain a (strict) spectrum map $\bfK(f)$ without alteration to $f$ by making appropriate choices of zero objects and natural transformations in $\calc$ and $\cald$. In particular, we select any zero object $\ast_\calc$ of $\calc$ and define $\ast_\cald := f(\ast_\calc)$. Then we select $\alpha^\cald$ and $\beta^\cald$ which have $\ast_\cald=\ast_\cald$ at the component of $\ast_\cald$. Then $\alpha^\calc$ is defined to be any lift in the diagram below.
\begin{equation}
\xymatrix@C=4pc{(\calc \times \{0\}) \cup (\ast_\calc \times \Delta[1]) \ar[rr]^-{\text{Id}_\calc \, \cup \, (\ast_\calc)} \ar[d]_{\text{anodyne}} & & \calc \ar[d]^{f \text{ Kan fibration}} \\
\calc \times \Delta[1] \ar[r]_-{f\times \text{Id}_{\Delta[1]}} \ar@{-->}[rru]^{\alpha^\calc} & \cald \times \Delta[1] \ar[r]_-{\alpha^\cald} & \cald }
\end{equation}
The left vertical map is anodyne because $\{0\} \hookrightarrow \Delta[1]$ is anodyne and $\ast_\calc \hookrightarrow \calc$ is a monomorphism, see \cite[Theorem~2.17]{JoyalQuadern}. A lift $\alpha_\calc$ exists because of the Kan model structure on $\mathbf{SSet}$ by Quillen. A natural transformation $\beta^\calc$ compatible with $f$ and $\beta^\cald$ is constructed similarly.
\end{remark}

\begin{remark}
It is not a major loss of generality to work with pointed Waldhausen quasicategories and pointed exact functors because any exact functor is naturally equivalent to a pointed one by Proposition~\ref{prop:exact_is_natequivalent_to_pointed_exact}. Also, choosing different $\alpha^\calc$ and $\beta^\calc$ in the $K$-theory spectrum yields homotopic structure maps. Choosing a different zero object yields a homotopy equivalent $K$-theory spectrum.
\end{remark}

\begin{remark}
Even if an exact functor $f\co \calc \to \cald$ is unpointed, it still induces continuous maps $\bfK(f)_n\co \bfK(\calc)_n \to \bfK(\cald)_n$ via formula \eqref{equ:nth_space_of_K-theory_spectrum}, though these maps are not compatible with the spectrum structure maps. Nevertheless, these $\bfK(f)_n$ are levelwise homotopic to continuous maps that do assemble to a morphism of spectra, see Proposition~\ref{prop:exact_is_natequivalent_to_pointed_exact}.
\end{remark}

As a consequence of $(S_\bullet^\infty)_\text{equiv}$ Additivity (see Theorem~\ref{thm:Additivity_for_quasicategories}), the $K$-theory spectrum is an $\Omega$-spectrum beyond the 0-th term, and the $K$-theory space is an infinite loop space. However, in this paper we do not make use of either of these facts, so we do not prove them. For a different approach to defining the spectrum structure maps via a section of a trivial fibration $S_1^\infty\calc \to \calc$, and an approach to proving the spectrum is an $\Omega$-spectrum beyond the 0-th term, see Definition~3.2.6 and Section~5.2 of \cite{PieperMastersThesis}.

Of course, $K$ and $\mathbf{K}$ extend to $(\infty,1)$-functors, as will be treated in \cite{FioreApproximation}.

\section{Quasicategorical Additivity and Split Exact Sequences of Waldhausen Quasicategories} \label{sec:split_exact_sequences_of_Waldhausen_quasicategories}

In this section, we develop the most basic properties of the notion of split exact sequences of Waldhausen quasicategories, associate a standard split exact sequence to $\cale(\cala,\calc,\calb)$ and $\cale^{po}(\cala,\calc,\calb)$ depending on choice of pointing for $\calc$, give sufficient conditions for a split exact sequence to be equivalent to a standard one, and prove $(S^\infty_\bullet)_\text{equiv}$ Additivity and $\bfK$ Additivity for split exact sequences equivalent to standard ones, such as any stable split exact sequence. We indicate when pointed is required, e.g. when dealing with spectra. Blumberg--Gepner--Tabuada discuss and relate various kinds of exact sequences of stable quasicategories, spectral categories, and triangulated categories in \cite[Section~5]{BlumbergGepnerTabuadaI}.

\begin{definition}[Exact Functor] \label{def:exact_functor}
Let $\calc$ and $\cald$ be Waldhausen quasicategories.
A functor $f\colon \calc \to \cald$ is called {\it exact} if it sends zero objects of $\calc$ to zero objects of $\cald$, it preserves cofibration sub quasicategories in the sense that $f(co\calc) \subseteq co \cald$, and it maps each pushout square along a cofibration to a pushout square along a cofibration.
\end{definition}

Note that by 1-fullness of the sub quasicategory of cofibrations, it suffices to check that $f$ maps cofibrations to cofibrations to conclude that $f(co\calc) \subseteq co \cald$. Note also that every functor of quasicategories maps equivalences to equivalences, so $f(\calc_\mathrm{equiv}) \subseteq \cald_\mathrm{equiv}$ by 1-fullness, so there is no need to require preservation of equivalences in the definition of exact functor.

\begin{definition}[Exact Sequence of Waldhausen Quasicategories] \label{def:exact_sequence_of_Waldhausen_quasicategories}
Let $\mathcal{A}$, $\cale$, and $\calb$ be Waldhausen quasicategories.
A sequence of exact morphisms
\begin{equation} \label{equ:exact_seq_of_quasicategories}
\xymatrix{\mathcal{A} \ar[r]^i & \mathcal{E} \ar[r]^f & \mathcal{B}}
\end{equation}
is called {\it exact} if
\begin{enumerate}
\item
the composite $f\circ i$ is naturally equivalent to some zero object of $\mathcal{B}$,
\item
the exact morphism $i\colon \mathcal{A}\to \mathcal{E}$ is fully faithful, and
\item \label{condition:iii:quasicategories}
the restricted morphism $f\vert_{\cale / \cala} \colon \cale / \cala \to \mathcal{B}$ is an equivalence of quasicategories. Here $\cale / \cala$ is the 0-full sub quasicategory of $\mathcal{E}$ on the objects $E \in \cale $ such that the simplicial set $\cale(i(A),E)$ is weakly contractible for all $A \in \cala$.
\end{enumerate}
\end{definition}

\begin{definition}[Split Exact Sequence of Waldhausen Quasicategories] \label{def:split-exact_sequence_of_Waldhausen_quasicategories}
An exact sequence of Waldhausen quasicategories as in equation \eqref{equ:exact_seq_of_quasicategories} is called {\it split} if there exist exact morphisms
$$\xymatrix{\mathcal{A}  & \mathcal{E}  \ar[l]_-j & \mathcal{B} \ar[l]_-g}$$
right adjoint to $i$ and $f$ respectively, such that the unit $\text{Id}_{\mathcal{A}} \to ji$ and the counit $fg \to \text{Id}_{\mathcal{B}}$ are natural equivalences.
\end{definition}

\begin{remark} \label{rem:f_and_g_are_inverses:quasicategories}
In a split exact sequence of Waldhausen quasicategories and exact functors, the functor $g$ is actually an inverse equivalence to $f\vert_{\cale / \cala}$.  We first see that $g$ goes into $\cale/\cala$, as there is a zig-zag of weak homotopy equivalences between $\cale(i(A),g(B))$ and $\calb(f(i(A)),B)=\calb(\ast,B)$, which is weakly equivalent to a point. The counit $(f\vert_{\cale / \cala})g \to \text{Id}_\calb$ is a natural equivalence by hypothesis, and the unit $\text{Id}_{\cale / \cala} \to g(f\vert_{\cale / \cala})$ is a natural equivalence since the left adjoint $f$ is fully faithful by Proposition~\ref{prop:g_fullyfaithful_iff_counit_equivalence:quasicategories} and Proposition~\ref{prop:qcat_equiv_iff_fullyfaithful_ess_surj}.
\end{remark}

\begin{remark}
The $(\infty,2)$-functor $S_n^\infty$ sends a split exact sequence of Waldhausen quasicategories to
a split exact sequence of Waldhausen quasicategories by Remarks~\ref{rem:infinity-2-category_of_Waldhausen_quasicategories} and \ref{rem:QWald2}.
\end{remark}

\begin{example}[Standard Split Exact Sequence from $\mathcal{E}(\cala,\calc,\calb)$] \label{examp:E(A,C,B)_quasicategorical}
Any Waldhausen quasicategory $\calc$ with selected sub Waldhausen quasicategories $\cala$ and $\calb$ with a selected common zero object produces a split exact sequence
\begin{equation} \label{equ:E(A,C,B)_quasicategorical}
\xymatrix{\mathcal{A} \ar[r]^-i & \mathcal{E}(\cala,\calc,\calb) \ar[r]^-q \ar@/^1.3pc/[l]^s & \mathcal{B} \ar@/^1.3pc/[l]^g}
\end{equation}
of Waldhausen quasicategories as follows.\footnote{If there is not a common zero object in $\cala$ and $\calb$, we may simply equivalently redefine $\cala$ and $\calb$ to include a common zero object.} Recall Notation~\ref{not:E(A,C,B)_quasicategorical} in which $\mathcal{E}(\cala,\calc,\calb)$, $s$, and $q$ are defined for quasicategories.

Select and fix a zero object $\ast$ of $\calc$ that is in both $\cala$ and $\calb$. Select and fix two natural transformations $\alpha, \beta \colon \Delta[1] \times \calc  \to \calc$ with $\alpha_0=\text{Id}_\calc$ and $\alpha_1=\ast$, and $\beta_0=\ast$ and $\beta_1=\text{Id}_\calc$. For $A\colon \Delta[n] \to \cala$ and $B\colon \Delta[n] \to \calb$, we define $i(A)$ and $g(B)$ to be the composites
$$\xymatrix@C=4pc{\Delta[2] \times \Delta[n] \ar[r]_{(s^0)_\ast \times A} \ar@/^1.5pc/[rr]^{i(A)} & \Delta[1] \times \cala \ar[r]_-{\alpha\vert\cala} & \calc}$$
$$\xymatrix@C=4pc{\Delta[2] \times \Delta[n] \ar[r]^{(s^1)_\ast \times B} \ar@/_1.5pc/[rr]_{g(B)} & \Delta[1] \times \calb \ar[r]^-{\beta\vert\calb} & \calc}$$
where $s^0\co [2] \to [1]$ is the surjective weakly increasing map that hits 0 twice, and $s^1\co [2] \to [1]$ is the surjective weakly increasing map that hits 1 twice. Then $i$ and $g$ are the analogues of \eqref{equ:canonical_split_exact_functors_categorical}.

Clearly $q \circ i = \ast$. Also $\text{Id}_\cala=si$, and $qg=\text{Id}_\calb$, so the unit of the adjunction $i \dashv s$ and the counit of the adjunction $q \dashv g$ are the identity natural transformations so $i$ and $g$ are fully faithful (see Proposition~\ref{prop:g_fullyfaithful_iff_counit_equivalence:quasicategories}). The counit and unit of these respective adjunctions can be defined as in Example~\ref{examp:E(A,C,B)}, and the triangle identities are also proved similarly (recall the discussion of adjunctions for quasicategories in Section~\ref{subsec:Adjunctions_and_Equivalences_between_Quasicategories}).

We next claim that the sub quasicategory $\cale / \cala$ is 0-full on the cofiber sequences of the form in equation \eqref{equ:BinE} with $\ast'$ a zero object of $\calc$ in $\cala$, the quotient map an isomorphism, and $B' \in \calb$. Suppose $A' \to C' \to B'$ is an object of $\cale/\cala$. Then $\cale(i(A),A'C'B')$ is weakly contractible, so $(\tau_1 \cale) (i(A),A'C'B')=\pi_0\cale(i(A),A'C'B')$ consists of a single point for every $A \in \cala_0$. Consideration of diagram \eqref{equ:one_map} in $\tau_1\cale$ shows that $A'$ is a terminal object of $\tau_1\cala$ (note that the commutative squares in $\tau_1\cale$ induced by cofiber sequences are not necessarily pushouts in $\tau_1\cale$). But $\ast$ is also terminal in $\tau_1\cala$, so there is an isomorphism in $\tau_1\cala$ between $A'$ and $\ast$, and therefore an equivalence in $\cala$ between $A'$ and $\ast$, so $A'$ is also a zero object of $\cala$. Finally, $C' \to B'$ is an equivalence, as it is a pushout in $\calc$ of an equivalence $A' \to \ast$.

From this we also see that $g\co \calb \to \cale/\cala$ is essentially surjective. The equivalence $C' \to B'$ is part of an equivalence in $\cale$ from $A' \to C' \to B'$ to $\ast \to B' = B'$, which is $g(B')$.

We now know $g\co \calb \to \cale/\cala$ is essentially surjective and fully faithful. Consequently, $g$ is an equivalence of quasicategories (see Proposition~\ref{prop:qcat_equiv_iff_fullyfaithful_ess_surj}). By the uniqueness of adjoints in a 2-category (here $\mathbf{SSet}^{\tau_1}$), and the fact that every equivalence is part of an adjoint equivalence (on either side), the left adjoint $q\vert_{\cale/\cala}\co \cale/\cala \to \calb$ to $g$ is also an equivalence of quasicategories, and we have shown that \eqref{equ:E(A,C,B)_quasicategorical} is a split exact sequence of quasicategories.
\end{example}

\begin{example}[Standard Split Exact Sequence from $\mathcal{E}^{po}(\cala,\calc,\calb)$] \label{examp:E'(A,C,B)_quasicategorical}
Recall Notation~\ref{not:E'(A,C,B)_quasicategorical} in which $\mathcal{E}^{po}(\cala,\calc,\calb)$, $s$, and $q$ are defined for quasicategories.

Select and fix a zero object $\ast$ of $\calc$ that is in both $\cala$ and $\calb$ (if there is not a common zero object, we may simply equivalently redefine $\cala$ and $\calb$ to include a common zero object). Select and fix two natural transformations $\alpha, \beta \colon \Delta[1] \times \calc \to \calc$ with $\alpha_0=\text{Id}_\calc$ and $\alpha_1=\ast$, and $\beta_0=\ast$ and $\beta_1=\text{Id}_\calc$. For $A\colon \Delta[n] \to \cala$ and $B\colon \Delta[n] \to \calb$, we define $i(A)$ and $g(B)$ to be the composites
$$\xymatrix@C=4pc{\Delta[1] \times \Delta[1] \times \Delta[n] \ar[r]_-{\text{pr}_1 \times A} \ar@/^2pc/[rr]^{i(A)} & \Delta[1] \times \cala \ar[r]_-{\alpha\vert\cala} & \calc }$$
$$\xymatrix@C=4pc{\Delta[1] \times \Delta[1] \times \Delta[n] \ar[r]^-{\text{pr}_2 \times B} \ar@/_2pc/[rr]_{g(B)} & \Delta[1] \times \calb \ar[r]^-{\beta\vert\calb} & \calc }$$
as pictured in \eqref{equ:canonical_split_exact_functors_categorical}.

Then one can prove
\begin{equation} \label{equ:E'(A,C,B)_quasicategorical}
\xymatrix{\mathcal{A} \ar[r]^-i & \mathcal{E}^{po}(\cala,\calc,\calb) \ar[r]^-q \ar@/^1.3pc/[l]^s & \mathcal{B} \ar@/^1.3pc/[l]^g}
\end{equation}
is a split exact sequence as in Example~\ref{examp:E(A,C,B)_quasicategorical}.
\end{example}

\begin{proposition}[Sufficient Conditions for Waldhausen Equivalence with Standard Split Exact Sequence] \label{prop:universal_split_exact_sequence:quasicategorical}
Suppose a split exact sequence of Waldhausen quasicategories and exact functors
$$\xymatrix{\mathcal{A} \ar[r]^-i & \mathcal{E} \ar[r]^-f \ar@/^1pc/[l]^j & \mathcal{B} \ar@/^1pc/[l]^g}$$
has the following three properties.
\begin{enumerate}
\item \label{item:i:prop:universal_split_exact_sequence:quasicategorical}
Each counit component $ij(E) \to E$ is a cofibration.
\item \label{item:i':prop:universal_split_exact_sequence:quasicategorical}
For each cofibration $E \rightarrowtail E'$ in $\cale$, the induced map $$E \cup_{ij(E)} ij(E') \to E'$$ is a cofibration in $\cale$.
\item \label{item:ii:prop:universal_split_exact_sequence:quasicategorical}
In every cofiber sequence in $\cala$ of the form
$A_0 \rightarrowtail A_1 \twoheadrightarrow \ast$, the first map is an equivalence.
\end{enumerate}
Then it is Waldhausen equivalent (see Definitions~\ref{def:Waldhausen_equivalence:qcats} and \ref{def:Waldhausen_equiv_of_sequences_quasicategorical} and Proposition~\ref{prop:Waldhausen_equiv_of_sequences_compatible_with_j's})
to a split exact sequence of the form \eqref{equ:E(A,C,B)_quasicategorical} in Example~\ref{examp:E(A,C,B)_quasicategorical}.
\end{proposition}
\begin{proof}
Let $\calc:=\cale$. Since $i\co \cala \to \calc$ is fully faithful by definition of exact sequence, its essential image $\cala'$ is the sub quasicategory of $\calc$ that is 0-full on objects equivalent to those in the image of $i$, and $i$ provides an equivalence to $\cala'$. Similarly, since the counit $fg \to \text{Id}_{\mathcal{B}}$ is a natural equivalence by split exactness, the right adjoint $g$ is fully faithful, its essential image is the sub quasicategory of $\calc$ that is 0-full on objects equivalent to those in the image of $g$, and $g$ provides an equivalence with $\calb'$. See diagram \eqref{equ:prop:universal_split_exact_sequence:essential_images}.

We would like to define a Waldhausen equivalence $\Phi\co \cale \to \cale(\cala',\calc,\calb')$. In the quasicategory $\cale^\cale$, first select a composite natural transformation of $ij \Rightarrow \text{Id}_\cale \Rightarrow gf$, that is, a composite of the $i\dashv j$ counit followed by the $f\dashv g$ unit. This composite $\Delta[2] \to \cale^\cale$ provides us with the desired functor $\Phi\co \cale \to \cale^{\Delta[2]}$, which makes the quasicategorical version of \eqref{equ:s_q_equivalence} strictly commute.

For any object $E \in \calc$, we claim that $\Phi(E)$
\begin{equation} \label{equ:quasicategorical_Phi(E)}
\xymatrix@C=3pc{ij(E) \ar@{>->}[r]^-{\text{counit}} \ar@/_1pc/[rr] & E \ar[r]^-{\text{unit}} & gf(E)}
\end{equation}
extends to a pushout square of the form \eqref{equ:cofiber_sequence_quasicategorical}. We select a zero object $\ast$ in $\calc$ and any morphisms $ij(E) \to \ast$ and $\ast \to gf(E)$, and draw the square below.
\begin{equation} \label{equ:pushout_to_be}
\begin{array}{c}
\xymatrix@R=3pc@C=3pc{ij(E) \ar@{>->}[r]^-{\text{counit}} \ar[d] \ar[dr] & E \ar[d]^-{\text{unit}} \\ \ast \ar@{>->}[r] & gf(E)}
\end{array}
\end{equation}
The upper triangle commutes (i.e. is a 2-simplex in $\cale$) by our selection of the composite, and the lower triangle commutes because $\cale(ij(E),gf(E))$ is weakly equivalent to a point, see Remark~\ref{rem:f_and_g_are_inverses:quasicategories}.

We claim that \eqref{equ:pushout_to_be} is a pushout square in the quasicategory $\calc=\cale$. Let $P$ be a pushout of the counit $ij(E) \rightarrowtail E$ along $ij(E) \to \ast$, as pictured on the left of diagram \eqref{equ:prop:universal_split_exact_sequence:pushouts}. There is a (homotopically unique) natural transformation from the $P$-square to \eqref{equ:pushout_to_be}. Applying $j$ to this natural transformation gives us the quasicategorical version of the right diagram in \eqref{equ:prop:universal_split_exact_sequence:pushouts}. Arguing as in that passage, using the 3-for-2 property of equivalences in $\calc$ and the fact that pushouts of equivalences are equivalences, we have $P \in \cale/\cala$. Then by Remark~\ref{rem:f_and_g_are_inverses:quasicategories}, there exists a $Q \in \calb$ and an equivalence $P \to g(Q)$. Proceeding as in diagram \eqref{equ:prop:universal_split_exact_sequence:pushout_image}, we obtain a natural equivalence between \eqref{equ:pushout_to_be} and a pushout square. Thus \eqref{equ:quasicategorical_Phi(E)} is indeed a cofiber sequence.
Hypothesis \ref{item:i':prop:universal_split_exact_sequence:quasicategorical} implies that $\Phi$ maps cofibrations to cofibrations. We can also conclude that $\Phi$ is exact by arguments as in the categorical case Proposition~\ref{prop:universal_split_exact_sequence} (using that $\Phi$ is an equivalence of quasicategories to argue that pushouts along cofibrations are mapped to pushouts along cofibrations).

One can show that the functor $\Phi$ is fully faithful.

To show $\Phi\colon \cale \to \cale(\cala',\calc,\calb')$ is essentially surjective it suffices to prove that the top row of the middle diagram in \eqref{equ:prop:universal_split_exact_sequence:Phi_essentially_surjective} is isomorphic in $\tau_1\cale$ to the bottom row of the middle diagram in \eqref{equ:prop:universal_split_exact_sequence:Phi_essentially_surjective} (every commutative cube in $\tau_1\cale$ comes from a ``commutative'' cube in $\cale$ as in Lemma~\ref{lem:commutative_squares_in_homotopy_category=commutative_squares_in_quasicategory}, so we can extend the three isomorphisms to get a natural equivalence of squares in $\cale$).
Recall from Proposition~\ref{prop:tau1_is_a_2-functor} that $\tau_1$ is a 2-functor, so that $i \dashv j$ and $f \dashv g$ induce ordinary adjunctions between the respective homotopy categories, and we can use the universality of the counit and unit on the level of homotopy categories. Therefore, morphisms $[m]$ and $[n]$ exist as in diagram \eqref{equ:prop:universal_split_exact_sequence:Phi_essentially_surjective}. Then $m$ and $n$ are seen to be equivalences, arguing via applications of $f$ and $j$ on the quasicategory level as in the passage after \eqref{equ:prop:universal_split_exact_sequence:Phi_essentially_surjective} with the word ``isomorphism'' replaced by ``equivalence'' (this is where hypothesis \ref{item:ii:prop:universal_split_exact_sequence:quasicategorical} on $\cala$ is used).

Lastly, $\Phi$ is a Waldhausen equivalence of quasicategories by Proposition~\ref{prop:exact_equivalence_is_Waldhausen_equivalence:qcats}.
\end{proof}

\begin{corollary}[Stable Split Exact Sequences are Standard]
Every split exact sequence of stable quasicategories, considered as Waldhausen quasicategories with all maps cofibrations, is equivalent to a standard split exact sequence of the form \eqref{equ:E(A,C,B)_quasicategorical} in Example~\ref{examp:E(A,C,B)_quasicategorical}. See Definition~\ref{def:stable_quasicategories} and Example~\ref{examp:stable_quasicategories_are_Waldhausen_quasicategories}.
\end{corollary}

\begin{theorem}[Waldhausen $\bfK$ Additivity for Standard Pointed Split Exact Sequences, Quasicategorical] \label{thm:K-Additivity_Pointed_Standard_Quasicategorical}
Let $\mathcal{C}$ be a pointed Waldhausen quasicategory and $\mathcal{A}$ and $\mathcal{B}$ sub pointed Waldhausen quasicategories, in particular $\ast_\calc=\ast_\cala=\ast_\calb$ and $\alpha^\cala$, $\alpha^\calb$ are restrictions of $\alpha^\calc$, and similarly for the $\beta$'s. Consider the standard split exact sequence in \eqref{equ:E(A,C,B)_quasicategorical}, which is consequently a pointed split exact sequence. Then
$$\xymatrix{\bfK(i)\vee\bfK(g)\colon \bfK(\cala) \vee \bfK(\calb) \ar[r] & \bfK(\cale(\cala, \calc, \calb))}$$
is a stable equivalence of spectra.
\end{theorem}
\begin{proof}
The argument is the same as in Theorem~\ref{thm:K-Additivity_for_standard_split_exact_sequence}.
We have an isomorphism $S^\infty_n \cale(\cala,\calc,\calb) \cong \cale( S^\infty_n \cala,S^\infty_n \calc,S^\infty_n \calb)$ of quasicategories.\footnote{Recall that the $S_\bullet^\infty$ construction in Definition~\ref{def:Sbulletinfinity} allowed any zero objects in the diagonal entries. This is important for the indicated isomorphism, for if we allowed only a selected zero object in the diagonal entries, then the inclusion
$S^\infty_n \cale(\cala,\calc,\calb) \hookrightarrow \cale( S^\infty_n \cala,S^\infty_n \calc,S^\infty_n \calb)$ would not be surjective, not even in level 0, as any equivalences $\ast \simeq \ast \simeq \ast$ are vertices in the codomain, but only $\ast = \ast = \ast$ is a vertex in the domain. In the 1-category case this difference does not occur, because there is only one map $\ast \to \ast$. } We also have $(S_\bullet^\infty)_\text{equiv}$ Additivity for ``subobject'' and ``quotient'' in Theorem~\ref{thm:Additivity_for_quasicategories}.
\end{proof}

\begin{corollary}[Waldhausen $(S^\infty_\bullet)_{\text{equiv}}$ Additivity and $\bfK$ Additivity for Split Exact Sequences
Waldhausen Equivalent to Standards] \label{cor:QCat_Additivity_GeneralAndSpectral}
Suppose a split exact sequence of Waldhausen quasicategories and exact functors
\begin{equation} \label{equ:cor:QCat_Additivity_GeneralAndSpectral:ses}
\xymatrix{\mathcal{A} \ar[r]^-i & \mathcal{E} \ar[r]^-f \ar@/^1pc/[l]^j & \mathcal{B} \ar@/^1pc/[l]^g}
\end{equation}
is Waldhausen equivalent to a standard split exact sequence of the form \eqref{equ:E(A,C,B)_quasicategorical}, for instance any split exact sequence satisfying the hypotheses of Proposition~\ref{prop:universal_split_exact_sequence:quasicategorical}.
Then the following hold.
\begin{enumerate}
\item \label{item:general_quasicategorical_additivity}
The map
$$\xymatrix{S_\bullet^\infty(j,f)_\mathrm{equiv} \colon \left( S_\bullet^\infty \mathcal{E} \right)_\mathrm{equiv} \ar[r] & \left( S_\bullet^\infty \mathcal{A} \right)_\mathrm{equiv} \times \left( S_\bullet^\infty \mathcal{B}\right)_\mathrm{equiv} }$$
is a diagonal weak equivalence of bisimplicial sets.
\item \label{item:spectral_quasicategorical_additivity}
Suppose now additionally that \eqref{equ:cor:QCat_Additivity_GeneralAndSpectral:ses} is a pointed split exact sequence. Then it is also Waldhausen equivalent to a pointed standard split exact sequence, and the functors $i$ and $g$ induce a stable equivalence of $K$-theory spectra
$$\xymatrix{\bfK(i)\vee\bfK(g)\colon \bfK(\cala) \vee \bfK(\calb) \ar[r] & \bfK(\cale).}$$
\end{enumerate}
\end{corollary}
\begin{proof}
The proof is just like the proof of Corollary~\ref{cor:Additivity_GeneralAndSpectral}, where we now use Theorem~\ref{thm:Additivity_for_quasicategories}, Corollary~\ref{cor:Waldhausen_equivalences_induce_weak_equivalences:qcats}, and Theorem~\ref{thm:K-Additivity_Pointed_Standard_Quasicategorical}, in combination with 3-for-2.

To conclude in \ref{item:spectral_quasicategorical_additivity} that Waldhausen equivalence of a pointed split exact sequence with an unpointed standard split exact sequence implies Waldhausen equivalence with a pointed standard split exact sequence, one first alters the unpointed standard split exact sequence to make it pointed standard, and then replaces the three Waldhausen equivalences using Proposition~\ref{prop:exact_is_natequivalent_to_pointed_exact}.
\end{proof}

\begin{corollary}[Additivity for Stable Quasicategories] \label{cor:stable_qcats_are_example_for_main_theorem}
Suppose
\begin{equation} \label{equ:cor:stable_qcats_are_example_for_main_theorem:stable_ses}
\xymatrix{\mathcal{A} \ar[r]^-i & \mathcal{E} \ar[r]^-f \ar@/^1pc/[l]^j & \mathcal{B} \ar@/^1pc/[l]^g}
\end{equation}
is a split exact sequence of stable quasicategories and exact functors. Then
\begin{enumerate}
\item
The map
$$\xymatrix{S_\bullet^\infty(j,f)_\mathrm{equiv} \colon \left( S_\bullet^\infty \mathcal{E} \right)_\mathrm{equiv} \ar[r] & \left( S_\bullet^\infty \mathcal{A} \right)_\mathrm{equiv} \times \left( S_\bullet^\infty \mathcal{B}\right)_\mathrm{equiv} }$$
is a diagonal weak equivalence of bisimplicial sets.
\item
If the sequence \eqref{equ:cor:stable_qcats_are_example_for_main_theorem:stable_ses} is additionally pointed, then the pointed exact functors $i$ and $g$ induce a stable equivalence of $K$-theory spectra
$$\xymatrix{\bfK(i)\vee\bfK(g)\colon \bfK(\cala) \vee \bfK(\calb) \ar[r] & \bfK(\cale).}$$
\end{enumerate}
\end{corollary}
\begin{proof}
Recall stable quasicategories from Definition~\ref{def:stable_quasicategories} and Example~\ref{examp:stable_quasicategories_are_Waldhausen_quasicategories}. Every split exact sequence in which $\cala$, $\cale$, and $\calb$ are stable quasicategories satisfies the hypotheses of Proposition~\ref{prop:universal_split_exact_sequence:quasicategorical} and is equivalent to a standard split exact sequence: \ref{item:i:prop:universal_split_exact_sequence:quasicategorical} and \ref{item:i':prop:universal_split_exact_sequence:quasicategorical} hold because every map is a cofibration, \ref{item:ii:prop:universal_split_exact_sequence:quasicategorical} holds because the pushout square associated to $A_0\rightarrowtail A_1\rightarrowtail \ast$ is also a pullback square, so the pullback $A_0\rightarrowtail A_1$ of the identity morphism $\ast=\ast$ is an equivalence.

The two claims now follow from Corollary~\ref{cor:QCat_Additivity_GeneralAndSpectral}.
\end{proof}

\begin{remark}
Small stable quasicategories and exact maps\footnote{Here an ``exact'' map between stable quasicategories is a map which preserves finite limits and finite colimits. Given stable quasicategories, it suffices for the map to preserve either finite limits or finite colimits by \cite[Proposition~1.1.4.1]{LurieHigherAlgebra}.} between them form the quasicategory $\mathrm{Cat}^\mathrm{ex}_\infty$, which is the domain of the {\it additive invariants} in the universal characterization of $K$-theory in \cite{BlumbergGepnerTabuadaI}. The quasicategory $\mathrm{Cat}^\mathrm{ex}_\infty$ \cite[1.1.4]{LurieHigherAlgebra} is a ``subcategory'' of $\mathrm{Cat}_\infty$ in the sense recalled in Proposition~\ref{prop:Barwick_equivalence:(i')}. Here $\mathrm{Cat}_\infty$ is the quasicategory of small quasicategories, which is the simplicial nerve of the Kan-complex enriched category with objects small quasicategories and hom-complexes $(Y^X)_\mathrm{equiv}$ \cite[page~164]{JoyalQuadern}. Any map in $\mathrm{Cat}_\infty$ homotopic to an exact functor is exact, so $\mathrm{Cat}^\mathrm{ex}_\infty$ is closed under the homotopy relation, as we would expect of any ``subcategory'' of a quasicategory.
\end{remark}

\section{Appendix 1: Waldhausen Equivalences of Waldhausen (Quasi)Categories} \label{sec:appendix1}

We characterize Waldhausen equivalences as those functors that are exact, reflect weak equivalences and cofibrations, and are equivalences of categories. Then we show that Waldhausen equivalences induce weak equivalences of $\mathbf{Cat}$-simplicial objects $wS_\bullet\calc \to wS_\bullet\cald$, and level-wise weak equivalences of spectra $\bfK(\calc) \to \bfK(\cald)$. We prove these statements first for Waldhausen categories, and then turn to the easier proofs for Waldhausen quasicategories. The differences in the proofs of the categorical and quasicategorical versions are due to the fact that the weak equivalences in a Waldhausen quasicategory are exactly the equivalences of the underlying quasicategory.

These results are needed in Corollary~\ref{cor:Additivity_GeneralAndSpectral} and Corollary~\ref{cor:QCat_Additivity_GeneralAndSpectral}.

\begin{definition}[Categorical Waldhausen Equivalence] \label{def:Waldhausen_equivalence}
An exact functor $F\colon \calc \to \cald$ between Waldhausen categories is a {\it Waldhausen equivalence} if there is an exact functor $G\colon \cald \to \calc$ such that $GF$ and $FG$ are naturally isomorphic to the respective identities. In other words, a Waldhausen equivalence is precisely an equivalence in the 2-category $\mathbf{Wald}_2$ of Remark~\ref{rem:Sn_is_a_2-functor}.
\end{definition}

The following (and its quasicategorical version Proposition~\ref{prop:exact_equivalence_is_Waldhausen_equivalence:qcats}) may be considered a weak variant of Waldhausen's Approximation Theorem \cite[Theorem~1.6.7]{WaldhausenAlgKTheoryI}.

\begin{proposition} \label{prop:exact_equivalence_is_Waldhausen_equivalence}
If a functor $F\colon \calc \to \cald$ between Waldhausen categories is exact, reflects weak equivalences and cofibrations, and is an equivalence of categories, then any inverse equivalence $G\colon \cald \to \calc$ is exact (after replacing $G(\ast_\cald)=\ast_\calc'$ by $\ast_\calc$). Thus, $F$ is a Waldhausen equivalence in the sense of Definition~\ref{def:Waldhausen_equivalence}.
\end{proposition}
\begin{proof}
Suppose $F\colon \calc \to \cald$ is exact, reflects weak equivalences and cofibrations, and is an equivalence of categories, and $G\colon \cald \to \calc$ is a (not necessarily exact) functor equipped with two natural isomorphisms $GF \cong \text{Id}_\calc$ and $FG \cong \text{Id}_\cald$.

We have $\ast_\calc \cong GF(\ast_\calc) = G(\ast_\cald)$.

Suppose $n\colon d \to d'$ is a weak equivalence in $\cald$. We claim that $G(n)$ is a weak equivalence. There exist $c$ and $c'$ in $\calc$ and isomorphisms
$d \cong Fc$ and $d' \cong Fc'$ in $\cald$, as well as a unique morphism $m\colon c \to c'$ in $\calc$ such that $F(m)$ is the composite
\begin{equation} \label{equ:proving_exact_reflection_equiv_has_exact_inv}
\xymatrix{Fc \cong d \ar[r]^-n & d' \cong Fc'.}
\end{equation}
Then $F(m)$ is a weak equivalence, so $m$ is also a weak equivalence. The natural isomorphism $GF \cong \text{Id}_\calc$ then shows that $GF(m)$ is also a weak equivalence by the naturality diagram. The morphism $G(n)$ is also a weak equivalence, as it is the composite of weak equivalences
$$\xymatrix@C=3pc{Gd \cong GFc \ar[r]^-{GF(m)} & GF(c') \cong Gd'}$$
(apply $G$ to \eqref{equ:proving_exact_reflection_equiv_has_exact_inv} and invert the isomorphisms).

Note that we did not assume the 3-for-2 property for weak equivalences in order to prove $G$ preserves weak equivalences, rather, we only used the axioms that weak equivalences are closed under composition and contain all isomorphisms. Since the class of cofibrations also satisfies these two axioms, the same proof shows that $G$ preserves cofibrations.

The functor $G$ sends pushouts along cofibrations to pushouts along cofibrations, as it is an equivalence.

Thus $G$ is an exact functor (after replacing $G(\ast_\cald)=\ast_\calc'$ by $\ast_\calc$).
\end{proof}

\begin{proposition}[Characterization of Categorical Waldhausen Equivalences]
A functor $F\colon \calc \to \cald$ between Waldhausen categories is a Waldhausen equivalence if and only if it is exact, reflects weak equivalences and cofibrations, and is an equivalence of categories.
\end{proposition}
\begin{proof}
Proposition~\ref{prop:exact_equivalence_is_Waldhausen_equivalence} is the direction ``$\Leftarrow$'', so it only remains to prove that every Waldhausen equivalence reflects weak equivalences and cofibrations.

If $F$ is a Waldhausen equivalence and $G$ is an inverse Waldhausen equivalence, then the naturality diagram
\begin{equation} \label{equ:reflection_from_naturality}
\begin{array}{c}
\xymatrix@C=3pc{GF(c) \ar[r]^{GF(m)} \ar[d]_\cong & GF(c') \ar[d]^\cong \\ c \ar[r]_m & c'}
\end{array}
\end{equation}
shows that $F(m)$ is weak equivalence or cofibration if and only if $m$ is.
\end{proof}

\begin{proposition} \label{prop:Waldhausen_equivalence_induces_equivalence of_wS_n}
Let $F\colon \calc \to \cald$ be a Waldhausen equivalence of Waldhausen categories. Then
\begin{enumerate}
\item \label{prop:i:Waldhausen_equivalence_induces_equivalence of_wS_n}
$wF\colon w\calc \to w\cald$ is an equivalence of categories,
\item \label{prop:ii:Waldhausen_equivalence_induces_equivalence of_wS_n}
$S_nF\colon S_n\calc \to S_n\cald$ is a Waldhausen equivalence of Waldhausen categories, and
\item \label{prop:iii:Waldhausen_equivalence_induces_equivalence of_wS_n}
For all $n \geq 0$, the functor $wS_nF\colon wS_n\calc \to wS_n\cald$ is an equivalence of categories.
\end{enumerate}
\end{proposition}
\begin{proof}
\begin{enumerate}
\item
If $G$ is an inverse Waldhausen equivalence to $F$, the functor $wG$ is an inverse equivalence to $wF$ using the same natural isomorphisms (every isomorphism is a weak equivalence).
\item
Recall from Remark~\ref{rem:Sn_is_a_2-functor} that $S_n$ is a 2-functor, so it maps equivalences in $\mathbf{Wald}_2$ to equivalences in $\mathbf{Wald}_2$.
\item
This follows directly from \ref{prop:i:Waldhausen_equivalence_induces_equivalence of_wS_n} and \ref{prop:ii:Waldhausen_equivalence_induces_equivalence of_wS_n}.
\end{enumerate}
\end{proof}

\begin{corollary} \label{cor:Waldhausen_equivalences_induce_weak_equivalences}
If $F\colon \calc \to \cald$ is a Waldhausen equivalence, then
$$\xymatrix{wS_\bullet F\colon wS_\bullet \calc \ar[r] & wS_\bullet\cald}$$
is a weak equivalence of simplicial objects in $\mathbf{Cat}$. That is, the diagonal of its level-wise nerve is a weak homotopy equivalence of simplicial sets. Moreover,
$$\xymatrix{\bfK(F)\co \bfK(\calc) \ar[r] & \bfK(\cald)}$$
is a level-wise weak equivalence of spectra.
\end{corollary}
\begin{proof}
The Realization Lemma in combination with Proposition~\ref{prop:Waldhausen_equivalence_induces_equivalence of_wS_n}~\ref{prop:iii:Waldhausen_equivalence_induces_equivalence of_wS_n} proves the first statement. The second statement follows from the first.
\end{proof}

Having reached the goal for Waldhausen categories, we now turn to the quasicategorical analogues. Since every equivalence of quasicategories reflects equivalences, we will not require reflection of (weak) equivalences in the quasicategorical analogues.

\begin{definition}[Quasicategorical Waldhausen Equivalence] \label{def:Waldhausen_equivalence:qcats}
An exact functor $F\colon \calc \to \cald$ between Waldhausen quasicategories is a {\it Waldhausen equivalence} if there is an exact functor $G\colon \cald \to \calc$ such that $GF$ and $FG$ are naturally equivalent to the respective identities.  In other words, a Waldhausen equivalence is precisely an equivalence in the 2-category $\mathbf{QWald}_2$ of Remark~\ref{rem:QWald2}.
\end{definition}

\begin{proposition} \label{prop:exact_equivalence_is_Waldhausen_equivalence:qcats}
If a functor $F\colon \calc \to \cald$ between Waldhausen quasicategories is exact, reflects cofibrations, and is an equivalence of quasicategories, then any inverse equivalence $G\colon \cald \to \calc$ is exact. Thus, $F$ is a Waldhausen equivalence in the sense of Definition~\ref{def:Waldhausen_equivalence:qcats}.
\end{proposition}
\begin{proof}
Suppose $F\colon \calc \to \cald$ is exact, reflects cofibrations, and is an equivalence of quasicategories, and $G\colon \cald \to \calc$ is a (not necessarily exact) functor equipped with two natural equivalences $GF \simeq \text{Id}_\calc$ and $FG \simeq \text{Id}_\cald$.

We have $\ast_\calc \simeq GF(\ast_\calc) \simeq G(\ast_\cald)$.

Suppose $n \co d \to d'$ is a cofibration in $\cald$. We have a commutative square in $\cald$
$$\xymatrix@C=3pc{FGd \ar[r]^{FGn} \ar[d]_{\mathrm{equiv}} & FGd' \ar[d]^{\mathrm{equiv}} \\ d \ar[r]_n & d'}$$
so that $FGn$ is a cofibration by the homotopy repleteness of $co\cald$ in $\cald$. But since $F$ reflects cofibrations, $Gn$ is also a cofibration, so $G$ sends cofibrations to cofibrations.

The functor $G$ sends pushouts along cofibrations to pushouts along cofibrations, as it is an equivalence.

Thus $G$ is an exact functor.
\end{proof}

\begin{proposition}[Characterization of Quasicategorical Waldhausen Equivalences]
A functor $F\colon \calc \to \cald$ between Waldhausen quasicategories is a Waldhausen equivalence if and only if it is exact, reflects cofibrations, and is an equivalence of categories.
\end{proposition}
\begin{proof}
Proposition~\ref{prop:exact_equivalence_is_Waldhausen_equivalence:qcats}  is the direction ``$\Leftarrow$'', so it only remains to prove that every Waldhausen equivalence of quasicategories reflects cofibrations.
Reflection of cofibrations follows from homotopy repleteness, compare with diagram \eqref{equ:reflection_from_naturality}.
\end{proof}

\begin{proposition} \label{prop:Waldhausen_equivalence_induces_equivalence of_S_n_equiv}
Let $F\colon \calc \to \cald$ be a Waldhausen equivalence of Waldausen quasicategories. Then
\begin{enumerate}
\item \label{prop:i:Waldhausen_equivalence_induces_equivalence of_S_n_equiv}
$(F)_\text{\rm equiv}\colon \calc_\text{\rm equiv} \to \cald_\text{\rm equiv}$ is an equivalence of quasicategories.
\item \label{prop:ii:Waldhausen_equivalence_induces_equivalence of_S_n_equiv}
$S_n^\infty F\colon S_n^\infty \calc \to S_n^\infty \cald$ is a Waldhausen equivalence of Waldhausen quasicategories, and
\item \label{prop:iii:Waldhausen_equivalence_induces_equivalence of_S_n_equiv}
For all $n \geq 0$, the map $(S_n^\infty F)_\text{\rm equiv}\colon (S_n^\infty \calc)_\text{\rm equiv} \to (S_n^\infty\cald)_\text{\rm equiv}$ is an equivalence of quasicategories.
\end{enumerate}
\end{proposition}
\begin{proof}
\begin{enumerate}
\item
If $G$ is an inverse Waldhausen equivalence to $F$, the functor $G_\text{\rm equiv}$ is an inverse equivalence to $F_\text{\rm equiv}$ using the same natural equivalences.
\item
Recall from Remark~\ref{rem:QWald2} that $S_n^\infty$ is a 2-functor, so it maps equivalences in $\mathbf{QWald}_2$ to equivalences in $\mathbf{QWald}_2$.
\item
This follows directly from \ref{prop:i:Waldhausen_equivalence_induces_equivalence of_S_n_equiv} and \ref{prop:ii:Waldhausen_equivalence_induces_equivalence of_S_n_equiv}.
\end{enumerate}
\end{proof}

\begin{corollary} \label{cor:Waldhausen_equivalences_induce_weak_equivalences:qcats}
If $F\colon \calc \to \cald$ is a Waldhausen equivalence of quasicategories, then
$$\xymatrix{\left( S_\bullet^\infty F \right)_\mathrm{equiv} \colon \left(S_\bullet^\infty \calc\right)_\mathrm{equiv} \ar[r] & \left(S_\bullet^\infty\cald\right)_\mathrm{equiv}}$$
is a diagonal weak equivalence of bisimplicial sets. Moreover, if $\calc$, $\cald$, and $F$ are pointed, then $$\xymatrix{\bfK(F)\co \bfK(\calc) \ar[r] & \bfK(\cald)}$$
is a level-wise weak equivalence of spectra.
\end{corollary}
\begin{proof}
The first result follows from Proposition~\ref{prop:Waldhausen_equivalence_induces_equivalence of_S_n_equiv}~\ref{prop:iii:Waldhausen_equivalence_induces_equivalence of_S_n_equiv}, the Realization Lemma, and the fact that every equivalence of quasicategories is a homotopy equivalence of simplicial sets. The second statement is a consequence of the first.
\end{proof}

\section{Appendix 2: Waldhausen Equivalences of Sequences} \label{sec:appendix_Wald_equivalences_of_ses}

In this section we focus on quasicategorical definitions and results about morphisms and equivalences of short sequences of Waldhausen quasicategories. We prove that the properties of ``exactness'' and ``split exactness'' are preserved under Waldhausen equivalence of sequences, and we show that Waldhausen equivalences between split exact sequences are compatible with splittings. Analogous statements for categories are obtained by simply replacing the 2-category $\mathbf{QWald}_2$ by the 2-category $\mathbf{Wald}_2$ from Remark~\ref{rem:Sn_is_a_2-functor}. Similar results also hold for pointed Waldhausen quasicategories. We are aware that some of the results in this section can be formulated for a general 2-category in place of $\mathbf{QWald}_2$ and $\mathbf{Wald}_2$, but we work with $\mathbf{QWald}_2$ for concreteness.

The categorical versions of results in this section are needed in Corollary~\ref{cor:Additivity_GeneralAndSpectral}, and the stated quasicategorical results are needed in Corollary~\ref{cor:QCat_Additivity_GeneralAndSpectral}.

For 2-categories $\bfC$ and $\bfD$, let $\text{\bf 2-Cat}_{ps}(\bfC,\bfD)$ be the strict 2-category of strict 2-functors $\bfC \to \bfD$, {\it pseudo} natural transformations, and modifications. Recall that an equivalence in this 2-category is a pseudo natural transformation with every component an equivalence in the 2-category $\bfD$. Such a pseudo natural transformation is called a {\it pseudo natural equivalence}. Recall the 2-category $\mathbf{QWald}_2$ of Waldhausen quasicategories from Remark~\ref{rem:QWald2}. An object of $\text{\bf 2-Cat}_{ps}([2],\mathbf{QWald}_2)$ is a (not necessarily exact) sequence $\xymatrix@1@C=1pc{\cala \ar[r] \ar@/_.75pc/[rr] & \cale \ar[r] & \calc}$ of Waldhausen quasicategories and exact functors.


Recall that every iso 2-cell in $\mathbf{QWald}_2$ is represented by a natural equivalence between functors of quasicategories, see Corollary~\ref{cor:nat_transf_is_nat_equiv_iff_comps_equivs}~\ref{item(i)}$\Leftrightarrow$\ref{item(ii)}$\Leftrightarrow$\ref{item(iv)}.

\begin{definition}[Morphism and Waldhausen Equivalence of Quasicategorical Sequences] \label{def:Waldhausen_equiv_of_sequences_quasicategorical}
A {\it morphism $\Phi$ of sequences of Waldhausen quasicategories} is a morphism $\Phi$ in $\text{\bf 2-Cat}_{ps}([2],\mathbf{QWald}_2)$, as pictured in \eqref{equ:def:Waldhausen_equiv_of_sequences_quasicategorical}. This consists of three exact maps $\Phi_\cala$, $\Phi_\cale$, and $\Phi_\calb$ and three iso 2-cells (represented by natural equivalences) between the respective composites, compatible in the sense that the two smaller iso 2-cells paste together to give the third.
\begin{equation} \label{equ:def:Waldhausen_equiv_of_sequences_quasicategorical}
\begin{array}{c}
\xymatrix{\cala \ar[r]^i \ar[d]_{\Phi_\cala} \ar@/^1.5pc/[rr]^{f\circ i} & \cale \ar[r]^f \ar[d] \ar@{=>}[dl]_\cong & \calb \ar[d]^{\Phi_\calb} \ar@{=>}[dl]^\cong \ar@{=>}[dll]_(.3)\cong \\ \cala' \ar[r]_{i'} \ar@/_1.5pc/[rr]_{f' \circ i'} & \cale' \ar[r]_{f'} & \calb' }
\end{array}
\end{equation}
A morphism of sequences is a {\it Waldhausen equivalence of sequences} if it is an equivalence in the 2-category $\text{\bf 2-Cat}_{ps}([2],\mathbf{QWald}_2)$, that is, if each of the 3 vertical maps in \eqref{equ:def:Waldhausen_equiv_of_sequences_quasicategorical} is a Waldhausen equivalence in the sense of Definition~\ref{def:Waldhausen_equivalence:qcats}. A {\it morphism of (split) exact sequences of Waldhausen quasicategories} is a morphism of the underlying sequences. A {\it Waldhausen equivalence of (split) exact sequences} is a Waldhausen equivalence of the underlying sequences.
\end{definition}

Since the iso 2-cell in \eqref{equ:def:Waldhausen_equiv_of_sequences_quasicategorical} stretching across two squares is determined as the pasting composite of the other two iso 2-cells, we no longer indicate it.

Recall the notions of exact sequence and split exact sequence in Definitions~\ref{def:exact_sequence_of_Waldhausen_quasicategories} and \ref{def:split-exact_sequence_of_Waldhausen_quasicategories}.

\begin{proposition} \label{prop:exactness_and_split_exactness_preserved_under_W_equiv_of_sequences}
The properties ``exactness'' and ``split exactness'' of sequences of Waldhausen quasicategories are preserved under Waldhausen equivalence of sequences in Definition~\ref{def:Waldhausen_equiv_of_sequences_quasicategorical}. That is, if $\Phi\co (\cala,\cale,\calb) \to (\cala',\cale',\calb')$ is a morphism of sequences, then $(\cala,\cale,\calb)$ is exact, respectively split exact, if and only if $(\cala',\cale',\calb')$ is exact, respectively split exact.
\end{proposition}
\begin{proof}
Suppose $\xymatrix@1@C=1pc{\cala \ar[r]^i & \cale \ar[r]^f & \calb}$ is exact, and let $\Psi\co (\cala',\cale',\calb') \to (\cala,\cale,\calb)$ be a pseudo inverse to $\Phi$ pictured in \eqref{equ:def:Waldhausen_equiv_of_sequences_quasicategorical}. \begin{enumerate}
\item
In the 2-category $\mathbf{QWald}_2$ we have an iso 2-cell $f \circ i \cong \ast$, so can construct $f' \circ i' \cong \ast'$ as follows. The outer rectangle of \eqref{equ:def:Waldhausen_equiv_of_sequences_quasicategorical} yields the following iso 2-cells via 2-categorical pastings of 2-cells.
$$\aligned \Phi_\calb \circ f \circ i & \cong f' \circ i' \circ \Phi_\cala \\
\ast':=\Phi(\ast) & \cong  f' \circ i' \circ \Phi_\cala \\
\ast' \circ \Psi_\cala & \cong (f' \circ i' \circ \Phi_\cala) \circ \Psi_\cala \\
\ast' & \cong f' \circ i' \circ \text{Id}_{\cala'} \\
\ast' & \cong f' \circ i' \\
\endaligned
$$
\item
The left square of \eqref{equ:def:Waldhausen_equiv_of_sequences_quasicategorical} shows that $i'$ is naturally equivalent in $\mathbf{QWald}_{\infty,2}$ to the fully faithful exact functor $\Phi_\cale \circ i \circ \Psi_\cala$, so $i'$ is also fully faithful.
\item
Let $\Phi_{\cale /\cala}$ be the restriction of $\Phi_\cale$ to $\cale /\cala$. We claim that $\Phi_{\cale /\cala}$ is an equivalence $\cale /\cala \to \cale' /\cala'$ of quasicategories. Suppose $\Phi_\cala(A) \simeq A' $ and $\Phi_\cale(E)  \simeq E'$ are equivalences. Since $\Phi_\cale$ is fully faithful, the map $\cale(iA,E) \to \cale'(\Phi_\cale iA,\Phi_\cale E)$ is a weak homotopy equivalence. But this latter simplicial set has the same homotopy type as $\cale'(i'A' ,E')$ because of the equivalences  $\Phi_\cale iA \simeq  i'\Phi_\cala A \simeq  i'A'$ and $\Phi_\cale(E)\simeq E'$.

Thus $\cale(iA,E)$ and $\cale'(i'A' ,E')$ have the same homotopy type when $\Phi_\cala(A) \simeq A' $ and $\Phi_\cale(E)  \simeq E'$. From this and the fully faithfulness of $\Phi_\cale$ we can see $\Phi_{\cale /\cala}$ is an equivalence of quasicategories as follows.
\begin{enumerate}
\item
We first show the codomain of $\Phi_{\cale/\cala}$ is indeed $\cale'/ \cala'$.
Let $E \in \cale/\cala$ and take $\Phi_\cale(E) \simeq E'$ to be the identity. Let $A' \in \cala'$ be arbitrary, then there exists some $A \in \cala$ and some equivalence $\Phi_\cala A \simeq A'$. Thus $\cale'(i'A' ,\Phi_\cale E)$ is weakly equivalent to $\cale(iA ,E)$, which is weakly contractible, and $\Phi_\cale E \in \cale'/ \cala'$.
\item
For essential surjectivity of $\Phi_{\cale/\cala}$, now let $E' \in \cale'/ \cala'$. Then there exists some $E \in \cale$ and some equivalence $\Phi_\cale E \simeq E'$. We need to show $E \in \cale/\cala$. Let $A \in \cala$ and take $\Phi(A) \simeq A'$ to be the identity. Then $\cale(iA,E)$ has the same homotopy type as $\cale'(i'\Phi(A),E')$, which is weakly contractible, so $E \in \cale/\cala$.
\end{enumerate}
Suppose now that the exact sequence $\xymatrix@1@C=1pc{\cala \ar[r]^i & \cale \ar[r]^f & \calb}$ is split with splitting $\xymatrix@1@C=1pc{\cala \ar@{<-}[r]^j & \cale \ar@{<-}[r]^g & \calb}$. We claim the exact functors $j':= \Phi_\cala j \Psi_\cale$ and $g':=\Phi_\cale g \Psi_\calb$ are splittings for $\xymatrix@1@C=1pc{\cala' \ar[r]^{i'} & \cale' \ar[r]^{f'} & \calb'}$. Recall that in any 2-category, a given equivalence is both a left and a right adjoint equivalence (for possibly different adjoints). To see $i' \dashv j'$, we consider $\Phi_\cala$ as a right adjoint equivalence with (not necessarily exact) left adjoint $\psi_\cala$, and we consider $\Phi_\cale$ as a left adjoint equivalence with (not necessarily exact) right adjoint $\psi_\cale$.
$$\xymatrix{\cala' \ar@/^.5pc/[r]^{\psi_\cala} \ar@{}[r]|{\perp}  & \cala \ar@/^.5pc/[l]^{\Phi_\cala} \ar@/^.5pc/[r]^i \ar@{}[r]|{\perp} & \ar@/^.5pc/[l]^j \cale \ar@/^.5pc/[r]^{\Phi_\cale} \ar@{}[r]|{\perp} & \cale' \ar@/^.5pc/[l]^{\psi_\cale} }$$
Adjunctions in a 2-category compose, so the top composite, which is isomorphic to $i'$, is left adjoint to the bottom composite, which is isomorphic to $j'$. Hence $i' \dashv j'$. To see its unit is a natural equivalence, one again uses 2-categorical pastings of iso 2-cells.

An analogous argument shows $f' \dashv g'$ with counit a natural equivalence.

For the converse part of the ``if and only if'' statement, we repeat the above argument with the roles of $\Phi$ and $\Psi$ reversed.
\end{enumerate}
\end{proof}

\begin{proposition}[Compatibility of Waldhausen Equivalences with Splittings] \label{prop:Waldhausen_equiv_of_sequences_compatible_with_j's}
Suppose $\Phi\co (\cala,\cale,\calb) \to (\cala',\cale',\calb')$ is a Waldhausen equivalence of sequences and the domain sequence $(\cala,\cale,\calb)$ is split exact. Then $(\cala',\cale',\calb')$ is split exact and $\Phi$ is compatible with any choices of splittings $(j,g)$ and $(j',g')$ in the sense that iso 2-cells exist in the following diagram.
\begin{equation} \label{equ:prop:Waldhausen_equiv_of_sequences_compatible_with_j's}
\begin{array}{c}
\xymatrix{\cala \ar@{<-}[r]^j \ar[d]_{\Phi_\cala} \ar@{=>}[dr]^\cong & \cale \ar@{<-}[r]^g \ar[d] \ar@{=>}[dr]^\cong \ar[d]^{\Phi_\cale} & \calb \ar[d]^{\Phi_\calb} \\ \cala' \ar@{<-}[r]_{j'} & \cale' \ar@{<-}[r]_{g'} & \calb' }
\end{array}
\end{equation}
\end{proposition}
\begin{proof}
Select a splitting $(j,g)$ of $(\cala,\cale,\calb)$. By Proposition~\ref{prop:exactness_and_split_exactness_preserved_under_W_equiv_of_sequences}, the codomain sequence $(\cala',\cale',\calb')$ is also split with $j':=\Phi_\cala j \Psi_\cale$ and $g':=\Phi_\cale g \Psi_\calb$. The iso 2-cell in the left square of \eqref{equ:prop:Waldhausen_equiv_of_sequences_compatible_with_j's} comes from $\text{Id}_\cale \cong \Psi_\cale \circ \Phi_\cale$, namely
$$\Phi_\cala \circ j =\Phi_\cala \circ j \circ \text{Id}_\cale \cong \Phi_\cala \circ j \circ (\Psi_\cale \circ \Phi_\cale)=(\Phi_\cala \circ j \circ \Psi_\cale) \circ \Phi_\cale \overset{\text{def }j'}{=} j' \circ \Phi_\cale.$$
The iso 2-cell in the right square of \eqref{equ:prop:Waldhausen_equiv_of_sequences_compatible_with_j's} comes from
$\text{Id}_\calb \cong \Psi_\calb \circ \Phi_\calb$, namely
$$\Phi_\cale \circ g =\Phi_\cale \circ g \circ \text{Id}_\calb \cong \Phi_\cale \circ g \circ (\Psi_\calb \circ \Phi_\calb)=(\Phi_\cale \circ g \circ \Psi_\calb) \circ \Phi_\calb \overset{\text{def }g'}{=} g' \circ \Phi_\calb.$$
If one makes different choices of splittings, then they are isomorphic to the ones above by uniqueness of adjoints in a 2-category, so the analogous iso 2-cells in \eqref{equ:prop:Waldhausen_equiv_of_sequences_compatible_with_j's} are obtained by pasting with the ones we have already constructed for the selected splittings.
\end{proof}

\begin{remark}
Notice that in the proof of Proposition~\ref{prop:Waldhausen_equiv_of_sequences_compatible_with_j's} we used the pseudo inverse equivalence $\Psi$. If $\Phi$ is merely a morphism of sequences and both domain and codomain are split exact with selected splittings, then one can construct not-necessarily-iso 2-cells in the directions pictured in \eqref{equ:prop:Waldhausen_equiv_of_sequences_compatible_with_j's} using the units and counits of the adjunctions $i \dashv j$, $i' \dashv j'$, $f \dashv g$, $f' \dashv g'$ and the iso 2-cells of \eqref{equ:def:Waldhausen_equiv_of_sequences_quasicategorical}.
\end{remark}

\section{Appendix 3: Simplicial Homotopy} \label{sec:appendix_simplicial_homotopy}

We recall the notion of simplicial homotopy and show to construct simplicial homotopies for $\mathfrak{s}_\bullet$ and $\mathfrak{s}_\bullet^\infty$ from natural transformations, as needed for Lemmas~\ref{lem:HardLemmaMadeEasy} and \ref{lem:HardLemmaMadeEasy_quasicategorical}.

\begin{definition}[Simplicial Homotopy]
A {\it simplicial homotopy} $H\colon f \simeq g$ for simplicial maps $f,g\colon X \to Y$ is a map $H\colon X \times \Delta[1] \to Y$ such that $H(-,0)=f$ and $H(-,1)=g$.
\end{definition}

The following equivalent description of a simplicial homotopy is easier to work with, and we use it extensively in Lemma~\ref{lem:HardLemmaMadeEasy}.

\begin{proposition}[Lemma~6.4.8 of \cite{RognesTextbook}] \label{prop:simplicial_homotopy_description}
A simplicial homotopy $H\colon f \simeq g$ is equivalently described by a collection of functions
$h^j_n\colon X_n \to Y_n$
for $0 \leq j \leq n+1$ and $n \geq 0$ such that:
\begin{itemize}
\item
for $n \geq 1$,
$$
d_i \circ h^j_n=
        \begin{cases}
           h^{j-1}_{n-1} \circ d_i & \quad 0\leq i \leq j-1 \\
           h^{j}_{n-1} \circ d_i & \quad j \leq i \leq n
        \end{cases}
$$
\item
for $n \geq 0$,
$$
s_i \circ h^j_n=
        \begin{cases}
           h^{j+1}_{n+1} \circ s_i & \quad 0\leq i \leq j-1 \\
           h^{j}_{n+1} \circ s_i & \quad j \leq i \leq n
        \end{cases}
$$
\item
and for all $n \geq 0$, we have $h_n^{n+1}=f_n$ and $h_n^0=g_n$.
\end{itemize}
The relation between the two descriptions of a simplicial homotopy is given by
\begin{equation} \label{equ:simp_homotopy_equivalent_descriptions}
h^j_n(x)=H_n(x,\zeta^j_n),
\end{equation}
where $H_n\co X_n \times \Delta([n],[1]) \to Y_n$ is the degree $n$ component of $H\co X \times \Delta[1] \to Y$ and $\zeta^j_n\co[n] \to [1]$ sends $0,1,\dots,j-1$ to $0$ and $j,j+1, \dots n$ to $1$.
\end{proposition}

Notice the functions here have codomain $Y_n$, while in the equivalent formulation of May~\cite{MaySimplicial} the functions have codomain $Y_{n+1}$ and satisfy different identities.

We next take another viewpoint on the correspondence \eqref{equ:simp_homotopy_equivalent_descriptions} in order to make simplicial homotopies from natural transformations apparent, and, importantly, to construct simplicial homotopies in the context of $\mathfrak{s}_\bullet$ and $\mathfrak{s}_\bullet^\infty$ from natural transformations. Consider, for $n \geq 0$ and $0 \leq j \leq n+1$, the functions $\psi_n^j \colon [n] \to [n] \times [1]$ which embed $[n]$ into $[n] \times [1]$, with start and finish at two of the four corners, and pass diagonally through the $j$-th square (the first square is square 1). In particular, $\psi_n^0$ passes through no square, so it is entirely the bottom path $[n] \times \{1\}$, while $\psi_n^{n+1}$ is entirely the top path $[n] \times \{0\}$. Below is a picture of $\psi^2_4$.
\begin{equation}  \label{equ:psi^2_4}
\begin{array}{c}
\xymatrix@C=3pc@R=3pc{ \ar@{-->}[r] \ar[d]  & \ar[r] \ar[d] \ar@{-->}[dr] & \ar[r] \ar[d] & \ar[r] \ar[d] & \ar[d] \\ \ar[r] & \ar[r] & \ar@{-->}[r] & \ar@{-->}[r] & }
\end{array}
\end{equation}
An explicit formula for $\psi^j_n\co [n] \to [n] \times [1]$ is
\begin{equation}  \label{equ:defn_of_psi}
\begin{array}{c}
\psi^j_n(i)=\begin{cases} (i,0) & 0 \leq i \leq j-1 \\
(i,1) & j \leq i \leq n.
\end{cases}
\end{array}
\end{equation}
In other words, $\psi^j_n=(\text{Id}_{[n]},\zeta^j_n)$ where $\zeta^j_n \co [n] \to [1]$ is defined in Proposition~\ref{prop:simplicial_homotopy_description}

\begin{corollary} \label{cor:simplicial_homotopy_using_psi}
Under the Yoneda identifications $$X_n\cong \mathbf{SSet}(\Delta[n],X)\hspace{.5in} \text{and} \hspace{.5in} Y_n\cong \mathbf{SSet}(\Delta[n],Y),$$ the maps $h^j_n\co X_n \to Y_n$ describing a simplicial homotopy $H\co X \times \Delta[1] \to Y$ are
\begin{equation} \label{equ:simplicial_homotopy_using_psi}
\Big( \overline{x} \co \Delta[n] \to X \Big) \xymatrix{ \ar@{|->}[r] &} \Big(H \circ (\overline{x} \times \text{\rm Id}_{\Delta[1]}) \circ (\psi^j_n)_*\co \Delta[n] \to Y \Big).
\end{equation}
\end{corollary}
\begin{proof}
First notice from \eqref{equ:defn_of_psi} that $\psi^j_n$ is $i \mapsto (1_{[n]} (i), \zeta^j_n(i))$. Recall by the Yoneda Lemma, $x \in X_n$ corresponds to the map $\overline{x}\co \Delta[n] \to X$ with $1_{[n]} \mapsto x$ in degree $n$. Evaluating the right hand side of \eqref{equ:simplicial_homotopy_using_psi} in degree $n$ on $1_{[n]}$ we have
$$\xymatrix@C=3.5pc@R=.5pc{\Delta([n],[n]) \ar[r]^-{(\psi^j_n)_\ast} & \Delta([n],[n]) \times \Delta([n],[1]) \ar[r]^-{\overline{x}_n \times \text{Id}} & X_n \times \Delta([n],[1]) \ar[r]^-{H_n} & Y_n\phantom{.} \\
1_{[n]} \ar@{|->}[r] & (1_{[n]},\zeta^j_n) \ar@{|->}[r] & (x,\zeta^j_n) \ar@{|->}[r] & H_n(x,\zeta^j_n). }$$
By Yoneda Lemma again, we see the assignment \eqref{equ:simplicial_homotopy_using_psi} corresponds to the assignment $h^j_n$ as in \eqref{equ:simp_homotopy_equivalent_descriptions}.
\end{proof}

As is well known, the nerve of a natural transformation is a simplicial homotopy because nerve commutes with finite products. The functions $h^j_n$ assign to a path of morphisms in the domain the dotted path of morphisms in the naturality diagram in the codomain indicated by \eqref{equ:psi^2_4}, as we see in the next example.

\begin{example}[A Simplicial Homotopy of Nerves of Functors from a Natural Transformation] \label{examp:simplicial_homotopy_from_natural_transformation}
Let $Z\colon\mathcal{C} \times [1] \to \mathcal{D}$ be a natural transformation. Then the induced simplicial homotopy
$NZ \co N\calc \times \Delta[1] \to N \cald$ has maps $h^j_n\colon \big( N \calc \big)_n \to \big( N \cald\big)_n$ which take an $n$-simplex $\alpha \colon [n] \to \calc$ to $Z \circ (\alpha \times \text{Id}_{[1]}) \circ \psi^j_n$, as follows from \eqref{equ:simplicial_homotopy_using_psi} and the fully faithfulness of the nerve. In particular, we see that for an identity natural transformation $F \Rightarrow F$, the associated maps $h^j_n$ for fixed $n$ are all the same, namely $h^j_n=(NF)_n$. Similarly, the homotopy $\lambda$ in the proof of Lemma~\ref{lem:HardLemmaMadeEasy} is an ``identity'' homotopy on a simplicial map.
\end{example}

We are also interested in homotopies of double nerves of functors in order to construct a second homotopy $\theta$ in Lemmas~\ref{lem:HardLemmaMadeEasy} and \ref{lem:HardLemmaMadeEasy_quasicategorical} in the context of the  $\mathfrak{s}_\bullet$ and $\mathfrak{s}_\bullet^\infty$ constructions. An $n$-simplex of the ordinary nerve $N\calc$ is of course a functor $[n] \to \calc$, while a simplex of the {\it double nerve} $N_d\,\calc$ is a functor $[n] \times [n] \to \calc$. Various nerves of double categories are considered in \cite{FiorePaoli2010} and \cite{FiorePaoliPronk2008}. We consider the double nerve as a simplicial set via pre-composition with $\delta^i_n \times \delta^i_n$ and $\sigma^i_n \times \sigma^i_n$. Clearly $N_d\co \mathbf{Cat} \to \mathbf{SSet}$ is a functor. We now define functions $\varphi^j_n \co [n] \times [n] \to [n] \times [n]\times [1]$ analogous to $\psi^j_n$ in \eqref{equ:psi^2_4} and \eqref{equ:defn_of_psi}.
\begin{equation}  \label{equ:defn_of_varphi}
\begin{array}{c}
\varphi^j_n(i,k)=\begin{cases} (i,k,0) & 0 \leq i \leq j-1 \; \text{ and } \; 0 \leq  k \leq j-1 \\
(i,k,1) & j \leq i\leq n \; \text{ or } \; j \leq k \leq n.
\end{cases}
\end{array}
\end{equation}

\begin{example}[A Simplicial Homotopy of Double Nerves of Functors from a Natural Transformation] \label{examp:simp_homotopy_double_nerve}
Let $Q\co \calc \times [1] \to \cald$ be a natural transformation. We induce a homotopy $N_d \, \calc \times \Delta[1] \to N_d \, \cald$ as follows. Its maps $h^j_n\colon \big( N_d\, \calc \big)_n \to \big( N_d\, \cald\big)_n$ take an $n$-simplex $\alpha \colon [n] \times [n] \to \calc$ to $Q \circ (\alpha \times \text{Id}_{[1]}) \circ \varphi^j_n$. These maps $h^j_n$ satisfy the identities required in Proposition~\ref{prop:simplicial_homotopy_description} because Example~\ref{examp:simplicial_homotopy_from_natural_transformation} does. Namely, restricting to any row or column of $[n] \times [n]$ brings us into the situation of Example~\ref{examp:simplicial_homotopy_from_natural_transformation}.
\end{example}

\begin{example}[A Simplicial Homotopy of $\mathfrak{s}_\bullet$ of Exact Functors from a Natural Transformation] \label{examp:nat_transf_induces_s_bullet_homotopy}
From Example~\ref{examp:simp_homotopy_double_nerve} we can now draw a similar conclusion for $\mathfrak{s}_\bullet$ by simply restricting $\varphi^j_n$ to $\text{Ar}[n] \subseteq [n] \times [n]$. Let $\cale$ and $\cale'$ be Waldhausen categories, and suppose that $Q \co \cale \times [1] \to \cale'$ is a natural transformation such that $Q_0$ and $Q_1$ are exact functors. Then $Q$ induces a simplicial homotopy from $\mathfrak{s}_\bullet Q_0$ to $\mathfrak{s}_\bullet Q_1$ as in Example~\ref{examp:simp_homotopy_double_nerve}.

A particular choice of $Q$ constructs the simplicial homotopy $\theta\co (\rho_n^0)_n \simeq \iota \circ r$ at the end of the proof of Lemma~\ref{lem:HardLemmaMadeEasy}. Consider the natural transformation $Q$ of functors $\cale(\cala,\calc,\calb)\to \cale(\cala,\calc,\calb)$ given by
\begin{equation} \label{equ:E_nat_transf_for_theta}
\begin{array}{c}
Q\co
\end{array}
\begin{array}{c}
\xymatrix{A \ar@{>->}[d]_i \\ C \ar@{->>}[d]_f \\ B }
\end{array}
\begin{array}{c}
\xymatrix{\ar@{|->}[r] & }
\end{array}
\begin{array}{c}
\xymatrix{\ast \ar@{>->}[d] \ar[r] & \ast \ar@{>->}[d] \\ C \cup_A \ast \ar@{->>}[d] \ar@{->>}[r] & B \ar@{=}[d] \\ B \ar[r] & B }
\end{array},
\end{equation}
where the assigned morphism on the right of \eqref{equ:E_nat_transf_for_theta} is induced by the selected functorial choice of pushout in $\calc$ (which preserves identities) applied to the following.
\begin{equation} \label{equ:E_nat_transformation_as_pushout}
\begin{array}{c}
\xymatrix{A \ar@{>->}[d]_i & A \ar@{>->}[l]_{\text{Id}} \ar[r] \ar@{=}[d] \ar[dl]_i & \ast \ar@{=}[d] \\ C \ar@{->>}[d]_f \ar@{}[dr]|{\text{p.o.}} & A \ar@{>->}[l]_i \ar[r] \ar[d] & \ast \ar@{=}[d] \\ B & \ast \ar@{>->}[l] \ar[r]^{\text{Id}} & \ast }
\end{array}
\begin{array}{c}
\xymatrix{\ar[r] & }
\end{array}
\begin{array}{c}
\xymatrix{\ast \ar[d] & \ast \ar@{=}[d] \ar@{>->}[l] \ar[r]^{\text{Id}} & \ast \ar@{=}[d] \\ B \ar@{=}[d] & \ast \ar@{=}[d] \ar@{>->}[l] \ar[r]^{\text{Id}} & \ast \ar@{=}[d] \\ B & \ast \ar@{>->}[l] \ar[r]^{\text{Id}} & \ast }
\end{array}
\end{equation}
Then the resulting simplical homotopy is $\theta\co s_\bullet Q_0 \simeq s_\bullet Q_1$ at the end of the proof of Lemma~\ref{lem:HardLemmaMadeEasy}.
\end{example}

\begin{example}[A Simplicial Homotopy of $\mathfrak{s}_\bullet^\infty$ of Exact Functors from a Natural Transformation] \label{examp:nat_transf_induces_s_bullet_homotopy_quasicategorical}
If $Q \co \cale \times \Delta[1] \to \cale'$ is a natural transformation of exact functors between Waldhausen quasicategories, then we obtain a simplicial homotopy from $\mathfrak{s}_\bullet^\infty Q_0$ to $\mathfrak{s}_\bullet^\infty Q_1$. Its maps $h^j_n\co \mathfrak{s}^\infty_n \cale \to \mathfrak{s}^\infty_n \cale'$ take $\alpha\co N\text{Ar}[n] \to \cale$ to $Q \circ (\alpha \times \text{Id}_{\Delta[1]}) \circ \varphi^j_n\vert_{N\text{Ar}[n]}$.

In particular, the quasicategorical version of \eqref{equ:E_nat_transf_for_theta} and \eqref{equ:E_nat_transformation_as_pushout} goes through for $Q\co \cale^{po}(\cala,\calc,\calb)\times \Delta[1] \to \cale^{po}(\cala,\calc,\calb)$ to construct the simplicial homotopy $\theta\co (\rho_n^0)_n \simeq \iota \circ r$ at the end of the proof of Lemma~\ref{lem:HardLemmaMadeEasy_quasicategorical}. All of the zero objects $\ast$ in \eqref{equ:E_nat_transformation_as_pushout} are the selected common zero object (if there is not a common zero object in $\cala$ and $\calb$, we may simply equivalently redefine $\cala$ and $\calb$ to include a common zero object), while the horizontal morphisms $\to \ast$ in the first part of \eqref{equ:E_nat_transformation_as_pushout} are from the selected natural transformation $\text{Id}_\calc \Rightarrow \ast$, and the  horizontal morphisms $\leftarrow \ast$ in the second part of \eqref{equ:E_nat_transformation_as_pushout} are from the selected natural transformation $\ast \Rightarrow \text{Id}_\calc$ (both of these natural transformations can be selected to have identity component on $\ast$). In the second part of \eqref{equ:E_nat_transformation_as_pushout}, the right arrows $\ast \to \ast$ marked with $\text{Id}$ are from the identity natural transformation on $\ast$. The square labelled ``p.o.'' in \eqref{equ:E_nat_transformation_as_pushout} is the commutative square  \eqref{equ:cofiber_sequence_quasicategorical} that is part of the data of the cofiber sequence in \eqref{equ:E_nat_transf_for_theta}. We now have the simplicial homotopy $\theta$ also for Lemma~\ref{lem:HardLemmaMadeEasy_quasicategorical}.
\end{example}

\section*{Acknowledgements}

{\bf Scientific Acknowledgements.}
We thank David Gepner and Andrew Blumberg for helpful discussions in explaining their work \cite{BlumbergGepnerTabuadaI}, some suggestions that lead to Propositions~\ref{prop:universal_split_exact_sequence} and \ref{prop:universal_split_exact_sequence:quasicategorical}, and discussions relating to some results in Section~\ref{sec:K-Theory_Space_and_Spectrum}. We also thank Clark Barwick for patiently answering our questions about his work \cite{Barwick} and for some comments about the present paper. We thank Emily Riehl and Dominic Verity for explaining to us their results \cite[Corollary~5.2.20]{RiehlVerity} and \cite[Theorem~1.1]{RiehlVerityCompleteness}; their explanations positively impacted Section~\ref{subsec:natural_pushout_along_cofibrations}.

We also thank Wolfgang L\"uck, Niko Naumann, Ulrich Bunke, Georgios Raptis, John Lind, Justin Noel, Mike Shulman, Marcy Robertson, and Parker Lowrey for suggestions, interesting discussions, or the occasional email.

Thomas Fiore thanks the Regensburg {\it Sonderforschungsbereich 1085: Higher Invariants} and the Regensburg Mathematics Department for a very stimulating working environment.

Finally, Thomas Fiore also thanks the organizers of the 2007--2008 thematic programme on Homotopy Theory and Higher Categories at the Centre de Recerca Mathem\`{a}tica in Bellaterra: Andr\'{e} Joyal, Carles Casacuberta, Joachim Kock, Amnon Neeman, and Frank Neumann. Andr\'{e} Joyal's {\it Advanced Course on Simplicial Methods in Higher Categories} and his accompanying foundational text \cite{JoyalQuadern} have been a tremendous influence on the present paper. At the Universitat Aut\`{o}noma de Barcelona in 2007--2008, Thomas Fiore was supported on grant SB2006-0085 of the Spanish Ministerio de Educaci\'{o}n y Ciencia.

{\bf Financial Acknowledgements.} Thomas Fiore gratefully acknowledges support from several sources during the genesis of this paper. A Humboldt Research Fellowship for Experienced Researchers supported Thomas Fiore during his 2015-2016 visit at Universit{\"a}t Regensburg. The Max-Planck-Institut f\"ur Mathematik supported his research stays in Bonn in May-June 2011 and July 2013. A Small Grant for Faculty Research from the University of Michigan-Dearborn supported some travel costs to Bonn, and a Rackham Faculty Research Grant of the University of Michigan made possible his weekend trip to Regensburg in July, 2011 to discuss with David Gepner. Malte Pieper was supported by the ERC Advanced Grant ``KL2MG-interactions'' (no.  662400) granted by the European Research Council to Wolfgang L\"{u}ck.




\end{document}